\documentclass[a4paper,oneside,11pt]{article}
 
\usepackage[a4paper]{geometry}
\usepackage{aeguill}               
\usepackage{graphicx}
\usepackage{amsmath,amsfonts,amssymb,amsthm}
\usepackage{pstricks,comment} 
\usepackage{epstopdf,enumerate}
\usepackage{srcltx}
\usepackage[utf8]{inputenc} 
\usepackage[T1]{fontenc} 
\usepackage{hyperref}

\usepackage[english]{babel}

\title{Cocycle superrigidity from higher rank lattices to $\Out(F_N)$}
\author{Vincent Guirardel, Camille Horbez and Jean Lécureux}

\usepackage[all]{xy}
\usepackage{bbm}

\newcommand{\bdd}{\mathrm{bdd}}
\newcommand{\bbun}{\mathbbm{1}}

\newcommand{\Prob}{\mathrm{Prob}}
\newcommand{\baro}{\overline{\calo}}

\newcommand{\ra}{\rightarrow}
\newcommand{\xra}{\xrightarrow}
\newcommand{\m}{^{-1}}
\newcommand{\dunion}{\sqcup}
\newcommand{\eps}{\varepsilon}
\renewcommand{\epsilon}{\varepsilon}

\newcommand{\calf}{\mathcal{F}}

\newcommand{\calt}{\mathcal{T}}

\newcommand{\calp}{\mathcal{P}}

\newcommand{\calo}{\mathcal{O}}

\newcommand{\bbP}{\mathbb{P}}
\newcommand{\bbR}{\mathbb{R}}
\newcommand{\bbQ}{\mathbb{Q}}
\newcommand{\bbN}{\mathbb{N}}
\newcommand{\bbZ}{\mathbb{Z}}
\newcommand{\actson}{\curvearrowright}
\newcommand{\es}{\emptyset}
\newcommand{\grp}[1]{\langle #1 \rangle}

\makeatletter
\edef\@tempa#1#2{\def#1{\mathaccent\string"\noexpand\accentclass@#2 }}
\@tempa\rond{017}
\makeatother
\newcommand{\Comp}{\mathrm{Comp}}

\newcommand{\Stab}{\mathrm{Stab}}
\newcommand{\Mod}{\mathrm{Mod}}

\newcommand{\id}{\mathrm{id}}
\newcommand{\Out}{\mathrm{Out}}
\newcommand{\Aut}{\mathrm{Aut}}

\newcommand{\AT}{\mathcal{AT}}
\newcommand{\PAT}{\mathbb{P}\AT}

\newcommand{\bary}{\mathrm{bar}}

\newcommand{\diam}{\text{diam}}
\newcommand{\LL}{\mathrm L}

\newcommand{\FF}{\mathrm{FF}}
\newcommand{\vol}{\mathrm{covol}}

\newcommand{\bbD}{\mathbb{D}}

\newtheorem{de}{Definition} [section]
\newtheorem{theo}[de]{Theorem} 
\newtheorem{prop}[de]{Proposition}
\newtheorem{lemma}[de]{Lemma}
\newtheorem{cor}[de]{Corollary}

\newtheorem{theointro}{Theorem}
\newtheorem*{deintro}{Definition}

\theoremstyle{remark}
\newtheorem{rk}[de]{Remark}
\newtheorem{ex}[de]{Example}

\newtheorem*{rkintro}{Remark}

\newcommand{\Isom}{\text{Isom}}

\newcommand{\Maps}{\text{Maps}}
\newcommand{\Char}{\mathrm{Char}}

\newcommand{\ol}[1]{\overline{#1}}

\newcommand{\Conv}{\mathrm{Conv}}
\DeclareMathOperator{\SL}{SL}
\newcommand{\erg}{\mathrm{erg}}
\newcommand{\Lip}{\mathrm{Lip}}
\newcommand{\Pbaro}{\mathbb{P}\baro}
\newcommand{\BBT}{\mathrm{BBT}}
\newcommand{\PML}{\mathrm{PML}}
\newcommand{\Axis}{\mathrm{Axis}}
\newcommand{\FS}{\mathrm{FS}}

\newcommand{\Homeo}{\mathrm{Homeo}}

\newcommand{\R}{\mathbf R}
\usepackage[all]{xy}
\usepackage{bbm}
\begin{document}
\maketitle

\begin{abstract}
We prove a rigidity result for cocycles from higher rank lattices to $\Out(F_N)$ and more generally to the outer automorphism group of a torsion-free hyperbolic group.
More precisely, let $G$ be either a product of connected higher rank simple algebraic groups over local fields, or a lattice in such a product.
Let $G\actson X$ be an ergodic measure-preserving action on a standard probability space,  and let $H$ be a torsion-free hyperbolic group. 
We prove that every Borel cocycle $G\times X\to\Out(H)$ is cohomologous to a cocycle with values in a finite subgroup of $\Out(H)$. This provides a dynamical version of theorems of Farb--Kaimanovich--Masur and Bridson--Wade asserting that every homomorphism from $G$ to either the mapping class group of a finite-type surface or the outer automorphism group of a free group, has finite image.

The main new geometric tool is a barycenter map that associates to every triple of points in the boundary of the (relative) free factor graph
a finite set of (relative) free splittings.
\end{abstract}

\section*{Introduction}

A celebrated theorem of Farb--Kaimanovich--Masur states that every morphism from an irreducible lattice $\Gamma$ in a higher rank semisimple real Lie group to the mapping class group of a finite-type surface has finite image \cite{KM,FM}. Later, Bridson--Wade proved the analogous result with the mapping class group replaced with $\Out(F_N)$, the outer automorphism group of a finitely generated free group \cite{BW} -- their result actually applies to a more general class of groups $\Gamma$ in the source, the only assumption being that no normal subgroup of $\Gamma$ or of a finite-index subgroup of $\Gamma$ surjects onto $\mathbb{Z}$. 

These strong rigidity results have since led to further developments. Haettel recently proved that when $\Gamma$ is an irreducible lattice in a (product of) higher rank simple connected algebraic groups over local fields, any action of $\Gamma$ on a hyperbolic space is elementary, i.e.\ has either bounded orbits or fixes a finite set in the boundary \cite{Hae}. Since, on the other hand, (subgroups of) mapping class groups and $\Out(F_N)$ have many interesting actions on hyperbolic spaces, this could be used to give a new proof of the Farb--Kaimanovich--Masur and Bridson--Wade results (for this more restrictive collection of source groups $\Gamma$), and actually tackle the more general situation where the target is $\Out(G)$, where $G$ is any torsion-free hyperbolic group. Among recent developments, let us also mention Mimura's extension of the Farb--Kaimanovich--Masur and Bridson--Wade's theorems to the case where $\Gamma$ is a non-arithmetic Chevalley group \cite{Mim}.

Meanwhile, inspired by Margulis' superrigidity theorem describing the representation theory of lattices in higher-rank semisimple Lie groups \cite{Margulis}, 
Zimmer proved a cocycle rigidity theorem for ergodic actions of these lattices \cite{Zim3}. 
Indeed, given a lattice $\Gamma$ in a locally compact group $G$, a natural induction procedure turns morphisms $\Gamma \ra \Lambda$ 
into cocycles $G\times G/\Gamma \ra \Lambda$, where one can take advantage of the rich structure of $G$.
Zimmer's cocycle superrigidity theorem turned out to have important consequences, for example to measure equivalence rigidity of such lattices \cite{Fur}. This paved the way to a whole cocycle rigidity theory:
Spatzier--Zimmer \cite{SZ}, Adams \cite{Ada2}, Monod--Shalom \cite{MS,MS2} and Hamenstädt \cite{Ham} established various strong cocycle rigidity results when the target group $\Lambda$ satisfies certain negative curvature properties, and this was developed further more recently by Bader and Furman \cite{BF,BFHyp} who introduced a generalized notion of \emph{Weyl group} towards tackling rigidity problems. 

The main theorem of the present paper is the following.

\begin{theointro}\label{theo:main}
Let $G_0$ be a product of connected higher rank simple algebraic groups over local fields. Let $G$ be either $G_0$, or a lattice in $G_0$. Let $X$ be a standard probability space equipped with an ergodic measure-preserving $G$-action. Let $H$ be a torsion-free hyperbolic group. 

Then every cocycle $c:G\times X\to \Out(H)$ is cohomologous to a cocycle that takes its values in a finite subgroup of $\Out(H)$. 
\end{theointro}

Here, we recall that a Borel map $c:G\times X\to \Out(H)$ is a \emph{cocycle} if for every $g_1,g_2\in G$ and a.e.\ $x\in X$, one has $c(g_1g_2,x)=c(g_1,g_2x)c(g_2,x)$. Two cocycles $c_1,c_2:G\times X\to \Out(H)$ are \emph{cohomologous} if there exists a Borel map $f:X\to \Out(H)$ such that for all $g\in G$ and a.e.\ $x\in X$, one has $c_2(g,x)=f(g x)^{-1}c_1(g,x)f(x)$. 

We would like to make a few comments on the statement. First, notice that the target can be in particular a mapping class group or $\Out(F_N)$, so this yields yet another proof of the Farb--Kaimanovich--Masur and Bridson--Wade theorems, though with certain restrictions on the source group.
Second, we would like to mention that our main theorem actually follows from a version for cocycles with values in the outer automorphism group of a free product (see Theorem~\ref{free-product}), which is used to carry on inductive arguments -- even to prove the result for $\Out(F_N)$.  

The proof of our main theorem follows a strategy initiated by Bader and Furman in \cite{BF,BFHyp}. We set up a geometric setting for the target group $\Lambda$ which enables us to derive rigidity results for cocycles $G\times X\to \Lambda$, where $G$ is as in the statement of the theorem. We now describe this geometric setting.

\paragraph*{Geometric setting for cocycle superrigidity and inductive argument.}

Let $\Lambda$ be a countable discrete group, and let $\bbD$ be a countable discrete space equipped with a $\Lambda$-action. The following definition, which is central in the present work, is a form of negative curvature for $\Lambda$ with respect to $\bbD$. One should think that stabilizers of points in $\bbD$ will be simpler subgroups of $\Lambda$, which will enable us to carry inductive arguments. We denote by $\calp_{<\infty}(\bbD)$ the (countable) set of nonempty finite subsets of $\bbD$.

\begin{deintro}
Let $\bbD$ be a countable discrete space, let $K$ be a compact metrizable space equipped with a $\Lambda$-action by homeomorphisms, and let $\Delta$ be a Polish space equipped with a $\Lambda$-action by Borel automorphisms. We say that the triple $(\bbD,K,\Delta)$ is \emph{geometrically rigid} if
the following properties hold:
\begin{itemize}
\item\textbf{Compatibility}:  
The space $K$ has a $\Lambda$-invariant Borel partition $K=K_{\bdd}\sqcup K_\infty$ coming with $\Lambda$-equivariant Borel maps $K_{\bdd}\ra\calp_{<\infty}(\bbD)$ and $K_\infty\to\Delta$.
\item\textbf{Barycenter map}: There is a $\Lambda$-equivariant Borel \emph{barycenter map}, assigning to every triple of pairwise distinct points of $\Delta$ a nonempty finite subset of $\bbD$. 
\item\textbf{Amenability}: The $\Lambda$-action on $\Delta$ is universally amenable.
\end{itemize}
We say that $\Lambda$ is \emph{geometrically rigid with respect to $\bbD$} if there exists a geometrically rigid triple $(\bbD,K,\Delta)$.
\end{deintro}

\begin{rkintro}
  In most examples we have in mind, $\bbD$ will be the vertex set of a hyperbolic graph and $\Delta$ will be its Gromov boundary (see however Section \ref{subsec_RAAG}
for a different setting concerning right-angled Artin groups).
\end{rkintro}

This geometric setting is ubiquitous in geometric group theory. It applies for example when $\Lambda$ is a hyperbolic group (with $\bbD=\Lambda$, with $K=\Lambda\dunion\partial_\infty\Lambda$ and $\Delta=\partial_\infty\Lambda$), or more generally a relatively hyperbolic group using a construction of Bowditch \cite{Bow}. It also applies to the mapping class group of a finite-type surface acting on the vertex set $\bbD$ of the curve graph -- whose point stabilizers are related to mapping class groups of simpler subsurfaces: here $K$ is the Thurston compactification of the Teichmüller space, $K_\infty$ is the subset of arational projective measured foliations, and $\Delta$ is the Gromov boundary of the curve graph.
A crucial part of the present paper is to check that outer automorphisms of free groups, or more generally free products (or \emph{relative outer automorphism groups}), also satisfy the above geometric setting in a natural way; we will say a word about this later in this introduction.

\begin{theointro}\label{theo:intro-cocycle-geometry}
Let $\Lambda$ be a countable discrete group, and let $\bbD$ be a countable discrete $\Lambda$-set. 
Assume that $\Lambda$ is geometrically rigid with respect to $\bbD$.

Let $G_0$ be a product of connected higher rank simple algebraic groups over local fields. Let $G$ be either $G_0$, or a lattice in $G_0$. 
Let $G\actson X$ be an ergodic measure-preserving action on a standard probability space.

Then every cocycle $G\times X\to\Lambda$ is cohomologous to a cocycle that takes its values in a subgroup that virtually fixes an element of $\bbD$.  
\end{theointro}

As already mentioned above, this will be the basis for inductive arguments: in all situations we have in mind, stabilizers of points in $\bbD$ will be simpler than the ambient group $\Lambda$. 

Let us say a word about our proof of Theorem~\ref{theo:intro-cocycle-geometry}. 
By a standard argument of Zimmer, proving that a cocycle $c$ is cohomologous to a cocycle with values in a subgroup having a finite orbit in $\bbD$
amounts to constructing a $c$-equivariant Borel map $X\ra \calp_{<\infty}(\bbD)$ with values in the set $\calp_{<\infty}(\bbD)$ of nonempty finite subsets of $\bbD$.

Compactness of $K$ is crucial to construct probability measures. 
More precisely, let $B$ be a Poisson--Furstenberg boundary of $G$ and $\Prob(K)$ be the space of probability measures on $K$. 
As the $G$-action on $B$ is amenable in the sense of Zimmer, and $K$ is compact and metrizable, there exists a $c$-equivariant map $p:X\times B\to\Prob(K)$. 

Assume first that for almost every $(x,b)\in X\times B$, the probability measure $p(x,b)$ is supported on $K_{\bdd}$.
 Pushing it through the  provided map  $K_{\bdd}\ra \calp_{<\infty}(\bbD)$, one gets a $c$-equivariant map  $X\times B\ra\Prob(\calp_{<\infty}(\bbD))\ra \calp_{<\infty}(\bbD)$, where the rightmost map $\Prob(\calp_{<\infty}\bbD)\to\calp_{<\infty}(\bbD)$ consists in taking the 
union of all points with highest mass. 
Then, a strong ergodicity property of the Poisson boundary ensures that this map does not depend on the $B$-coordinate and 
gives the desired map $X\ra\calp_{<\infty}(\bbD)$.

Assume now that for almost every $(x,b)\in X\times B$, the support of $p(x,b)$ contains at least three points of $K_\infty$,
then we can use the barycenter map to get a map $X\ra \calp_{<\infty}(\bbD)$ as above.

In the remaining cases, the map $p$ yields an essentially unique equivariant map $p':X\times B\to \Delta$ (or with values in the set of pairs of points in $\Delta$).
It may happen that $p'$ does not depend on the $B$-coordinate in which case we have an equivariant map $X\ra \Delta$.
Using the fact that the $G$-action on $\Delta$ is universally amenable while $G$ has property $(T)$ and acts on $X$ preserving a probability measure, an argument of Spatzier--Zimmer \cite{SZ} then ensures that the cocycle $c$ is cohomologous to a cocycle with values in a finite subgroup of $\Lambda$.

If $p'$ depends on the $B$-coordinate, we get a non-constant equivariant Borel map $q:B\ra L^0(X,\Delta)$, where
$L^0(X,\Delta)$ denotes the space of all equivalence classes of Lebesgue-measurable maps from $X$ to $\Delta$. This map $q$ defines an equivariant quotient of $B$. Denoting by $\Delta^{(2)}$ the complement of the diagonal in $\Delta^2$, one then shows that there are exactly two Borel maps  $F,F^*:B\times B\ra L^0(X,\Delta^{(2)})$,
 namely 
 $$F:(b,b')\mapsto (x\mapsto (p'(x,b),p'(x,b')))\text{ and } F^*:(b,b')\mapsto (x\mapsto (p'(x,b'),p'(x,b))).$$

We now use an argument of Bader--Furman relying on the analysis of the \emph{Weyl group} of $G$, defined as the group $W$ of all $G$-equivariant automorphisms of $B\times B$. It turns out that
the Weyl group $W$ is a Coxeter group, whose special subgroups (i.e.\ subgroups generated by subsets of the standard generating set) 
are related to measurable equivariant quotients $q:B\ra Q$ as follows.
From $q$ and the first projection $\pi_1:B\times B\ra B$, one gets a map $q\circ \pi_1:B\times B\ra Q$, 
and the subgroup of all $\sigma\in W$ such that $q\circ \pi_1\circ \sigma=q\circ \pi_1$
happens to be a special subgroup of $W$.
In our situation $W$ acts nontrivially on the $2$-element set $\{F,F^*\}$; the kernel of this action is then a special subgroup of $W$ of index $2$, contradicting the assumption
 that all factors of $G$ have rank at least $2$.

\paragraph*{The case of outer automorphism groups.} We conclude this introduction by saying a few words about why $\Out(F_N)$ has a geometrically rigid triple
$(\bbD,K,\Delta)$. 

Here, $\bbD$ is the set of conjucacy classes of proper free factors of $F_N$,
the space $K$ is the compactification of Culler--Vogtmann's Outer space, 
and $\Delta$ is the Gromov boundary of the free factor graph. 
This boundary was described by Bestvina--Reynolds \cite{BR} and Hamenstädt \cite{Ham2} in terms of so-called \emph{arational trees}, giving the compatibility property. Amenability of the $\Out(F_N)$-action on the boundary of this graph was essentially established by Bestvina and the first two named authors in \cite{BGH}. 

To prove Theorem \ref{theo:main}, one not only needs to prove that the triple $(\bbD,K,\Delta)$ defined above is geometrically rigid for $\Out(F_N)$, but 
to carry out an inductive argument, 
one also needs to work in a \emph{relative setting} 
and prove that stabilizers of elements of $\bbD$
(i.e.\ conjugacy classes of free factors) also have a geometrically rigid triple.

The key point is to construct the barycenter map. There are two main difficulties.

The first difficulty comes from the lack of local finiteness of the free factor graph: hyperbolicity of this graph allows to associate to every triple of points in the boundary a bounded region of the graph, but this bounded region is still infinite. 
When working in the non-relative setting  (i.e.\ with $\Out(F_N)$ and not with the stabilizer of the conjugacy class of a  free factor),
we can take advantage of the structure of Outer space as a locally finite simplicial complex (with missing faces), and build the \emph{barycenter} of three pairwise inequivalent arational trees in the boundary by looking at pencils of geodesics between them, as studied by Bestvina and Reynolds in \cite{BR}. 

The second difficulty arises in the relative setting, because in this case, the corresponding Outer space itself is a locally infinite simplicial complex (with missing faces). As before, given a triple of pairwise inequivalent relative arational trees,  
the pencils of geodesics between them allow
to construct a compact subset of the relative unprojectivized Outer space.
In general, a compact set $K$ may intersect infinitely many simplices, but
a covolume argument enables us to extract a non-empty canonical finite set of simplices: more precisely, we show that the subset of $K$ made of trees of maximal covolume and minimal number of $F_N$-orbits of edges always meets only finitely many simplices.

In fact, our arguments are phrased in the following more general setting. Let $H_1,\dots,H_p$ be countable groups, let $F_r$ be a free group of rank $r$, and let $$H:=H_1\ast\dots\ast H_p\ast F_r.$$ 
Let $\calf=\{[H_1],\dots,[H_p]\}$ be the set of conjugacy classes of these subgroups in $H$.
The group $\Out(H,\calf)$ made of outer automorphisms of $H$ that preserve the conjugacy classes of each subgroup $H_i$ acts on a \emph{relative free factor graph} $\FF$, whose Gromov boundary $\partial_\infty\FF$ was described in \cite{GH} in terms of a notion of \emph{relatively arational trees}. 

\begin{theointro}
Let $H_1,\dots,H_p$ be countable groups, let $F_r$ be a free group of rank $r$, and let $$H:=H_1\ast\dots\ast H_p\ast F_r.$$ There exists a Borel $\Out(H,\calf)$-equivariant map assigning to every triple of pairwise distinct points in $\partial_\infty\FF$
a nonempty finite set of free splittings of $(H,\calf)$. 
\end{theointro}

\paragraph*{Remark.} In the appendix of \cite{Hae}, the first two named authors exploited the rigidity of actions of higher-rank lattices on hyperbolic spaces to prove a weak form of Theorem~1, dealing only with homomorphisms instead of cocycles. Our proof in the present paper has some similarities: both proofs use an induction argument to transfer the rigidity problem from the lattice to the ambient algebraic group, both proofs exploit some form of negative curvature of the target group, and both proofs require working in the context of free products and an inductive argument on the complexity of the decomposition. But the heart of the argument is different: Haettel's rigidity theorem cannot be directly applied for cocycles, and it is replaced in the present paper by arguments from the work of Bader and Furman. These new techniques require establishing some extra geometric properties like the existence of the barycenter map, which is the core technical novelty of the paper in the context of outer automorphism groups. 

\paragraph*{Organization of the paper.} This article has two parts. 
The first part aims at proving the geometric rigidity criterion (Theorem~\ref{theo:intro-cocycle-geometry}). 
After a short preliminary section (Section \ref{sec_prelim}), 
we compare various notions of amenability of group actions in Section~\ref{sec:amenability-variations}. 
In Section~\ref{sec:stability}, we establish various stability properties of cocycle rigidity, viewed as a property of the target group $\Lambda$; we also explain that cocycle rigidity when the source is a lattice is equivalent to cocycle rigidity for the ambient algebraic group.
Finally, we establish the geometric rigidity criterion in Section~\ref{sec:criterion}, where we also explain how this criterion is used when $\Lambda$ is a relatively hyperbolic group or a right-angled Artin group.

The second part of the paper is concerned with cocycle rigidity in the case of outer automorphism groups. Section~\ref{sec:background} reviews some background about outer automorphisms of free products. In Section~\ref{sec:amenability}, we derive the amenability of the action on the boundary of the free factor graph from \cite{BGH}. The barycenter map is constructed in Section~\ref{sec-barycenter}. Finally, the proofs of our main results are completed in Section~\ref{sec:conclusion}.
 
\paragraph*{Acknowledgments.} We would like to thank Uri Bader and Roman Sauer for conversations regarding cocycles and in particular the stability properties established in Section~\ref{sec:stability}. We are particularly indebted to Uri Bader for explaining us an alternative viewpoint on cocycles, described in Section~\ref{sec:developed-actions}, and how to use it to derive some stability properties.  We are also grateful to Bruno Duchesne and Cyril Houdayer for useful discussions. Finally, we thank the referee for their useful comments.

We acknowledge support from the Agence Nationale de la Recherche under Grants ANR-16-CE40-0006 DAGGER,
ANR-14-CE25-0004 GAMME and ANR-16-CE40-0022-01 AGIRA. 
The first author thanks the Centre Henri Lebesgue ANR-11-LABX-0020-01 for creating an attractive mathematical environment.

\setcounter{tocdepth}{1}
\tableofcontents

\part{Cocycle rigidity: general setting}

The goal of this first part is to establish a general geometric framework for proving cocycle rigidity statements. 

\section{Preliminaries}\label{sec_prelim}

Let $G$ be a locally compact group, equipped with its Borel $\sigma$-algebra. Let $\Delta$ be a standard Borel space.
A \emph{Borel action} of $G$ on $\Delta$ is an action by Borel automorphisms of $\Delta$ such that the map $G\times \Delta\ra \Delta$ is Borel.

When $(X,\mu)$ is a standard measured space, actions are understood in the following weaker sense. A \emph{measure-class preserving action} of $G$ on $(X,\mu)$ is a Borel map 
$$
\begin{array}{ccl}
G\times X&\to& X\\
(g,x)&\mapsto& g\cdot x
\end{array}
$$ 
such that for all $g,h\in G$ and $\mu$-almost every $x\in X$, one has $(gh)\cdot x=g\cdot (h\cdot x)$ and $1\cdot x=x$,
and for every null set $A\subseteq X$, the set $g\cdot A$ is a null set. It is a \emph{measure-preserving action} if for every Borel subset $A\subseteq X$ and every $g\in G$, one has $\mu(g\cdot A)=\mu(A)$.

Equivariance of maps between $G$-spaces has to be understood $\mu$-almost everywhere.
A Borel subset $A\subseteq X$ is \emph{$G$-invariant} if for every $g\in G$, the sets $g\cdot A$ and $A$ are equal up to null sets.
The action is \emph{ergodic} if any $G$-invariant Borel subset is null or conull. 

We say that an action is \emph{strict} if for all $g,h\in G$ and \textbf{all} $x\in X$, one has $(gh)\cdot x=g\cdot (h\cdot x)$ and $1\cdot x=x$. Forgetting the measure, a strict action on $(X,\mu)$ defines a Borel action on $X$.

By \cite[Theorem B.10]{Zim}, if $G$ has a measure-class preserving action on $(X,\mu)$, there exists a standard Borel space $(X',\mu')$
with a strict action of $G$ and an isomorphism $\phi:X_0\ra X'_0$ between full measure subsets $X_0\subseteq X$ and $X'_0\subseteq X'$
such that for all $g\in G$ and almost every $x\in X_0$, one has $\phi(g\cdot x)=g\cdot \phi(x)$.
Many results (including \cite{Zim}) are stated for strict actions, but this fact allows to use them in this broader context.

Let $\Lambda$ be a countable discrete group.
A Borel map $c:G\times X\to \Lambda$ is a \emph{cocycle} if for every $g_1,g_2\in G$ and a.e.\ $x\in X$, one has $c(g_1g_2,x)=c(g_1,g_2x)c(g_2,x)$. Two cocycles $c_1,c_2:G\times X\to \Lambda$ are \emph{cohomologous} if there exists a Borel map $f:X\to \Lambda$ such that for all $g\in G$
and a.e.\ $x\in X$, one has $c_2(g,x)=f(g x)^{-1}c_1(g,x)f(x)$. 

Given a standard Borel space $\Delta$ equipped with a $\Lambda$-action by Borel automorphisms, we say that a Borel map $f:X\to\Delta$ is \emph{$c$-equivariant} if for all $g\in G$ and a.e.\ $x\in X$, one has $f(g x)=c(g,x)f(x)$.

\section{Various notions of amenability of a group action}\label{sec:amenability-variations}

\subsection{Zimmer amenability of an action on a probability space}

Let $G$ be a locally compact group, and let $(\Omega,\mu)$ be a standard probability space equipped with a measure-class preserving action of $G$.

Let $B$ be a separable real Banach space. We denote by $\Isom(B)$ the group of all linear isometries of $B$, equipped with the strong operator topology. We denote by $B_1^\ast$ the unit ball of the dual Banach space $B^\ast$, equipped with the weak-$\ast$ topology. Given a linear isometry $T$ of $B$, we denote by $T^\ast$ the restriction to $B_1^\ast$ of the adjoint operator of $T$. The \emph{adjoint cocycle} of a Borel cocycle $\rho:G\times\Omega\to\Isom(B)$ is the cocycle $\rho^{\ast}:G\times\Omega\to\Homeo(B_1^{\ast})$ defined by letting $\rho^{\ast}(g,\omega):=\rho(g^{-1},g\omega)^\ast$. We denote by $\Conv(B_1^*)$ the set of all nonempty convex compact subsets of $B_1^\ast$. 
A \emph{Borel convex field over $\Omega$} is a map $K:\Omega\to\Conv(B_1^*)$ such that $$\{(\omega,k)\in \Omega\times B_1^\ast| k\in K(\omega)\}$$ is a Borel subset of $\Omega\times B_1^\ast$. A \emph{Borel section} of a Borel convex field $K$ over $\Omega$ is a Borel map $s:\Omega\to B_1^\ast$ such that for a.e.\ $\omega\in \Omega$, one has $s(\omega)\in K(\omega)$. 

\begin{de}[\textbf{\emph{Zimmer amenability of a group action}} \cite{Zim2}]
Let $G$ be a locally compact group, and let $\Omega$ be a standard probability space equipped with a measure-class preserving $G$-action. The $G$-action on $\Omega$ is \emph{Zimmer amenable} if for every separable real Banach space $B$ and every cocycle $\rho:G\times\Omega\to\Isom(B)$, every $\rho^{\ast}$-equivariant Borel convex field over $\Omega$ admits a $\rho^{\ast}$-equivariant Borel section.
\end{de}

Amenability of group actions will mostly be used through the following fact, which is established as in \cite[Proposition~4.3.9]{Zim} (see e.g.\ the last paragraph of \cite[p.~103]{Zim}).

\begin{lemma}[{Zimmer \cite{Zim}}]\label{lemma:invariant-proba}
Let $G$ be a locally compact group and let $(\Omega,\mu)$ be a standard probability space equipped with an ergodic measure-preserving $G$-action. 
Assume that the $G$-action on $(\Omega,\mu)$ is Zimmer amenable.

Let $\Lambda$ be a countable group, and let $K$ be a compact metrizable space equipped with a $\Lambda$-action by homeomorphisms. Equip $\Prob(K)$ with the weak-$\ast$ topology.

Then for every cocycle $c:G\times \Omega\to\Lambda$, there exists a $c$-equivariant Borel map $\Omega\to\Prob(K)$. 
\end{lemma}

\subsection{Universal and Borel amenability of an action on a Borel space}

\begin{de}[\textbf{\emph{Universal amenability of a group action}}]
Let $\Lambda$ be a countable group, and let $\Delta$ be a standard Borel space equipped with a Borel action of $\Lambda$. 
The $\Lambda$-action on $\Delta$ is \emph{universally amenable} if for every probability measure $\mu$ on $\Delta$ whose measure class is $\Lambda$-invariant, the $\Lambda$-action on $(\Delta,\mu)$ is Zimmer amenable. 
\end{de}

 We denote by $\Prob(\Lambda)$ the space of all probability measures on $\Lambda$, equipped with the topology of pointwise convergence. The left action of $\Lambda$ on itself induces an action of $\Lambda$ on $\Prob(\Lambda)$ by
$\lambda_*\nu(\lambda')=\nu(\lambda\m \lambda')$ for all $\lambda,\lambda'\in \Lambda$.

\begin{de}[\textbf{\emph{Borel amenability of a group action}}] 
Let $\Lambda$ be a countable group, and let $\Delta$ be a standard Borel space equipped with a Borel $\Lambda$-action.
The $\Lambda$-action on $\Delta$ is \emph{Borel amenable} if there exists a sequence of Borel maps 
\begin{displaymath}
\begin{array}{cccc}
\nu_n: &\Delta &\to &\Prob(\Lambda)\\
& \delta & \mapsto & \nu_n^{\delta}
\end{array}
\end{displaymath}
such that for every $\delta\in\Delta$, the norm $||\nu^{\lambda.\delta}_n-\lambda_*\nu_n^\delta||_1$ converges to $0$ as $n$ goes to $+\infty$.
\end{de}

 A sequence $(\nu_n)_{n\in\mathbb{N}}$ as in the above definition is said to be \emph{asymptotically $\Lambda$-equivariant}.

\begin{prop}\label{prop:borel-universally}
Let $\Lambda$ be a countable group, and let $\Delta$ be a standard Borel space equipped with a Borel $\Lambda$-action. If the $\Lambda$-action on $\Delta$ is Borel amenable, then it is universally amenable.
\end{prop}

This proposition is probably well-known to the experts. However, we could not find this precise statement in the literature, so we provide a proof. We mention that a similar statement is proved in \cite[Theorem~3.3.7]{ADR} for locally compact groups, under the extra assumption that the space $\Delta$ on which $\Lambda$ is acting is also a locally compact topological space. Similar facts are also established in \cite{JKL}.

\begin{proof}
We will follow Zimmer's proof of \cite[Proposition~2.2]{Zim2}. Let $\mu$ be a probability measure on $\Delta$ whose measure class is $\Lambda$-invariant. Let $B$ be a separable real Banach space; we denote by $B_1^\ast$ the closed unit ball of the dual Banach space $B^\ast$ of $B$. Let $\rho:\Lambda\times\Delta\to\Isom(B)$ be a cocycle, and let $K:\Delta\to\Conv(B_1^*)$ be a $\rho^\ast$-equivariant Borel convex field over $\Delta$. Starting from a Borel section $s:\Delta\to B_1^\ast$ of $K$ (see \cite[Lemma~1.7]{Zim2} for the existence of such a section), we will build one which is $\rho^{\ast}$-equivariant.

Since the $\Lambda$-action on $\Delta$ is Borel amenable, we can find an asymptotically $\Lambda$-equivariant sequence of Borel maps 
\begin{displaymath}
\begin{array}{cccc}
\nu_n: &\Delta &\to &\Prob(\Lambda)\\
& \delta & \mapsto & \nu_n^{\delta}
\end{array}
\end{displaymath}
 Notice that $\Lambda$ acts on the set of all Borel sections of $K$ by letting $$\lambda\cdot s'(\delta):=\rho^{\ast}(\lambda,\lambda^{-1}\delta)s'(\lambda^{-1}\delta)$$ (notice indeed that $\lambda\cdot s'(\delta)\in K(\delta)$ by $\rho^{\ast}$-equivariance of $K$). Our goal is to find a $\Lambda$-invariant section, i.e.\ a section $s'$ such that for every $\lambda\in\Lambda$, the maps $s'$ and $\lambda\cdot s'$ coincide almost everywhere. For every $n\in\mathbb{N}$ and every $\delta\in \Delta$, we then let $$s_n(\delta):=\sum_{\lambda\in\Lambda}\nu_n^\delta(\lambda)\lambda\cdot s(\delta),$$ which belongs to $K(\delta)$ by convexity. 
 
 Let $L^1(\Delta,B)$ be the space of all Lebesgue-measurable functions $f:\Delta\to B$ such that $$\int_{\Delta}||f(\delta)||d\mu(\delta)<+\infty,$$
 where as usual we identify two functions if their difference vanishes on a full measure subset. This is a Banach space, whose dual we now describe. A map $f:\Delta\to B^\ast$ is \emph{scalarwise measurable} if for every $v\in B$, the function $\langle f(\cdot),v\rangle$ is Lebesgue-measurable. We denote by $L^\infty(\Delta,B^\ast)$ the space of all scalarwise measurable functions $f:\Delta\to B^\ast$ such that $||f(\cdot)||$ is bounded. Given $f\in L^1(\Delta,B)$ and $g\in L^\infty(\Delta,B^\ast)$, we let $$\langle f,g\rangle:=\int_{\Delta}\langle f(\delta),g(\delta)\rangle d\mu(\delta).$$ This determines a map $L^\infty(\Delta,B^\ast)\to (L^1(\Delta,B))^\ast$, which is an isomorphism \cite[Proposition~8.18.2]{Edw}. The functions $s_n$ all belong to the unit ball of $L^\infty(\Delta,B^\ast)$, which is compact for the weak-$\ast$ topology. We will show that every accumulation point $s_\infty$ of the sequence $(s_n)_{n\in\mathbb{N}}$ is $\rho^\ast$-equivariant.
 Since $B$ is separable, it then follows from \cite[Proposition~8.15.3]{Edw} that $s_\infty$ is a.e.\ equal to some Borel map $\Delta\to B_1^{\ast}$. 

To prove that $s_\infty$ is $\rho^\ast$-equivariant, given $f\in L^1(\Delta,B)$ and $h\in \Lambda$, we compute
$$\langle h\cdot s_n-s_n,f\rangle  = \int_{\Delta}\langle (h\cdot s_n-s_n)(\delta),f(\delta)\rangle d\mu(\delta),$$ and we aim to prove that this converges to $0$. A quick computation shows that for every $\delta\in\Delta$, the difference $h\cdot s_n(\delta)-s_n(\delta)$ is equal to $$\sum_{\lambda\in\Lambda}(\nu_n^{h^{-1}\delta}(h^{-1}\lambda)-\nu_n^\delta(\lambda))\rho^{\ast}(\lambda,\lambda^{-1}\delta)s(\lambda^{-1}\delta),$$ whose norm is bounded by $2$. As $f$ is integrable, by applying Lebesgue's dominated convergence theorem, we deduce that $\langle h\cdot s_n-s_n,f\rangle$ converges to $0$ as $n$ goes to $+\infty$, as desired. 
\end{proof}

\section{Stabilities}\label{sec:stability}

The goal of the present section is to establish some stability properties of cocycle rigidity. As a property of the target group, we will show in Section~\ref{sec:stability-target} that it is invariant under finite-index subgroups or overgroups and under group extensions. As a property of the source, we will show in Section~\ref{sec:stability-source} that cocycle rigidity for the lattice or its ambient algebraic group are equivalent. The proofs rely on an interpretation of cocycles in terms of their \emph{developed actions} given in Section~\ref{sec:developed-actions}, which was explained to us by Uri Bader (see the MathOverflow answer \cite{MO}). We make the following definition.

\begin{de}[Cocycle-rigid pair] Let $G$ be a locally compact group, and let $\Lambda$ be a countable discrete group. We say that the pair $(G,\Lambda)$ is \emph{cocycle-rigid} if for every standard probability space $(X,\mu)$ equipped with an ergodic measure-preserving $G$-action, every cocycle $G\times X\to \Lambda$ is cohomologous to a cocycle that takes  
its values in some finite subgroup of $\Lambda$. 
\end{de}

\begin{rk}\label{rk_non-ergodic}
Using a decomposition into ergodic components \cite[Theorem~4.4]{Varadarajan}, a standard argument allows one to consider non-ergodic actions as follows.
The pair $(G,\Lambda)$ is cocycle-rigid if and only if for every action of $G$ on a standard probability space $X$ and every cocycle $c:G\times X\to \Lambda$, the space $X$ has a countable almost $G$-invariant Borel partition
$X=\dunion_{i\in\mathbb{N}} X_i$ such that for every $i\in\mathbb{N}$, the cocycle $c_{|G\times X_i}$ is cohomologous to a cocycle with values in a finite subgroup $\Lambda_i$ of $\Lambda$. 
\end{rk}

\subsection{Cocycles and developed actions}\label{sec:developed-actions}

The goal of the present section is to characterize cocycle-rigidity of the pair $(G,\Lambda)$ in terms of commuting actions of $G$ and $\Lambda$: this is the contents of Corollary~\ref{cor:caracterisation_rigid} below. The key notion towards this goal is the following.

\begin{de}[\textbf{\emph{Developed action}}]
Let $G$ be a locally compact group and let $(X,\mu)$ be a standard probability space equipped with a measure-preserving $G$-action. Let $\Lambda$ be a countable discrete group. Let $c:G\times X\to \Lambda$ be a cocycle. Then $G\times \Lambda$ acts on $X\times \Lambda$ via $$(g,\lambda)\cdot (x,\lambda'):=(gx,c(g,x)\lambda'\lambda^{-1}).$$ 
We equip $X\times \Lambda$ with the product measure $\nu$ of $\mu$ and of the counting measure on $\Lambda$. We call $G\times \Lambda\actson X\times \Lambda$ the \emph{developed action} of the cocycle $c$.
\end{de}

Notice that the developed action of $G\times \Lambda$ on $X\times \Lambda$ preserves $\nu$.
Moreover, the $\Lambda$-action on $X\times \Lambda$ is free and has a finite measure fundamental domain, namely $X\times\{1\}$. We make the following definition, which is adapted to measure-preserving actions defined as in Section~\ref{sec_prelim} (i.e.\ where the action relation is only required to be satisfied almost everywhere).

\begin{de}\label{de:fundamental-domain}
Let $(Y,\nu)$ be a standard measured space 
equipped with a measure-preserving action of a countable group $\Lambda$. 
We say that the $\Lambda$-action on $Y$ is \emph{essentially free with a finite measure fundamental domain} if
there exist a conull subset $Y_0\subseteq Y$ and a Borel subset $X\subseteq Y$ of finite measure such that $Y_0$ decomposes as the countable disjoint union
$$Y_0=\bigsqcup_{\lambda\in \Lambda} \lambda  X.$$
\end{de}

\begin{prop}\label{prop:interpretation-cocycle}
Let $G$ be a locally compact group, and let $\Lambda$ be a countable discrete group. Let $(Y,\nu)$ be a standard measured space 
equipped with a measure-preserving action of $G\times \Lambda$. 
Assume that the $\Lambda$-action on $Y$ is essentially free with a finite measure fundamental domain $X$.

Then there exist a measure-preserving $G$-action on $X$  and a cocycle $c:G\times X\to \Lambda$ whose developed action is isomorphic to the $(G\times \Lambda)$-action on $Y$. 
\end{prop}

\begin{proof}
Let $Y_0\subseteq Y$ be a conull subset provided by Definition~\ref{de:fundamental-domain}. Let $\phi:X\times \Lambda\to Y_0$ be the Borel isomorphism defined by $\phi(x,\lambda)=\lambda^{-1} x$. 
Let $\alpha:G\times X\ra X$ and $c:G\times X\ra \Lambda$ be such that for all $g\in G$ and a.e.\ $x\in X$, one has $\phi^{-1}(gx)=(\alpha_g(x),c(g,x))$. 
Thus we have $gx=c(g,x)^{-1}\alpha_g(x)$. Therefore, for all $g_1,g_2\in G$ and a.e.\ $x\in X$, one has $g_1(g_2 x)=g_1 c(g_2,x)^{-1}\alpha_{g_2}(x)$. Since the actions of $G$ and $\Lambda$ on $Y$ commute, we deduce that for all $g_1,g_2\in G$ and a.e.\ $x\in X$, one has  
$$g_1g_2 x=c(g_2,x)^{-1} g_1\alpha_{g_2}(x)=c(g_2,x)^{-1} c(g_1,\alpha_{g_2}(x))^{-1}\alpha_{g_1}(\alpha_{g_2}(x)),$$ and on the other hand $$(g_1g_2)x=c(g_1g_2,x)^{-1}\alpha_{g_1g_2}(x).$$
Thus we get that for all $g_1,g_2\in G$ and a.e.\ $x\in X$, one has $\alpha_{g_1}(\alpha_{g_2}(x))=\alpha_{g_1g_2}(x)$, which means that $g\mapsto \alpha_g$ is indeed an action of $G$ on $X$, and it is measure-preserving. Furthermore we also get that $c(g_1g_2,x)^{-1}=c(g_2,x)^{-1} c(g_1,\alpha_{g_2}(x))^{-1}$. Hence $c$ is a cocycle for the $\alpha$-action of $G$ on $X$.

Now let $\beta$ be the developed action of $c$, which is an action of $G\times \Lambda$ on $X\times \Lambda$. We claim that $\phi \beta \phi^{-1}$ coincides almost everywhere 
with the original action of $G\times \Lambda$ on $Y$. Indeed, for almost every $y\in Y$, letting $\phi(x,\lambda_1)=y$ we get 
\begin{displaymath}
\begin{array}{rl}
\phi \beta_{g,\lambda} \phi^{-1}(y)&=\phi \beta_{g,\lambda}(x,\lambda_1)\\
&=\phi(\alpha_g(x),c(g,x)\lambda_1\lambda^{-1})\\
&= \lambda\lambda_1^{-1}c(g,x)^{-1}\alpha_g(x)\\
&=\lambda\lambda_1^{-1}gx\\
&=\lambda gy
\end{array}
\end{displaymath}
 (where the last equality uses the fact that the actions of $G$ and $\Lambda$ on $Y$ commute). 
\end{proof}

\begin{rk}
We can be more precise regarding the correspondence between cocycles and their developed actions. Let $X$ be a standard probability space equipped with a measure-preserving $G$-action, and let $\Lambda$ be a countable group.
Identifying the fundamental domain $X$ with $\ol Y=Y/\Lambda$ in Proposition \ref{prop:interpretation-cocycle} yields a cocycle
 $c:G\times \ol Y\ra \Lambda$. If one changes the fundamental domain $X$ to another fundamental domain $X'$,
the resulting cocycle $G\times \ol Y\ra \Lambda$ is cohomologous to $c$.

In fact the following holds.
Define a \emph{developing space $(Y,\pi)$ over $X$} as a standard measure space $Y$ with a measure-preserving action $G\times \Lambda\actson Y$ together with a $G$-equivariant map $\pi:Y\ra X$ such that 
$\Lambda$ acts essentially freely with a finite measure fundamental domain, and $\pi$ induces an isomorphism $Y/\Lambda\ra X$.
Say that two developing spaces $(Y,\pi)$, $(Y',\pi')$ over $X$ are \emph{isomorphic} if there is an equivariant isomorphism $\phi:Y\ra Y'$ such 
that $\pi'\circ\phi=\pi$.
Then the set of all cohomology classes of cocycles $G\times X\to\Lambda$
is in 1-1 correspondence with the set of isomorphism classes of developing spaces over $X$. 
\end{rk}

\begin{lemma}\label{interpretation-cohomologous}
Let $G$ be a locally compact group and let $(X,\mu)$ be a probability space equipped with an ergodic measure-preserving $G$-action. Let $\Lambda$ be a countable discrete group. Let $c:G\times X\to \Lambda$ be a cocycle. Equip $X\times \Lambda$ with the developed action of $G\times\Lambda$. 
Let $\mathbb{D}$ be a countable $\Lambda$-set. Then the following assertions are equivalent.
  \begin{enumerate}\renewcommand{\theenumi}{(\roman{enumi})}
  \item There exists $d_0\in \mathbb{D}$ such that $c$ is cohomologous to a cocycle which takes its values in $\Stab_\Lambda(d_0)$ (the $\Lambda$-stabilizer of $d_0$).
  \item There exists a $c$-equivariant Borel map $f:X\ra \mathbb{D}$.
  \item There exists a Borel map $F:X\times \Lambda\to \bbD$ 
which is $G$-invariant and $\Lambda$-equivariant.
\end{enumerate}

 If $G\actson X$ is not ergodic, the statement
  remains true if one replaces $(i)$ by
  \begin{itemize}\renewcommand{\theenumi}{(\roman{enumi})}
  \item[$(i').$] There exists a finite or countable almost $G$-invariant Borel partition $X=\dunion X_i$, and a finite or countable subset $\{d_i\}\subseteq\mathbb{D}$, such that $c$ is cohomologous to a cocycle $c'$ such that for every $i$, the restriction $c'_{|G\times X_i}$ takes its values in $\Stab_\Lambda(d_i)$. 
  \end{itemize}

\end{lemma}

\begin{proof}
If the action is ergodic, then $(i')$ is equivalent to $(i)$ so we prove the lemma without assuming ergodicity.

Assume $(i')$. Let $h:X\ra \Lambda$ be a Borel map such that for every $g\in G$, every $i$ and a.e.\ $x\in X_i$, one has $h(gx)\m c(g,x)h(x)\in \Stab_\Lambda(d_i)$.
We claim that the Borel map $f:X\ra \mathbb{D}$ defined by $f(x)=h(x)\cdot d_i$ for $x\in X_i$ is $c$-equivariant.
Indeed, for all $g\in G$ and a.e.\ $x\in X_i$, one has (using the fact that $h(gx)\m c(g,x)h(x)\in \Stab_\Lambda(d_i)$ for the second equality):
$$f(gx)=h(gx)d_i= c(g,x)h(x) d_i= c(g,x)f(x),$$
so $(ii)$ holds.

Now assume that $(ii)$ holds and define $F:X\times \Lambda\ra \mathbb{D}$ by $F(x,\lambda)=\lambda\m f(x)$.
Then for every $(g,\lambda)\in G\times \Lambda$ and a.e.\ $(x,\lambda')\in X\times \Lambda$, one has
$$F((g,\lambda).(x,\lambda'))=F(gx,c(g,x)\lambda'\lambda\m)=\lambda\lambda'^{-1} c(g,x)\m f(gx)=\lambda\lambda'^{-1}f(x)=\lambda F(x,\lambda')$$
and $(iii)$ follows.

We now prove that $(iii)$ implies $(i')$. 
Let $\{d_i\}$ be a family of representatives of the $\Lambda$-orbits of $\bbD$.
Let $X_i$ be the set of all $x\in X$ such that for every $\lambda\in\Lambda$, one has $F(x,\lambda)\in \Lambda d_i$.
Note that $X_i$ is $G$-invariant.
Using sections $\Lambda d_i \ra \Lambda$ 
one can define a Borel map $h:X\ra \Lambda$ such that for almost every $x\in X_i$, one has
$F(x,1_\Lambda)=h(x)d_i$.
Let $c'(g,x)=h(gx)\m c(g,x)h(x)$.
Then for all $g\in G$ and a.e.\ $x\in X_i$, one has
\begin{align*}
c'(g,x)d_i&=h(gx)\m c(g,x) h(x)d_i\\
&=h(gx)\m c(g,x) F(x,1_\Lambda) \\
&=h(gx)\m  F(x,c(g,x)\m).
\end{align*}
Since $F$ is $G$-invariant, and since $g(x,c(g,x)\m)=(gx,1_\Lambda)$, one gets
\begin{align*}
c'(g,x)d_i&=h(gx)\m  F(gx,1_\Lambda)=d_i.
\end{align*}
This concludes the proof.
\end{proof}

As a consequence of Lemma~\ref{interpretation-cohomologous}, we get the following characterization of cocycle rigidity.

\begin{cor}\label{cor:caracterisation_rigid}
  Let $G$ be a locally compact group and $\Lambda$ be a countable group.
Then the following assertions are equivalent.
\begin{itemize}
\item[(i)] The pair $(G,\Lambda)$ is cocycle-rigid. 
\item[(ii)] For every measure-preserving action $G\times\Lambda\actson Y$ on a standard measured space $Y$
such that the action of $\Lambda$ is essentially free with a fundamental domain of finite measure,
there exists a countable $\Lambda$-set $\bbD$ with finite point stabilizers and
a $G$-invariant $\Lambda$-equivariant Borel map $F:Y\ra \bbD$.
\end{itemize}
\end{cor}

\begin{proof}
  Assume that $(G,\Lambda)$ is cocycle-rigid. Let $G\times \Lambda\actson Y$ be an action 
such that the action of $\Lambda$ is essentially free with a fundamental domain $X$ of finite measure.
By Proposition~\ref{prop:interpretation-cocycle}, the action $G\times\Lambda\actson Y$ is the developed action of a cocycle $c:G\times X\ra \Lambda$.
By cocycle rigidity (as in Remark \ref{rk_non-ergodic}), the space $X$ has a countable almost $G$-invariant Borel partition $X=\sqcup_i X_i$
such that  $c_{|G\times X_i}$ is cohomologous to a cocycle with values in a finite group $\Lambda_i\subseteq \Lambda$.
Applying Lemma \ref{interpretation-cohomologous} with $\bbD=\dunion_i \Lambda/\Lambda_i$,
one gets a Borel map $F:Y\ra \bbD$ which is $G$-invariant and $\Lambda$-equivariant.

Conversely, consider a cocycle $c:G\times X\ra \Lambda$, and let $Y=X\times \Lambda$ be its developed action.
Our assumption gives us a countable set $\bbD$ to which one can apply Lemma~\ref{interpretation-cohomologous} 
and deduce that $X$ has a countable almost $G$-invariant Borel partition 
$X=\sqcup_i X_i$ such that $c_{|G\times X_i}$ is cohomologous to a cocycle with values in a finite group $\Lambda_i\subseteq \Lambda$.
This concludes the proof.
\end{proof}

\subsection{Stabilities in the target: finite index overgroups and extensions}\label{sec:stability-target}

We now prove that, when viewed as a property of the target group $\Lambda$, cocycle rigidity of $(G,\Lambda)$ is stable under finite-index subgroups, finite-index overgroups and group extensions.

\begin{prop}\label{finite-index}
Let $G$ be a locally compact group, let $\Lambda$ be a countable discrete group, and let $\Lambda'$ be a finite index subgroup of $\Lambda$.

Then $(G,\Lambda)$ is cocycle-rigid if and only if $(G,\Lambda')$ is cocycle-rigid.
\end{prop} 

\begin{proof} 
We use the criterion given in Corollary~\ref{cor:caracterisation_rigid}. 
Assume that the pair $(G,\Lambda)$ is cocycle-rigid, and let $G\times\Lambda'\actson Y$ be a measure-preserving action which is essentially free with a finite measure fundamental domain. 
We construct a larger space $\hat Y$ endowed with a $(G\times \Lambda)$-action so that
$Y$ embeds $(G\times\Lambda')$-equivariantly in $\hat Y$ as follows. 
We define $\hat{Y}:=(Y\times\Lambda)/{\sim_{\Lambda'}}$, where $\sim_{\Lambda'}$ is the equivalence relation where $(y,\lambda_1)\sim_{\Lambda'} (\lambda'y,\lambda'\lambda_1)$ for every $\lambda'\in\Lambda'$.
We will denote by $[y,\lambda_1]$ the equivalence class of $(y,\lambda_1)$. 
Then $G\times\Lambda$ acts on $\hat Y$ via $(g,\lambda)\cdot [y,\lambda_1]:=[gy,\lambda_1\lambda^{-1}]$
and the embedding $\iota:Y\ra \hat Y$ defined by $\iota(y)=[y,1_\Lambda]$ is $(G\times \Lambda')$-equivariant.
The space $\hat{Y}$ is isomorphic to $Y\times (\Lambda/\Lambda')$ so the measure on $Y$ defines 
a measure on $\hat Y$ which is $(G\times \Lambda)$-invariant, and the action of $\Lambda$ is essentially free with a finite measure fundamental domain.

Since $(G,\Lambda)$ is cocycle-rigid, there exists a countable $\Lambda$-set $\bbD$ with finite point stabilizers 
and a $G$-invariant $\Lambda$-equivariant Borel map $\hat{F}:\hat{Y}\to\bbD$. 
The map $\hat F\circ\iota:Y\ra \bbD$ is $G$-invariant and $\Lambda'$-equivariant.
This shows that $(G,\Lambda')$ is cocycle-rigid. 

Conversely, we now prove that cocycle-rigidity of $(G,\Lambda')$ implies cocycle-rigidity of $(G,\Lambda)$, using again the characterization given in Corollary~\ref{cor:caracterisation_rigid}.
Consider an action $G\times \Lambda\actson Y$ such that the action of $\Lambda$ is essentially free with a fundamental domain of finite measure. 
Since $(G,\Lambda')$ is cocycle-rigid, and since $\Lambda'$ acts on $Y$ essentially freely with a fundamental domain of finite measure,
there exists a countable $\Lambda'$-set $\bbD$ and a Borel map $F:Y\ra \bbD$ which is $G$-invariant and $\Lambda'$-equivariant.

We define $\hat \bbD=\Maps_{\Lambda'}(\Lambda,\Lambda\times \bbD)$ as the countable set of $\Lambda'$-equivariant maps from $\Lambda$ to $\Lambda\times \bbD$
for the standard left actions, i.e.\ the set of maps $\phi=(\phi_1,\phi_2)$ such that 
$$\phi_1(\lambda'\lambda)=\lambda'\phi_1(\lambda)\text{ and }\phi_2(\lambda'\lambda)=\lambda'\phi_2(\lambda).$$
The group $\Lambda$ acts on $\hat \bbD$ by 
$$\lambda_0\cdot (\phi_1,\phi_2)=\left [\lambda\mapsto (\phi_1(\lambda \lambda_0)\lambda_0\m,\phi_2(\lambda\lambda_0))\right].$$
To check that the action of $\Lambda$ on $\hat \bbD$ has finite stabilizers, it suffices to check that the $\Lambda'$-stabilizer of any $(\phi_1,\phi_2)$ is finite.
If $\lambda'_0\in \Lambda'$ fixes $(\phi_1,\phi_2)$ then
$$(\phi_1(1_\Lambda),\phi_2(1_\Lambda))=[\lambda'_0\cdot (\phi_1,\phi_2)](1_\Lambda)=(\lambda'_0\phi_1(1_\Lambda)(\lambda'_0)\m,\lambda'_0\phi_2(1_\Lambda))$$
so $\lambda'_0$ fixes the point $\phi_2(1_\Lambda)\in\bbD$, and we conclude since $\bbD$ has finite stabilizers.

We finally define $\hat F:Y\ra\hat\bbD$ by 
$$\hat F(y)=[\lambda\mapsto (\lambda,F(\lambda y))].$$
This map is $\Lambda$-equivariant and $G$-invariant because the actions of $G$ and $\Lambda$ on $Y$ commute.
This concludes the proof.
\end{proof} 

\begin{prop}\label{extension}
Let $G$ be a locally compact group, and let $\Lambda,N,Q$ be countable discrete groups sitting in a short exact sequence $$1\to N\to \Lambda\to Q\to 1.$$ If $(G,N)$ and $(G,Q)$ are cocycle-rigid, then so is $(G,\Lambda)$.
\end{prop}

\begin{proof}
Let $(X,\mu)$ be a standard probability space equipped with an ergodic measure-preserving action of $G$, and let $c:G\times X\to \Lambda$ be a cocycle. Using the projection map $\Lambda\to Q$, we get a cocycle $G\times X\to Q$, which is cohomologous by assumption to a cocycle that takes its values in a finite subgroup of $Q$. In other words $c$ is cohomologous to a cocycle that takes its values in a finite extension $N^1$ of $N$ (inside $\Lambda$). Since $(G,N)$ is cocycle-rigid, Proposition~\ref{finite-index} implies that $(G,N^1)$ is cocycle-rigid. Therefore $c$ is cohomologous to a cocycle that essentially takes its values in a finite subgroup of $\Lambda$.
\end{proof}

\subsection{Stabilities in the source: groups and their lattices}\label{sec:stability-source}

The goal of the present section is to prove the following proposition. This is probably well-known to the experts, but we could not find the statement given as such in the literature, so we include a proof.

\begin{prop}\label{induction}
Let $G$ be a locally compact second countable group, let $\Gamma$ be a lattice in $G$, and let $\Lambda$ be a countable discrete group.

Then $(G,\Lambda)$ is cocycle-rigid if and only if $(\Gamma,\Lambda)$ is cocycle-rigid.
\end{prop}

The proof of Proposition~\ref{induction} will be carried through Lemmas~\ref{lemma:induction-1} and~\ref{lemma:induction-2} below.

\subsubsection{From the lattice to the ambient group}

\begin{lemma}\label{lemma:induction-1}
Let $G$ be a locally compact second countable group, let $\Gamma$ be a lattice in $G$, and let $\Lambda$ be a countable discrete group.

If $(\Gamma,\Lambda)$ is cocycle-rigid, then $(G,\Lambda)$ is cocycle-rigid.
\end{lemma}

\begin{proof}
  We use the characterization of cocycle rigidity given in Corollary \ref{cor:caracterisation_rigid}.

  Consider an action of  $G\times \Lambda$  on a standard measure space $Y$ such that $\Lambda$ acts essentially freely
  with a fundamental domain of finite measure.
  Restricting the action to $\Gamma\times\Lambda$, cocycle rigidity of $(\Gamma,\Lambda)$ provides a countable $\Lambda$-set $\bbD$
  with finite point stabilizers and a
  $\Gamma$-invariant $\Lambda$-equivariant map $F:Y\ra \bbD$.

  We denote by $\Maps(\Gamma\backslash G,\mathbb{D})$ the collection of all Borel maps from $\Gamma\backslash G$ to $\mathbb{D}$: this is endowed with a $(G\times\Lambda)$-action given by
  $$(g,\lambda)\cdot \theta=(\Gamma h\mapsto \lambda\theta(\Gamma hg)).$$
  The map $\tilde{F}:Y\to\Maps(\Gamma\backslash G,\mathbb{D})$ defined by letting $$\tilde{F}(y):=(\Gamma h\mapsto F(h y))$$ is
  well-defined (because $F$ is $\Gamma$-invariant) and $(G\times\Lambda)$-equivariant. 

  Let $\mu$ be a (finite) Haar measure on $\Gamma\backslash G$. We denote by $\calp_{<\infty}(\mathbb{D})$ the countable set of all nonempty finite subsets of $\mathbb{D}$.
  Consider the map $$\Psi:\Maps(\Gamma\backslash G,\mathbb{D})\to\calp_{<\infty}(\mathbb{D}),$$
  sending a map $\theta$ to the collection of all elements of $\mathbb{D}$ with maximal $\theta_\ast\mu$-measure (which is finite because $\mathbb{D}$ is countable and $\mu$ is a finite measure). The map $\Psi$ is $G$-invariant and $\Lambda$-equivariant.
The composition $\Psi\circ \tilde F$ is a $G$-invariant $\Lambda$-equivariant Borel map $Y\to\calp_{<\infty}(\mathbb{D})$. The space $\calp_{<\infty}(\mathbb{D})$ is countable, and the $\Lambda$-action on $\calp_{<\infty}(\mathbb{D})$ has finite point stabilizers (because the $\Lambda$-action on $\mathbb{D}$ has finite point stabilizers). This concludes the proof.
\end{proof}

\subsubsection{From the ambient group to the lattice}

We now prove the converse to Lemma~\ref{lemma:induction-1}. 

\begin{lemma}\label{lemma:induction-2}
Let $G$ be a locally compact second countable group, let $\Gamma$ be a lattice in $G$, and let $\Lambda$ be a countable discrete group.

If $(G,\Lambda)$ is cocycle-rigid, then $(\Gamma,\Lambda)$ is cocycle-rigid.
\end{lemma}

\begin{proof}
  We use the characterization of cocycle rigidity given in Corollary \ref{cor:caracterisation_rigid}.

Consider a measure preserving action of $\Gamma\times \Lambda$ on a standard measured space $Y$ such that
the action of $\Lambda$ is essentially free with a fundamental domain of finite measure.
Let $\hat Y=(G\times Y)/{\sim_\Gamma}$ where $\sim_\Gamma$ is the equivalence relation defined by 
$$(g,\gamma y)\sim_\Gamma (g\gamma,y)$$
for $g\in G$, $y\in Y$ and $\gamma\in \Gamma$.
We denote by $[g,y]$ the class of $(g,y)$ in $\hat Y$.
The space $\hat Y$ can also be viewed as $(G/\Gamma)\times Y$,
and thus be endowed with the product measure of the Haar measure with the measure of $Y$.
The group $G\times \Lambda$ acts on $\hat Y$ by 
$$(g_0,\lambda_0)\cdot [h,y]=[g_0h,\lambda_0 y]$$
and preserves the measure of $\hat Y$.
Denote by $X$ a fundamental domain for the action of $\Lambda$ on $Y$.
Then the image $(G\times X)/{\sim_\Gamma}$  is isomorphic to $G/\Gamma\times X$
and is a fundamental domain of finite measure for the action of $\Lambda$ on $\hat Y$.

By cocycle rigidity for $(G,\Lambda)$, there exists a countable $\Lambda$-set $\bbD$ with finite point stabilizers and a
$G$-invariant $\Lambda$-equivariant map $F:\hat Y\ra \bbD$.
Recall that this means that for all $(g_0,\lambda_0)\in G\times \Lambda$ and a.e.\ $y\in \hat Y$, one has
$F((g_0,\lambda_0)\cdot y)=\lambda_0 F(y)$.
It would be natural to consider the restriction of $F$ to the subset $\{1\}\times Y\subseteq \hat Y$
but since this is a null set, we do not know that this restriction is $\Gamma$-invariant and $\Lambda$-equivariant.

By Fubini's Theorem, for almost every $y\in Y$, the map 
$$\begin{array}{cccc}
G & \to & \bbD\\
g&\mapsto & F([g,y])
\end{array}$$ is essentially constant.
Denote by $\tilde F(y)\in \bbD$ the corresponding constant. Extending this definition arbitrarily on a null set, we get a Borel  map 
$\tilde F:Y\ra \bbD$. 
This map is clearly $\Gamma$-invariant and $\Lambda$-equivariant, thus concluding our proof.
\end{proof}

\section{Cocycle rigidity: a general criterion}\label{sec:criterion}

In this section, we use an argument due to Bader and Furman (see e.g.\ \cite{BF} or the preprint \cite{BFHyp}) to establish a general criterion ensuring rigidity phenomena (Theorem~\ref{theo:abstract} below, which is Theorem~\ref{theo:intro-cocycle-geometry} from the introduction).

\subsection{A digression on measurable functions}\label{sec-stuff}

We start by recalling some facts about convergence in measure that we will need throughout the section. Let $(X,\mu)$ be a standard probability space. Let $Y$ be a Polish space, and let $d_Y$ be a metric on $Y$ which is compatible with the topology. Let $\LL^0(X,Y)$ be the space of equivalence classes of Borel-measurable functions from $X$ to $Y$, where two functions are equivalent if they coincide on a full measure subset. We equip $\LL^0(X,Y)$ with the topology of convergence in measure: this is the topology for which a sequence $(f_n)_{n\in\mathbb{N}}$ converges to $f$ if and only if for every $\epsilon>0$, the $\mu$-measure of $\{x\in X|d_Y(f_n(x),f(x))> \epsilon\}$ converges to $0$ as $n$ goes to $+\infty$. This topology can be metrized through the metric $\delta$ defined by letting $$\delta(\phi,\phi')=\int_X \min\{d_Y(\phi(x),\phi'(x)),1\} d \mu(x).$$ 
The set $\LL^0(X,Y)$, with this topology, is again a Polish space. 
The following basic lemma is taken from \cite[VII.1.3]{Margulis}; it is essentially an application of Fubini's theorem.
\begin{lemma}\label{lem:L(X,Y)}
Let $X$ and $\Omega$ be  standard probability spaces, and let $Y$ be a Polish space. Let $q:\Omega\times X\to Y$ be a Borel map. Then the formula $\overline q(\omega)(x)=q(\omega,x)$ defines a map $\overline q\in\LL^0(\Omega,\LL^0(X,Y))$.
\end{lemma}

Assume now that a locally compact group $G$ acts on $X$ and $\Omega$ in a measure-preserving way, that a countable group $\Lambda$ acts on $Y$, and that we have a cocycle $c:G\times X\to \Lambda$. 
Let $c_1:G\times X\times\Omega\to \Lambda$ be the cocycle defined by letting $c_1(g,x,\omega):=c(g,x)$. Then $G$ acts on $\LL^0(X,Y)$ via $g\cdot f(x)=c(g,g\m x)f(g^{-1}x)$. Under the identification provided by Lemma~\ref{lem:L(X,Y)}, $c_1$-equivariant Borel maps $X\times \Omega\to Y$ correspond to $G$-equivariant Borel maps $\Omega\to \LL^0(X,Y)$. Note also that if $\Lambda$ acts isometrically on $Y$ and $G$ preserves the measure on $X$ then $G$ acts isometrically on $\LL^0(X,Y)$.

\subsection{Strong boundaries and Weyl groups}

A useful tool is the notion of a \emph{strong boundary}. A slightly weaker notion first appeared in the work of
Burger--Monod \cite{BM} (see also \cite[Definition 2.3]{MS}). A stronger definition was then coined by Bader and Furman in \cite[Definition~2.3]{BF}.  We use the following version. 

\begin{de}[\textbf{Isometrically ergodic action}]
Let $G$ be a locally compact group, and let $(B,\mu)$ be a standard probability space equipped with a measure class preserving $G$-action. The action of $G$ on $(B,\mu)$ is \emph{isometrically ergodic} if for every separable metric space $Y$ equipped with a continuous isometric $G$-action, every $G$-equivariant Borel map $B\to Y$ is essentially constant.
\end{de}

\begin{rk}\label{ergodicities}
If the action of $G$ on $(B,\mu)$ is isometrically ergodic, then it is ergodic. Indeed, it suffices to consider $Y=\{0,1\}$ 
with the trivial action of $G$. 
More generally, if $(X,\nu)$ is a standard probability space, and $G$ acts ergodically on $X$ in a measure-preserving way, then the diagonal action of $G$ on $B\times X$ is again ergodic. Indeed, $G$ acts isometrically on  the  separable metric space $Y=\LL^0(X,\mathbb{R})$,
and if $A\subseteq B\times X$ is $G$-invariant, then the function 
$f:B \to \LL^0(X,\mathbb{R})$ defined by 
$f(\omega)(x)=\bbun_{A}(\omega,x)$ 
is a  $G$-equivariant function. Isometric ergodicity implies that $f$ must be essentially constant, which means that $A$ is null or conull. Note that the fact that $G$ preserves the measure $\nu$ on $X$ is crucial to ensure that $G$ acts by isometries on $\LL^0(X,\mathbb{R})$.
\end{rk}

\begin{de}[\textbf{Strong boundary}]
Let $G$ be a locally compact group, and let $(B,\mu)$ be a standard probability space equipped with a measure class preserving $G$-action. The space $(B,\mu)$ is a \emph{strong boundary} for $G$ if 
\begin{itemize}
\item[(1)] the $G$-action on $(B,\mu)$ is Zimmer amenable, and
\item[(2)] the $G$-action on $(B\times B,\mu\times\mu)$ is isometrically ergodic.
\end{itemize}
\end{de}

\begin{rk}
Notice that if the $G$-action on $B\times B$ is isometrically ergodic, then so is the $G$-action on $B$.
\end{rk}

We mention that every locally compact second countable group $G$ admits a strong boundary, which can be obtained as the Poisson boundary of $(G,\mu)$ for a suitable measure $\mu$ on $G$ (see \cite{Kaimanovich} or \cite[Theorem 2.7]{BF}). The following definition first
appeared in a preprint of Bader--Furman--Shaker \cite{BFS}, see also \cite[\S4]{BF}. 

\begin{de}[\textbf{Weyl group}] \label{dfn_Weyl}
Let $G$ be a locally compact group, and let $(B,\mu)$ be a strong boundary for $G$. The \emph{Weyl group} $W_{G,B}$ is the group of all measure class preserving Borel automorphisms of $(B\times B,\mu\times\mu)$ which commute with the $G$-action. 

A subgroup $W'\subseteq W_{G,B}$ is \emph{special} if there exists a standard Borel space $C$ equipped with a Borel $G$-action, and a $G$-equivariant Borel map $\pi:B\to C$ such that $W'=\{w\in W_{G,B}|\pi\circ p_1\circ w=\pi\circ p_1\}$, where $p_1:B\times B\to B$ denotes the first projection (we write $W'=W_{G,B}(\pi)$ in this situation).
\end{de} 

In other words $W_{G,B}(\pi)$ is the subgroup of $W_{G,B}$ made of all automorphisms $w$ of $B\times B$ that commute with the $G$-action, and make the following diagram commute:
$$\xymatrix@C=1cm@R=0.02cm{
w&&\\
\actson&&\\
B\times B\ar[r]^{p_1}&B\ar[r]^{\pi}&C}
$$

The main example of boundaries and Weyl groups comes from the realm of simple Lie groups. The following proposition was established by Bader--Furman \cite[Theorem~2.5]{BF} in the context of real Lie groups, and by Bader--Furman--Shaker \cite[Lemma~2.8 and Proposition~2.13]{BFS} in the more general context of algebraic groups over local fields but with a slightly different notion of strong boundary. We therefore include a proof for completeness of our argument. The reader less familiar with the formalism of algebraic groups is invited to read the example of $\SL_3(\R)$ (Example \ref{ex:SL3}) before Proposition \ref{prop:WeylAlg}.

Before stating the proposition, let us introduce a few notations. Let $I$ be a finite set. For each $i\in I$, choose a local field $k_i$, a connected simple algebraic group $\mathbf G_i$ defined over $k_i$ of $k_i$-rank at least 2.  Let $\mathbf S_i$ be a maximal $k_i$-split torus, let $\mathbf Z_i$ the centralizer of $\mathbf S_i$, and $\mathbf P_i$ be a minimal parabolic subgroup containing $\mathbf Z_i$. Let $G_i=\mathbf{G}_i(k_i)$, $P_i=\mathbf P_i(k_i)$, $Z_i=\mathbf Z_i(k_i)$, $S_i=\mathbf S_i(k_i)$ and $B_i=G_i/P_i$. Let $W_i=N_{G_i}(S_i)/Z_i$.

Finally, let $G=\prod_{i\in I} G_i$, $P=\prod_{i\in I}P_i$, $W=\prod_{i\in I} W_i$, $S=\prod_{i\in I}S_i$ and $Z=\prod_{i\in I}Z_i$. The group $W=N_G(S)/Z$ is the so-called \emph{restricted Weyl group} of $G$. It is a Coxeter group, generated by some set $R=\dunion_{i\in I} R_i$ of simple reflections.

\begin{prop}\label{prop:WeylAlg}
The space $B:=G/P$ (with the Haar measure class) is a strong boundary for $G$. The Weyl group of $(G,G/P)$ is isomorphic to $W$. Furthermore, $W_{G,G/P}$ does not have any nontrivial special normal subgroup of index 2.
\end{prop}

\begin{proof}
For each $i$, the group $P_i$ is an extension of its unipotent radical (which is amenable) by a $k_i$-anisotropic reductive group, which is therefore compact (see \cite{Rousseau} or \cite{Prasad}). Hence $P_i$ is amenable, and $P=\prod P_i$ also is. Therefore $B$ is an amenable $G$-space.
 
 The Bruhat decomposition \cite[5.15]{BorelTits} says that every element of $g$ can be written as $b n_w b'$, for some $b,b'\in P$, and $n_w$ is a lift of some $w\in W$ in $G$. Furthermore, the element $w\in W$ is unique. This means that $P\backslash G/P$ is in bijection with $W$. The $G$-orbits of the space $G/P\times G/P$ are in bijection (via the map $G.(hP,gP)\mapsto Ph^{-1}gP$) with $P\backslash G/P$, hence with $W$.
 
Among these orbits, there is one which is open and dense (in the Zariski and the Hausdorff topology): 
it consists in the set of opposite parabolic points (the longest element of $W$) \cite[21.26, 21.28]{Borel}. 
The stabilizer of a point in this orbit is conjugate to $Z$ \cite[4.8]{BorelTits}. All the other orbits are of lower dimension, and therefore of Haar measure 0.  It follows that $B\times B$ is $G$-equivariantly measurably isomorphic to $G/Z$, and we can (and shall) choose this isomorphism so that the composition with the first projection $G/Z\simeq G/P\times G/P\to G/P$ is the map $G/Z\to G/P$ induced by the inclusion $Z\subseteq P$. 
The isometric ergodicity of $B\times B$ follows from \cite[Theorem~6.6]{BaderGelander} (see  \cite[Lemma~2.6]{BF} for an alternative argument). 
 
 The Weyl group of $B$ is $\Aut_G(B\times B)=\Aut_G(G/Z)$. An automorphism of $G/Z$ is uniquely determined by the image of $eZ$, which is $g_0Z$ for some $g_0\in N_G(Z)$, where $N_G(Z)$ is the normalizer of $Z$. Hence we have $\Aut_G(G/Z)\simeq N_G(Z)/Z$ (see \cite[Lemma~2.14]{BFS} for more details). So the Weyl group of $B$ is exactly $W$. 
 
 We can also describe the special subgroups of $W$. A quotient of $G/P$ is of the form $G/Q$, with $P<Q$. By \cite[5.18,5.20]{BorelTits} the groups containing $P$ are in bijection with the subsets $T\subseteq R$, the bijection being given by $T\mapsto P_T=P W_T P$, where $W_T=\langle T\rangle$, $T\subseteq R$. Then the special subgroup associated to the quotient $\pi_T:G/P\to G/P_T$ is exactly $W_T$. Indeed, if $w\in W_T$, lift $w$ to $n_w\in N_G(S)$. Then $n_w\in P_T$ and therefore for every $g\in G$, $gP$ and $gn_wP$ have the same image in $G/P_T$. Hence $W_T\subseteq W_{G,B}(\pi_T)$. Conversely, if $w\not \in W_T$, lift  again $w$ to an element $n_w$ of $N_G(S)$. By disjointness in the Bruhat decomposition we have $P_T\cap n_wP_T=\varnothing$. Hence $n_w P_T$ and $P_T$ are distinct cosets, which means that $w\not \in W_{G,B}(\pi_T)$. Therefore $W_T=W_{G,B}(\pi_T)$ as claimed.
  
  Since all $G_i$ are simple, each $W_i$ is irreducible, hence no non-trivial proper special subgroup $W_T\subseteq W_i$ is normal.
Hence the special subgroups of $W$ are all products of the form $\prod_{i\in J} W_i$, for some $J\subseteq I$. Since each $G_i$ has rank at least 2, we get that none of these subgroups have index 2.  
\end{proof}

\begin{ex}\label{ex:SL3}
Let $G=\SL_3(\R)$. Then the previous definitions unfold as follows: $P$ is the subgroup of upper-triangular matrices, $Z=S$ is the subgroup of diagonal matrices, and $W$ is the group of permutations $S_3$. The space $B=G/P$ is the set of flags (line $\subset$ plane) in $\R^3$. The set $B\times B$ of pairs of flags has a $G$-orbit of full measure: the set of  pairs of flags which are transverse (in the sense that the line of one of them is in direct sum with the plane of the other). A pair of transverse flags $(d\subset P)$ and $(d'\subset P')$ determines a frame (i.e. a triple of lines $(d_1,d_2,d_3)$ such that $d_1\oplus d_2\oplus d_3=\R^3$): take for example $d_1=d$, $d_2=P\cap P'$, $d_3=d'$. The group $W$ acts by permuting the lines of the frame. 
If an element of $B\times B$ is represented by the frame $(d_1,d_2,d_3)$ then the first projection $p_1:B\times B\ra B$ is given by the flag $d_1\subset d_1\oplus d_2$, and the second projection by $d_3\subset d_3\oplus d_2$.

There are exactly two non-trivial proper $G$-quotients of $G/P$, namely $G/Q_1$ and $G/Q_2$, where $Q_1$ is the stabilizer of a line and $Q_2$ is the stabilizer of a plane; they correspond to maps $(d\subset P)\mapsto d$ and $(d\subset P)\mapsto P$. The special subgroup associated with the quotient $G/P\to G/Q_1$ is the group $\{\id,(2,3)\}$ whereas the special subgroup associated to the quotient $G/P\to G/Q_2$ is the group $\{\id,(1,2)\}$. None of them is normal in $W$.
\end{ex}

\subsection{Geometrically rigid actions}

Given a standard Borel space $\Delta$ and $i\in\mathbb{N}$, we denote by $\Delta^{(i)}$ the subset of $\Delta^i$ made of all tuples having pairwise distinct coordinates, which is again a standard Borel space. Given a set $\bbD$, we denote by $\calp_{<\infty}(\bbD)$ the set of all nonempty finite subsets of $\bbD$. The following property is the central definition of the present work; it will be used to derive rigidity statements regarding cocycles ranging in the group $\Lambda$.

\begin{de}[\textbf{Geometrically rigid}]\label{de-ast}
Let $\Lambda$ be a countable discrete group. Let $\bbD$ be a countable discrete $\Lambda$-space, let $K$ be a compact metric space equipped with a $\Lambda$-action by homeomorphisms, and let $\Delta$ be a Polish space equipped with a $\Lambda$-action by Borel automorphisms. We say that the triple $(\bbD,K,\Delta)$ is \emph{geometrically rigid} if the following hold:
\begin{itemize}
\item[$(1)$]
The space $K$ admits a Borel $\Lambda$-invariant partition $K=K_{\bdd}\sqcup K_\infty$ such that
\begin{itemize}
\item there exists a Borel $\Lambda$-equivariant map $\theta_{\bdd}:K_{\bdd}\to\calp_{<\infty}(\bbD)$,
\item there exists a Borel $\Lambda$-equivariant map $\theta_{\infty}:K_\infty\to\Delta$.
\end{itemize}
\item[$(2)$] There exists a Borel $\Lambda$-equivariant map $\text{bar}:\Delta^{(3)}\to\calp_{<\infty}(\bbD)$.
\item[$(3)$] The $\Lambda$-action on $\Delta$ is universally amenable.
\end{itemize}
We say that $\Lambda$ is \emph{geometrically rigid with respect to $\mathbb{D}$} if there exists a compact metric $\Lambda$-space $K$ and a Polish $\Lambda$-space $\Delta$ such that the triple $(\bbD,K,\Delta)$ is geometrically rigid.
\end{de}

\paragraph{Some examples.} A first example is the following: every Gromov hyperbolic group $\Lambda$ is geometrically rigid with respect to itself. Indeed, let $\Delta$ be the Gromov boundary of $\Lambda$, and $K:=\Lambda\cup\Delta$. Then Assumption~$(1)$ clearly holds true by letting $K_{\bdd}:=\Lambda$ and $K_{\infty}:=\Delta$, and letting $\theta_{\bdd}$ and $\theta_{\infty}$ be the identity maps. Assertion~$(2)$ states that we can associate a `barycenter' (which is a finite subset of $\Lambda$) to any triple of pairwise distinct points in $\Delta=\partial_\infty \Lambda$: this was established by Adams in \cite[Section~6]{Ada3}. The Borel amenability of the $\Lambda$-action on $\Delta$ was also established by Adams in \cite{Ada3}.

More generally, we will see in Section~\ref{sec:rel-hyp} that if $\Lambda$ is hyperbolic relative to a finite set $\calp$ of subgroups, then $\Lambda$ is geometrically rigid with respect to the vertex set of a hyperbolic graph whose point stabilizers are all parabolic. 

We will also see (Proposition~\ref{prop:ast-mcg}) that the mapping class group $\Mod(\Sigma)$ of a connected, orientable hyperbolic surface $\Sigma$ of finite type is geometrically rigid with respect to the collection $\bbD$ of isotopy classes of essential simple closed curves, by taking for $K$ the Thurston compactification of the Teichmüller space, and for $\Delta$ the Gromov boundary of the curve graph. 

Finally, a key point of the present paper is to show that $\Out(F_N)$, or more generally the group $\Out(G,\calf^{(t)})$ in the context of free products (see Section~\ref{sec:background} for definitions), is geometrically rigid with respect to the countable collection $\bbD$ of all conjugacy classes of proper free factors (taking for $K$ the compactification of the projectivized Outer space, and for $\Delta$ the boundary of the free factor graph): this will be completed in Proposition~\ref{prop:ast-out}. A crucial task in the sequel of the paper will be to build the barycenter map in this setting (in particular for free products).

\subsection{A Bader--Furman type argument}

Our goal is now to prove the following proposition, following a method developed by Bader and Furman in \cite{BF} or \cite{BFHyp}. 

\begin{prop}\label{abstract-2}
Let $G$ be a locally compact group, and let $B$ be a strong boundary for $G$. Let $X$ be a standard probability space equipped with an ergodic measure-preserving $G$-action. Let $\Lambda$ be a countable discrete group, let $\bbD$ be a countable discrete $\Lambda$-space, and assume that $\Lambda$ is geometrically rigid with respect to $\bbD$. 

Then there exists a Polish space $\Delta$ equipped with a universally amenable $\Lambda$-action, such that for every cocycle $c:G\times X\to \Lambda$, either 
\begin{enumerate}
\item the cocycle $c$ is cohomologous to a cocycle that takes its values in the setwise stabilizer of a finite subset of $\bbD$, or
\item there exists a $c$-equivariant Borel map $X\to\Delta$ or a $c$-equivariant Borel map $X\to\Delta^{(2)}/\mathfrak{S}_2$, or else
\item there exists a nontrivial homomorphism $W_{G,B}\to\mathbb{Z}/2\mathbb{Z}$ whose kernel is a special subgroup of the Weyl group $W_{G,B}$.  
\end{enumerate}
\end{prop}

In the sequel of this section, we will adopt the notations from Proposition~\ref{abstract-2}. We let $K$ be a compact metrizable space with a $\Lambda$-action by homeomorphisms, and $\Delta$ be a Polish space equipped with a universally amenable $\Lambda$-action, both coming from the definition of geometric rigidity of $\Lambda$ with respect to $\bbD$. We let $\Prob(\Delta)$ be the set of all Borel probability measures on $\Delta$ equipped with the topology generated by the maps $\mu\mapsto\int fd\mu$, where $f$ varies over the set of bounded continuous real-valued functions on $\Delta$, again a Polish space \cite[Theorem~17.23]{Kec}. We denote by $\text{Prob}_{\ge 3}(\Delta)$ the space of all probability measures on $\Delta$ whose support contains at least $3$ points: this is a Borel subset of $\Prob(\Delta)$. For every $i\in\mathbb{N}$, we let 
\begin{displaymath}
\begin{array}{cccc}
c_i:& G\times X\times B^i&\to & \Lambda\\
& (g,x,b_1,\dots,b_i)&\mapsto & c(g,x)
\end{array}
\end{displaymath}
which is again a cocycle.

\begin{lemma}\label{cases}
One of the following holds.
\begin{enumerate}
\item There exists a $c_1$-equivariant Borel map $X\times B\to\calp_{<\infty}(\bbD)$.
\item There exists a $c_1$-equivariant Borel map $X\times B\to\Prob_{\ge 3}(\Delta)$.
\item There exist $i\in\{1,2\}$ and a $c_1$-equivariant Borel map $X\times B\to\Delta^{(i)}/\mathfrak{S}_i$.
\end{enumerate}
\end{lemma}

\begin{proof}
  Since $B$ is a strong boundary for $G$, the $G$-action on $B$ is Zimmer amenable, so by \cite[Proposition~4.3.4]{Zim}, the $G$-action on $X\times B$ is also Zimmer amenable. As $K$ is a compact metric space,
  Lemma~\ref{lemma:invariant-proba} ensures that there exists a $c_1$-equivariant Borel map
  $$
  \begin{array}{rcl}
    \nu:X\times B&\to&\text{Prob}(K)\\
    (x,b)&\mapsto& \nu_{x,b}
  \end{array}
$$
 Let $$V_\bdd:=\{(x,b)\in X\times B| \nu_{x,b}(K_{\infty})=0\}$$ and $$V_\infty:=\{(x,b)\in X\times B|\nu_{x,b}(K_{\bdd})=0\}.$$ Then $V_\bdd$ and $V_\infty$ are Borel subsets of $X\times B$, see e.g.\ \cite[Theorem~17.25]{Kec}. In addition, as $K_{\bdd}$ and $K_{\infty}$ are both $\Lambda$-invariant, and $\nu$ is $c_1$-equivariant, the sets $V_\bdd$ and $V_\infty$ are both $G$-invariant. Since $B$ is a strong boundary for $G$, the $G$-action on $X\times B$ is ergodic (Remark~\ref{ergodicities}). Therefore, either $V_{\bdd}$ or $V_{\infty}$ (or both) has measure $0$. By replacing $\nu$ by (the renormalization of) its restriction to either $K_\infty$ or $K_{\bdd}$, we can thus assume that either
\begin{enumerate}
\item[(1)] the probability measure $\nu_{x,b}$ gives full measure to $K_{\bdd}$ for a.e.\ $(x,b)\in X\times B$, or 
\item[(2)] the probability measure $\nu_{x,b}$ gives full measure to $K_\infty$ for a.e.\ $(x,b)\in X\times B$.
\end{enumerate}

We first assume that (1) holds. Since there is a $\Lambda$-equivariant Borel map $\theta_{\bdd}:K_{\bdd}\to \calp_{<\infty}(\bbD)$ (by the first assumption from Definition~\ref{de-ast} of geometric rigidity), we deduce that there exists a $c_1$-equivariant Borel map $X\times B\to\text{Prob}(\calp_{<\infty}(\bbD))$. Since $\calp_{<\infty}(\bbD)$ is countable, there is a $\Lambda$-equivariant Borel map $\text{Prob}(\calp_{<\infty}(\bbD))\to\calp_{<\infty}(\bbD)$, sending a probability measure $\lambda$ on $\calp_{<\infty}(\bbD)$ to the union of all finite subsets of $\bbD$ having maximal $\lambda$-measure. Hence the first conclusion of the lemma holds.

We now assume that (2) holds. Since there is a Borel $\Lambda$-equivariant map $\theta_{\infty}:K_\infty\to \Delta$, we deduce that there exists a $c_1$-equivariant Borel map $X\times B\to\text{Prob}(\Delta)$. Since the cardinality of the support of the measure $\nu_{x,b}$ is $G$-invariant, using again the ergodicity of the $G$-action on $X\times B$, we deduce that one of the last three conclusions of the lemma holds. 
\end{proof}
The following lemma deals with the first case of Lemma~\ref{cases}: when applied to $\hat{\mathbb{D}}=\calp_{<\infty}(\bbD)$ and $i=1$, it shows that the first case from Lemma~\ref{cases} yields the first conclusion of Proposition~\ref{abstract-2}.

\newcommand{\hD}{\hat{\mathbb{D}}}

\begin{lemma}\label{case-1}
Let $\hD$ be a countable discrete $\Lambda$-space. Assume that for some $i\in\{1,2\}$, there exists a $c_i$-equivariant Borel map $f:X\times B^i\to\hD$. 

Then $c$ is cohomologous to a cocycle that takes its values in the stabilizer of an element of $\hD$. 
\end{lemma}

\begin{proof}
In view of Lemma~\ref{interpretation-cohomologous}, it is enough to construct a $c$-equivariant Borel map $X\to\hD$.

We equip $\hD$ with the discrete metric (i.e.\ $d_{\hD}(x,y)=1$ if $x\neq y$), which is $\Lambda$-invariant. Let $\LL:=\LL^0(X,\hD)$, equipped with the metric $\delta$ and the isometric action of $G$ given in Section~\ref{sec-stuff}. 

By Lemma~\ref{lem:L(X,Y)}, the map $f$ gives rise to a Lebesgue-measurable map $\overline f:B^i\to \LL$, defined by $\overline f(b)(x)=f(x,b)$ for all $b\in B^i$ and all $x\in X$. Furthermore, the fact that $f$ is $c_i$-equivariant translates into the fact that $\overline f$ is $G$-equivariant.

Since $B$ is a strong boundary for $G$, the $G$-actions on both $B$ and $B\times B$ are isometrically ergodic. It follows that the map $\overline f$ is essentially constant. Its essential image is therefore a $G$-invariant measurable function $s_0\in \LL$. In other words $s_0:X\to \hD$ satisfies that, for a.e.\ $x\in X$ and all $g\in G$, we have $s_0(g x)=c(g,x)s_0(x)$. That is, $s_0$ is a $c$-equivariant measurable function, and it is therefore almost everywhere equal to some $c$-equivariant Borel map $X\to\hD$, as desired.
\end{proof}

The following lemma deals with the second case from Lemma~\ref{cases}.

\begin{lemma}\label{case-2}
  Assume that for some $i\in\{1,2\}$, there exists a $c_i$-equivariant Borel map $\nu:X\times B^i\to\text{Prob}_{\ge 3}(\Delta)$.

Then $c$ is cohomologous to a cocycle that takes its values in the setwise stabilizer of a finite subset of $\bbD$.
\end{lemma}

\begin{proof}
  For a.e.\ $(x,b)\in X\times B^i$, the product measure $\nu_{x,b}\times \nu_{x,b}\times \nu_{x,b}$ on $\Delta^3$ gives positive measure to $\Delta^{(3)}$, so it induces a probability measure on $\Delta^{(3)}$ (after renormalizing).  Pushing it forward through the $\Lambda$-equivariant map 
$bar:\Delta^{(3)}\to\calp_{<\infty}(\bbD)$, this yields a $c_i$-equivariant Borel map $X\times B^i\to\Prob(\calp_{<\infty}(\bbD))$. Since $\calp_{<\infty}(\bbD)$ is countable, there is a Borel $\Lambda$-equivariant map $\Prob(\calp_{<\infty}(\bbD))\to\calp_{<\infty}(\bbD)$ (sending a measure $\lambda$ to the union of all elements of $\calp_{<\infty}(\bbD)$ with maximal $\lambda$-measure). The conclusion then follows from Lemma~\ref{case-1} applied to $\hat{\bbD}:=\calp_{<\infty}(\bbD)$.
\end{proof}

In order to tackle the third case from Lemma~\ref{cases}, we start with the following lemma.

\begin{lemma}\label{case-3}
Assume that for some $i\in\{1,2\}$, there exists a $c_i$-equivariant Borel map $X\times B\to\Delta^{(i)}/\mathfrak{S}_i$. Then either
\begin{enumerate}
\item $c$ is cohomologous to a cocycle that takes its values in the stabilizer of a finite subset of $\bbD$, or
\item there exist $i\in\{1,2\}$ and a $c$-equivariant Borel map $X\to\Delta^{(i)}/\mathfrak{S}_i$, or else
\item there is a $c_1$-equivariant Borel map $X\times B\to\Delta$, and an essentially unique $c_2$-equivariant Borel map $X\times B\times B\to\Delta^{(2)}/\mathfrak{S}_2$.
\end{enumerate}
\end{lemma}

\begin{proof}
Let $f_1:X\times B\to\Delta^{(i)}/\mathfrak{S}_i$ be a $c_i$-equivariant Borel map. Then the sets $$V_{=}:=\{(x,b,b')\in X\times B\times B|f_1(x,b)=f_1(x,b')\}$$ and $$V_{\neq}:=\{(x,b,b')\in X\times B\times B| f_1(x,b)\neq f_1(x,b')\}$$ are $G$-invariant (up to null sets). By ergodicity of the $G$-action on $X\times B\times B$, it follows that one of these sets is null and the other is conull.

If $V_{=}$ is conull, then the map $f_1$ essentially does not depend on its second coordinate. Therefore it yields a $c$-equivariant Borel map $X\to\Delta^{(i)}/\mathfrak{S}_i$, so the second conclusion of the lemma holds.

We now assume that $V_{\neq}$ is conull. If $i=2$, then there is a $c_2$-equivariant Borel map $f'$ from a subset of full measure of $X\times B\times B$ to $\Delta^{(3)}/\mathfrak{S}_3\cup\Delta^{(4)}/\mathfrak{S}_4$, defined by letting $f'(x,b,b')=f_1(x,b)\cup f_1(x,b')$. In this case Lemma~\ref{case-2} shows that the first conclusion of the lemma holds.

We now assume that $i=1$. As above,
we have a $c_2$-equivariant map
$$
\begin{array}{rccl}
f_2:&X\times B\times B &\to& \Delta^{(2)}/\mathfrak{S}_2\\
&(x,b,b') &\mapsto & \{f_1(x,b),f_1(x,b')\}
\end{array}
$$
If such a map is essentially unique, then the third conclusion of the lemma holds.

We can therefore assume that there exists another $c_2$-equivariant map $f_2':X\times B\times B\to \Delta^{(2)}/\mathfrak{S}_2$ which is not a.e.\ equal to $f_2$. By ergodicity of the $G$-action on $X\times B\times B$, we have $f_2(x,b,b')\neq f'_2(x,b,b')$ for a.e.\ $(x,b,b')\in X\times B\times B$. Hence the map $(x,b,b')\mapsto f_2(x,b,b')\cup f'_2(x,b,b')$ yields a $c_2$-equivariant Borel map from $X\times B\times B$ to $\Delta^{(3)}/\mathfrak{S}_3\cup \Delta^{(4)}/\mathfrak{S}_4$. Lemma~\ref{case-2} therefore implies that $c$ is cohomologous to a cocycle that takes its values in the stabilizer of a finite subset of $\mathbb{D}$, so the first conclusion of the lemma holds.
\end{proof}

The following lemma completes our proof of Proposition~\ref{abstract-2}.

\begin{lemma}
Assume that there is no $c$-equivariant Borel map $X\to\Delta$. Assume in addition that there exist a $c_1$-equivariant Borel map $f_1:X\times B\to\Delta$ and an essentially unique $c_2$-equivariant Borel map $X\times B\times B\to\Delta^{(2)}/\mathfrak{S}_2$.

Then there exists a nontrivial homomorphism $W_{G,B}\to\mathbb{Z}/2\mathbb{Z}$ whose kernel is a special subgroup of the Weyl group $W_{G,B}$.
\end{lemma}

\begin{proof}
  Consider the $c_2$-equivariant Borel maps $X\times B\times B\to\Delta^{(2)}$ defined by
  $$f_1^2:(x,b,b')\mapsto (f_1(x,b),f_1(x,b'))$$ and $$(f_1^2)^\ast:(x,b,b')\mapsto (f_1(x,b'),f_1(x,b)).$$
  We note that $f_1^2$ and $(f_1^2)^\ast$ essentially take their values in $\Delta^{(2)}$: otherwise, by ergodicity,
  one has $f_1(x,b)=f_1(x,b')$ for almost every $(x,b,b')\in X\times B\times B$ so $f_1$ defines a $c$-equivariant map $f:X\ra \Delta$, contradicting our hypothesis.

  We claim that the maps $f_1^2$ and $(f_1^2)^\ast$ are the only $c_2$-equivariant Borel maps $X\times B\times B\to\Delta^{(2)}$.  
  Indeed, if $\psi:X\times B\times B\to \Delta^{(2)}$ is a $c_2$-equivariant Borel map, by uniqueness of the map $X\times B\times B\to \Delta^{(2)}/\mathfrak{S}_2$, we get that for a.e.\ $(x,b,b')$, either $\psi(x,b,b')=f_1^2(x,b,b')$ or $\psi(x,b,b')=(f_1^2)^\ast(x,b,b')$. By ergodicity of the $G$-action on $X\times B\times B$, one of these equalities holds almost everywhere. This proves our claim.

Let $F_1:B\to  \LL^0(X,\Delta)$ be the map induced by $f_1$, and let $F_1^2,(F_1^2)^\ast:B\times B\to\LL^0(X,\Delta^{(2)})$ be the maps induced by $f_1^2$ and $(f_1^2)^\ast$.  Thus, the set of $G$-equivariant Borel maps $B\times B\ra \LL^0(X,\Delta^{(2)})$ is a set consisting of exactly two elements, namely $F_1^2$ and $(F_1^2)^*$.
The Weyl group $W_{G,B}$ acts by precomposition on this set, thus defining a homomorphism $m:W_{G,B}\to \bbZ/2\bbZ$.
The flip $(b,b')\mapsto (b',b)$ is not in the kernel of $m$,  as otherwise $F_1$ would be essentially constant, contradicting that there is no $c$-equivariant Borel map $X\to\Delta$. 
Thus $m$ is nontrivial and $\ker m$ is a normal subgroup of index $2$ of $W_{G,B}$.

There remains to prove that $\ker m$ is a special subgroup of $W_{G,B}$ (see Definition \ref{dfn_Weyl}).
Taking $C=L^0(X,\Delta)$ and $\pi=F_1:B\ra C$, it suffices to check that $\ker m=W_{G,B}(\pi)$
i.e.\ that
$w\in \ker m$ if and only if $F_1\circ p_1=F_1\circ p_1\circ w$.
Write $w\in W_{G,B}$ as $w(b,b')=(w_1(b,b'),w_2(b,b'))$.
If $w\in \ker m$, then $F_1^2\circ w=F_1^2$ so
$F_1(b)=F_1(w_1(b,b'))$ for almost every $(b,b')\in B^2$, which precisely means that $F_1\circ p_1=F_1\circ p_1\circ w$.

Conversely, assume that $F_1\circ p_1=F_1\circ p_1\circ w$. If
$w\notin \ker m$, then $F_1^2\circ w=(F_1^2)^*$ so
$F_1(b')=F_1(w_1(b,b'))$ for almost every $(b,b')\in B^2$, i.e.\ $F_1\circ p_2=F_1\circ p_1\circ w$.
If follows that $F_1\circ p_1=F_1\circ p_2$, hence $F_1$ is essentially constant. Its image is therefore a $G$-invariant function in $\LL^0(X,\Delta)$, in other words, a $c$-equivariant measurable map from $X$ to $\Delta$, a contradiction. 
This proves that $\ker m=W_{G,B}(\pi)$ is a special subgroup of $W_{G,B}$. 
\end{proof}

\subsection{Property (T) versus amenability of the action}

In the case where the group $G$ has property (T), and in fact more generally as soon as every cocycle from $G\times X$ to an amenable group is cohomologous to one with finite values, the second conclusion from Proposition~\ref{abstract-2} yields cocycle superrigidity, as was established by Spatzier and Zimmer in \cite{SZ}.

\begin{lemma}[Spatzier--Zimmer \cite{SZ}]\label{lem:elementary}
Let $G$ be a locally compact  second countable group with property $(T)$, and let $X$ be a standard Borel probability space equipped with an ergodic measure-preserving $G$-action. Let $\Lambda$ be a countable group acting on a standard Borel space $\Delta$, so that the $\Lambda$-action on $\Delta$ is universally amenable.

Let $c:G\times X\to \Lambda$ be a measurable cocycle, and assume that there exists a $c$-equivariant Borel map from $X$ to $\Delta$.

Then $c$ is cohomologous to a cocycle with values in a finite subgroup of $\Lambda$.  
\end{lemma}

\begin{proof}
The argument appears in the proof of \cite[Lemma~4.1]{SZ}, we recall it here. 
Let $G\times \Lambda\actson Y=X\times \Lambda$ be the developed action. The $c$-equivariant Borel map $X\to\Delta$ 
yields a $G$-invariant $\Lambda$-equivariant map $F:Y\ra \Delta$.
Let $M$ be the \emph{Mackey range} of $c$, i.e.\ the space of ergodic components of $G\actson Y$. 
It is naturally equipped with a $\Lambda $-action, and the map $F$ factors through a $\Lambda$-equivariant map $M\ra \Delta$.
As $\Delta$ is universally amenable, it follows that the $\Lambda $-action on $M$ is amenable in the sense of Zimmer, see e.g.\ \cite[Corollary~C]{AEG}. 

We first assume that the $\Lambda$-action on $M$ is essentially transitive. By \cite[Proposition 4.3.2]{Zim}, there exists an amenable subgroup $\Lambda^1\subseteq \Lambda$ and a $\Lambda$-equivariant isomorphism $M\ra \Lambda/\Lambda^1$. 
By Lemma~\ref{interpretation-cohomologous}, this implies that $c$ is cohomologous to a cocycle with values in $\Lambda^1$. The conclusion then follows from Zimmer's cocycle rigidity theorem \cite[Theorem~9.1.1]{Zim}.

We now assume that the $\Lambda$-action on $M$ is properly ergodic. By \cite{OW}, it is orbit-equivalent to an action of $\mathbb{Z}$. Let $\alpha:\Lambda\times M\to\mathbb{Z}$ be a cocycle coming from the orbit equivalence. Consider the cocycle $\beta:G\times X\to\mathbb{Z}$ defined by letting $\beta(g,x):=\alpha(c(g,x)^{-1},[gx,e])$ (where $[gx,e]$ denotes the image of $(gx,e)\in X\times \Lambda$ in $M$).
Then the Mackey range $\bbZ\actson M'$ of $\beta$ is orbit equivalent to $\Lambda\actson M$, and is properly ergodic in particular.
But since $G$ has Property~(T) and $\mathbb{Z}$ is amenable, it follows from  \cite[Theorem~9.1.1]{Zim} that the cocycle $\beta$ is cohomologically trivial, so that $\bbZ\actson M'$ has trivial orbits. This yields a contradiction. 
\end{proof}

\subsection{Specialization to the case where $G$ is an algebraic group}\label{sec-lie}

We now specialize Proposition~\ref{abstract-2} to the case where $G$ is a product of higher rank simple algebraic groups. The following is Theorem~\ref{theo:intro-cocycle-geometry} from the introduction.

\begin{theo}\label{theo:abstract}
Let $\Lambda$ be a countable group. Let $\bbD$ be a countable $\Lambda$-set. Assume that $\Lambda$ is geometrically rigid with respect to $\bbD$.  

Let $G$ be a product of connected higher rank simple algebraic groups over local fields. Let $X$ be a standard probability space equipped with an ergodic measure-preserving $G$-action. 

Then every cocycle $c:G\times X\to \Lambda$ is cohomologous to a cocycle that takes its values in a subgroup of $\Lambda$ that virtually fixes an element of $\bbD$.
\end{theo}

\begin{proof}
We apply Proposition~\ref{abstract-2} to the cocycle $c$. If the first conclusion of Proposition~\ref{abstract-2} holds, then $c$ is cohomologous to a cocycle with values in the setwise stabilizer of a nonempty finite subset of $\bbD$. Assume that the second conclusion of Proposition~\ref{abstract-2} holds, i.e.\ there is a $c$-equivariant Borel map from $X$ to either $\Delta$ or $\Delta^{(2)}/\mathfrak{S}_2$. The $\Lambda$-action on $\Delta$ is universally amenable by assumption, and therefore so is the $\Lambda$-action on $\Delta^{(2)}/\mathfrak{S}_2$ (see the argument in \cite[Lemma~5.2]{Ada}). Since $G$ has Property~$(T)$, Lemma~\ref{lem:elementary} therefore ensures that $c$ is cohomologous to a cocycle taking its values in a finite subgroup of $\Lambda$ and the conclusion holds. Finally, by Proposition~\ref{prop:WeylAlg}, the Weyl group $W_{G,B}$ does not contain any normal special subgroup of index 2, which excludes the third conclusion from Proposition~\ref{abstract-2}.
\end{proof}

\section{First examples of cocycle rigidity results}

\subsection{Cocycle rigidity with target a relatively hyperbolic group}\label{sec:rel-hyp}

The following theorem was announced by Bader--Furman, see the unpublished paper \cite{BFHyp}. Since, to our knowledge, its statement does not appear elsewhere in the literature, we include a proof in the present section.  

\begin{theo}[{Bader--Furman \cite{BFHyp}}]\label{theo:rhyp}
Let $\Lambda$ be a group which is hyperbolic relative to a finite collection $\calp$ of subgroups. 
Let $G_0$ be a product of connected higher rank simple algebraic groups defined over local fields. Let $G$ be either $G_0$, or a lattice in $G_0$. Let $X$ be a standard probability space equipped with an ergodic measure-preserving action of $G$. 

Then every cocycle $G\times X\to\Lambda$ is cohomologous to a cocycle that takes its values in a parabolic or in a finite subgroup of $\Lambda$.
\end{theo} 

We recall here that a subgroup of $G$ is \emph{parabolic} if it is conjugate into one of the subgroups in $\calp$. In particular, if we assume that for every $P\in\calp$, the pair $(G,P)$ is cocycle-rigid, then $(G,\Lambda)$ is cocycle-rigid. In the particular case where $\calp=\emptyset$, we recover the following result originally due to Adams.

\begin{cor}[Adams \cite{Ada2}]
Let $\Lambda$ be a hyperbolic group. 
Let $G_0$ be a product of connected higher rank simple algebraic groups defined over local fields. Let $G$ be either $G_0$, or a lattice in $G_0$. Let $X$ be a standard probability space equipped with an ergodic measure-preserving action of $G$.

Then every cocycle $G\times X\to\Lambda$ is cohomologous to a cocycle that takes its values in a finite subgroup of $\Lambda$.
\end{cor}

\begin{proof}[Proof of Theorem~\ref{theo:rhyp}]
By work of Bowditch \cite{Bow}, there exists a \emph{uniformly fine} hyperbolic graph $Y$ on countably many vertices equipped with an isometric $\Lambda$-action whose nontrivial point stabilizers are all conjugate to subgroups in $\calp$. Here, we recall that uniform fineness means that for every $n\in\mathbb{N}$, there exists $C_n>0$ such that for every edge $e\subseteq Y$, there are at most $C_n$ closed circuits of length $n$ that contain $e$. 

We let $\bbD$ be the vertex set of $Y$. We will prove that $\Lambda$ is geometrically rigid with respect to $\bbD$; 
Theorem~\ref{theo:abstract} then shows that the cocycle is cohomologous to a cocycle with values in a subgroup $K\subseteq\Lambda$
where some finite index subgroup of $K$ fixes a point in $\bbD$.
Thus if $K$ is infinite, then $K$ has a finite index subgroup which is contained in a parabolic subgroup $P$.
Since parabolic subgroups are almost malnormal (i.e.~$P^g\cap P$ is finite for $g\notin P$), it follows that $H$ is in fact contained in $P$.

Let $\Delta$ be the Gromov boundary of $Y$, and let $K:=\bbD\cup\Delta$. Then $K$ has a natural structure of a compact metric space (usually called the \emph{Bowditch boundary} of $(\Lambda,\calp)$) in which $\Delta$ is a Borel subset, and $\Lambda$ acts on $K$ by homeomorphisms \cite[Sections~8 and~9]{Bow}. In particular $\Delta$ is a Polish space.

The `barycenter map' was essentially constructed by Bowditch in \cite[Section~8]{Bow} (see e.g.\ the comment right after \cite[Lemma~8.2]{Bow}), with an argument of Kida \cite[Section~4.1.2]{Kid} for its measurability. More precisely, using \cite[Lemma~8.2]{Bow} (which relies on the fineness on $Y$), one sees that the graph $Y$ has the following two properties:
\begin{enumerate}
\item for every $r>0$, there exists $C_1>0$ such that given any two points $x,y\in Y$, and any point $z$ lying on a geodesic from $x$ to $y$, there are at most $C_1$ points that lie on some geodesic from $x$ to $y$ and are contained in $B(z,r)$; 
\item for every $r>0$, there exist $C_2>0$ and $\kappa>0$ such that for every $r'>0$, given any two points $x,y\in Y$, and any point $z\in Y$ lying on a geodesic from $x$ to $y$ and such that $d(z,\{x,y\})\ge r+\kappa$, the intersection of $B(z,r)$ with the collection of all geodesics in $Y$ with origin in $B(x,r')$ and endpoint in $B(y,r')$ has cardinality at most $C_2$.
\end{enumerate}
These are precisely the conditions required by Kida to ensure the existence of a $\Lambda$-equivariant Borel map $\Delta^{(3)}\to\calp_{<\infty}(\bbD)$, see \cite[Section~4.1.2]{Kid}.

Finally, universal amenability of the $\Lambda$-action on $\Delta$ was established by Ozawa in \cite[Lemma~8]{Oza}.
\end{proof}

\subsection{Cocycle rigidity with target a right-angled Artin group}\label{subsec_RAAG}

In order to illustrate the ubiquity of our notion of geometric rigidity in geometric group theory, we will now use it to prove cocycle rigidity with target a right-angled Artin group. The results of the present section will not be used later in the paper. The analogous rigidity statement for homomorphisms from higher-rank lattices to right-angled Artin groups was first established by Behrstock and Charney in \cite{BC}. We also mention that  rigidity in the realm of group actions on cube complexes has been extensively studied by Chatterji, Fern\'os and Iozzi in \cite{CFI} using bounded cohomology methods, with applications to orbit equivalence rigidity through work of Monod and Shalom \cite{MS3}.

Given a finite simple graph $\Gamma$ (i.e.\ $\Gamma$ has no loop-edge and no multiple edges joining two vertices), recall that the \emph{right-angled Artin group} $A_\Gamma$ associated to $\Gamma$ is the group defined by the following presentation: it has one generator per vertex of $\Gamma$, and two generators commute whenever the corresponding vertices of $\Gamma$ are joined by an edge. 
 
\begin{theo}\label{theo:raag}
Let $\Gamma$ be a finite simple graph, and let $A_{\Gamma}$ be the associated right-angled Artin group.

Let $G_0$ be a product of connected higher rank simple algebraic groups defined over local fields. Let $G$ be either $G_0$, or a lattice in $G_0$. Let $X$ be a standard probability space equipped with an ergodic measure-preserving action of $G$. 

Then every cocycle $G\times X\to A_{\Gamma}$ is cohomologous to the trivial cocycle.
\end{theo}

Given a full subgraph $\Lambda\subseteq\Gamma$ (i.e.\ two vertices of $\Lambda$ are joined by an edge if and only if they are joined by an edge in $\Gamma$), the Salvetti complex $S_\Lambda$ naturally embeds as a subcomplex of $S_\Gamma$. A \emph{standard subcomplex} of type $\Lambda$ of the universal cover $\widetilde{S}_\Gamma$ is a connected component of the preimage of $S_\Lambda$ under the covering map. Notice that its stabilizer is a conjugate of $A_\Lambda$. In particular, the vertices of $\widetilde{S}_\Gamma$ are the standard subcomplexes associated to $\Lambda=\emptyset$. We denote by $\mathbb{D}_\Gamma$ the countable set of all proper (i.e.\ with $\Lambda\neq\Gamma$) standard subcomplexes of $\widetilde{S}_\Gamma$. 

\begin{lemma}\label{lemma:rigid-raag}
Let $\Gamma$ be a finite simple graph that does not split nontrivially as a join. Then $A_\Gamma$ is geometrically rigid with respect to $\bbD_\Gamma$.
\end{lemma}

\begin{proof}
Let $\widetilde{S}_\Gamma$ be the universal cover of the Salvetti complex of $A_\Gamma$. Let $K$ be the Roller boundary of $\widetilde{S}_\Gamma$, a compact metrizable space. Let $K_\infty\subseteq K$ be the Borel subset that consists of all \emph{regular} points as defined by Fern\'os in \cite[Definition~3.7]{Fer}, and we let $K_{\bdd}=K\setminus K_\infty$ and $\Delta=K_\infty$. 

By \cite[Theorem~2.1]{HH}, as $\Gamma$ does not split nontrivially as a join, there exists a Borel $A_\Gamma$-equivariant map $K_\bdd\to\mathbb{D}_\Gamma$ that associates a proper standard subcomplex to every nonregular point of the Roller boundary.

The barycenter map, associating a vertex in $\tilde{S}_\Gamma$ to every triple of pairwise distinct points in $K_\infty$, is provided by \cite[Lemma~5.14]{FLM}. 

Finally, as the $A_\Gamma$-action on the Roller boundary of $\tilde{S}_\Gamma$ is Borel amenable \cite[Theorem~5.11]{Duc} whence universally amenable, so is the $A_\Gamma$-action on $K_\infty$ (alternatively, we may use \cite[Theorem~7.2]{NS} and
the fact that regular points of the Roller boundary correspond to non-terminating ultrafilters). 
\end{proof}

We are now in position to complete our proof of Theorem~\ref{theo:raag}.

\begin{proof}[Proof of Theorem~\ref{theo:raag}]
The proof is by induction on the number of vertices of the defining graph $\Gamma$. When $\Gamma$ is reduced to a point, the group $A_\Gamma$ is isomorphic to $\mathbb{Z}$, so the conclusion follows from \cite[Theorem~9.1.1]{Zim}.

If $\Gamma$ is a join of two subgraphs $\Gamma_1$ and $\Gamma_2$, then $A_\Gamma$ decomposes as a direct product $A_{\Gamma_1}\times A_{\Gamma_2}$, and the conclusion follows by induction using the fact that cocycle superrigidity is stable under direct products as a property of the target group. 

We now assume that $\Gamma$ is not a join. As $A_\Gamma$ is geometrically rigid with respect to $\bbD_\Gamma$ (Lemma~\ref{lemma:rigid-raag}), it follows from Theorem~\ref{theo:abstract} that every cocycle $G\times X\to A_\Gamma$ is cohomologous to a cocycle that takes its values in the stabilizer of a proper standard subcomplex of $\tilde{S}_\Gamma$, which is a conjugate of some proper parabolic subgroup $A_\Lambda$. The theorem follows by induction.
\end{proof}

\part{Outer automorphism groups as targets}

\section{Background on outer automorphisms of free products}\label{sec:background}

\subsection{General context} Let $G_1,\dots,G_k$ be countable groups, let $F_N$ be a free group of rank $N$, and let $$G:=G_1\ast\dots\ast G_k\ast F_N.$$ We denote by $\calf=\{[G_1],\dots,[G_k]\}$ the set of conjugacy classes of the subgroups $G_i$ in $G$. We define the \emph{complexity} of $(G,\calf)$ as the pair $\xi(G,\calf)=(k+N,N)$; complexities will be ordered lexicographically. We say that $(G,\calf)$ is \emph{sporadic} if $\xi(G,\calf)\le (2,1)$, and \emph{nonsporadic} otherwise. More concretely, sporadic cases correspond to the following five possibilities:
\begin{itemize}
\item $G=\{1\}$ and $\calf=\emptyset$ (equivalently $\xi(G,\calf)=(0,0)$),
\item $G=G_1$ and $\calf=\{[G_1]\}$ (equivalently $\xi(G,\calf)=(1,0)$),
\item $G=\mathbb{Z}$ and $\calf=\emptyset$ (equivalently $\xi(G,\calf)=(1,1)$),
\item $G=G_1\ast G_2$ and $\calf=\{[G_1],[G_2]\}$ (equivalently $\xi(G,\calf)=(2,0)$),
\item $G=G_1\ast\mathbb{Z}$ and $\calf=\{[G_1]\}$ (equivalently $\xi(G,\calf)=(2,1)$).
\end{itemize}
Subgroups of $G$ that are conjugate into one of the subgroups in $\calf$ are called \emph{$\calf$-peripheral}. We denote by $\Out(G,\calf)$ the subgroup of $\Out(G)$ made of all outer automorphisms that preserve the conjugacy class of each subgroup $G_i$. We denote by $\Out(G,\calf^{(t)})$ the subgroup made of all outer automorphisms which have a representative in $\Aut(G)$ whose restriction to each of the factors $G_i$ is the conjugation by an element $g_i\in G$. 

\subsection{Trees, free splittings, free factors of $(G,\calf)$} A \emph{$(G,\calf)$-tree} is an $\mathbb{R}$-tree $T$ equipped with a minimal $G$-action by isometries, such that every subgroup in $\calf$ fixes a point in $T$ (we recall that the $G$-action on $T$ is said to be \emph{minimal} if $T$ does not contain any proper $G$-invariant subtree). It is \emph{relatively free} if every subgroup of $G$ that fixes a point is $\calf$-peripheral. A minimal $(G,\calf)$-tree is \emph{trivial} if it is reduced to one point. Given a $(G,\calf)$-tree $T$ and an element $g\in G$, the \emph{translation length} of $g$ in $T$ is defined as $||g||_T=\inf_{x\in T}d(x,gx)$. The infimum in this definition is always achieved: one has $||g||_T=0$ if and only if $g$ fixes a point in $T$, and otherwise $g$ has an invariant axis homeomorphic to a line in $T$ on which it acts as a translation of length $||g||_T$.

A \emph{$(G,\calf)$-free splitting} is a simplicial (non-metric) $(G,\calf)$-tree in which all edge stabilizers are trivial. We will always consider $(G,\calf)$-free splittings up to $G$-equivariant homeomorphisms. We denote by $\FS$ the countable set of all (equivariant homeomorphism classes of) $(G,\calf)$-free splittings. 

A \emph{$(G,\calf)$-free factor} is a subgroup of $G$ which is equal to a point stabilizer in some $(G,\calf)$-free splitting. A $(G,\calf)$-free factor is \emph{proper} if it is nontrivial, non-$\calf$-peripheral and not equal to $G$. More generally, a collection of conjugacy classes of proper $(G,\calf)$-free factors is a \emph{$(G,\calf)$-free factor system} if it coincides with the collection of conjugacy classes of nontrivial point stabilizers in some $(G,\calf)$-free splitting. 

 Every proper $(G,\calf)$-free factor $A$ inherits a decomposition as a free product $A=A_1\ast\dots\ast A_l\ast F_r$ from the splitting of $G$: the factors $A_i$ that arise in this decomposition form a set of representatives of the nontrivial point stabilizers in the minimal $A$-invariant subtree of any relatively free $(G,\calf)$-free splitting. We will denote by $\calf_{|A}$ the collection of all the $A$-conjugacy classes of the subgroups $A_i$ from the above decomposition of $A$: this is a free factor system of $A$.

\subsection{Outer space and its closure}

A \emph{Grushko $(G,\calf)$-tree} is a minimal simplicial metric relatively free $(G,\calf)$-tree (in the sequel, when $\calf$ is clear from context, we will simply say `Grushko tree'). The \emph{unprojectivized relative Outer space} $\calo$ is the space of all $G$-equivariant isometry classes of Grushko $(G,\calf)$-trees \cite{CV,GL}, and
we denote by $\mathbb{P}\calo$ the projectivized Outer space, where trees are considered up to homothety instead of isometry. 
The set $\bbP\calo$ has a natural structure of a simplicial complex (with missing faces)
but this complex is in general locally infinite and we will not use
the corresponding weak topology.
Instead, we equip $\calo$ with the equivariant Gromov--Hausdorff topology introduced by Paulin in \cite{Pau}, which coincide with the translation lengths topology \cite{Pau2}.
We equip $\bbP\calo$ with the quotient topology.

The closure $\baro$ of $\calo$ in the space of all minimal nontrivial $(G,\calf)$-trees (equipped with the equivariant Gromov--Hausdorff topology, or equivalently with the translation lengths topology \cite{Pau2}) is equal to the space of all \emph{very small} $(G,\calf)$-trees, i.e.\ those trees $T$ for which stabilizers of nondegenerate tripods are trivial, and stabilizers of nondegenerate arcs are either trivial, or isomorphic to $\mathbb{Z}$, root-closed, and nonperipheral \cite{CL,BeF2,Hor2}. We denote by $\mathbb{P}\baro$ the projectivization of $\baro$, equipped with the quotient topology. This is a compact space, as was proved by Culler and Morgan in \cite{CM} (see also \cite[Theorem~1]{Hor2} in this setting). 
We let $\partial\calo:=\overline{\calo}\setminus\calo$. We warn the reader that $\partial\calo$ is not closed in $\overline{\calo}$ in general, see \cite[Figure 2]{Gui00} for an explicit example of a sequence of trees in $\partial\calo$ converging to a point in $\calo$. Similarly in the projective setting, $\bbP\calo$ is not open in $\bbP\baro$.

\subsection{Morphisms and folding paths}\label{sec_morphisms} 
 
Given two trees $S,T\in\overline{\calo}$, a $G$-equivariant map $f:S\to T$ is \emph{piecewise linear} if each segment of $S$ can be subdivided into finitely many
subsegments in restriction to which $f$ is a (maybe constant) homothety.
It is a \emph{morphism} if each segment of $S$ can be subdivided into finitely many
subsegments in restriction to which $f$ is isometric.
A piecewise linear map $f:S\to T$ is \emph{optimal} if every point $x\in S$ is contained in a segment $[u,v]\subset S$ with $x\notin \{u,v\}$ and such that $f_{|[u,v]}$ is injective.

For any $S\in\calo$ and $T\in \bar\calo$, there exists an optimal
piecewise linear map $f:S\ra T$ \cite[Corollary 6.8]{FrMa}, and by changing lengths of edges of $S$ (allowing some edge lengths to be set to $0$), one can construct $S'$
with  an optimal morphism $f:S'\ra T$.

Given a compact interval $[a,b]\subset\bbR$,
a \emph{folding path} 
$(T_t)_{t\in [a,b]}$ 
is a continuous family of trees in $\calo$ with morphisms $f_{t_1,t_2}:T_{t_1}\ra T_{t_2}$ for every $a\leq t_1\leq t_2\leq b$, 
such that $f_{t_1,t_3}=f_{t_2,t_3}\circ f_{t_1,t_2}$ for every
$a\leq t_1\le t_2\le t_3\leq b$.

A folding path is \emph{optimal} if the morphism $f_{a,b}$ is optimal (this implies that all morphisms $f_{t_1,t_2}$ are optimal). Given any optimal morphism $f:T\to T'$, there exists a folding path $(T_t)_{t\in [a,b]}$ with $f_{a,b}=f$ (in particular $T_a=T$ and $T_b=T'$), see e.g.\ \cite[Section~3.1]{GL}.

\subsection{Arational trees} 

The notion of an arational tree was first introduced by Reynolds in the context of free groups \cite{Rey}. In the context of free products, a tree $T\in\overline{\calo}$ is \emph{arational} if $T\in\partial\calo$, no proper $(G,\calf)$-free factor $A$ fixes a point in $T$, and for every proper $(G,\calf)$-free factor $A$, the $A$-minimal invariant subtree $T_A\subseteq T$ is a Grushko $(A,\calf_{|A})$-tree. We denote by $\AT$ the subspace of $\overline{\calo}$ made of all arational $(G,\calf)$-trees, and by $\mathbb{P}\AT$ its projectivized version (a subspace of $\mathbb{P}\overline{\calo}$). These are Borel subsets of $\baro$ and $\mathbb{P}\baro$, respectively \cite[Lemma~5.5]{Hor}.

Two arational trees $T$ and $T'$ are \emph{equivalent} (denoted by $T\sim T'$) if they are $G$-equivariantly homeomorphic when equipped with the \emph{observers' topology} introduced in \cite{CHL3}: this is the topology on a tree $T$ which is generated (as a subbasis of open sets) by the connected components of complements of points in $T$.  Equivalently $T$ and $T'$ are equivalent if there exist $G$-equivariant alignment-preserving bijections from $T$ to $T'$ and from $T'$ to $T$ (here a map is \emph{alignment-preserving} if it sends every segment to a segment). A third equivalent characterization is that $T$ and $T'$ are equivalent if they are compatible, i.e.\ there exists a $G$-tree $T''$ coming with alignment-preserving $G$-equivariant maps to $T$ and $T'$ (notice that in this case $T''$ can always be chosen in $\overline{\calo}$ because $T$ and $T'$ have trivial arc stabilizers). See \cite[Corollary~13.4]{GH1} for the equivalence of these definitions.

\subsection{The free factor graph and its boundary}

We adopt the definition of the free factor graph from \cite{GH}, see \cite[Section~2.2]{GH} for a discussion of its quasi-isometry with other models of the graph found in the literature. We recall that two free splittings $S,S'$ of $(G,\calf)$ are \emph{compatible} if there exists a free splitting $S''$ of $(G,\calf)$ such that both $S$ and $S'$ are obtained from $S''$ by collapsing every edge from a $G$-invariant subset of edges to a point. The \emph{free factor graph} $\FF$ is the graph whose vertices are the $G$-equivariant homeomorphism classes of free splittings of $(G,\calf)$, where two vertices are joined by an edge if the corresponding splittings are either compatible or have a common nonperipheral elliptic element.

With this definition, there is a natural $\Out(G,\calf)$-equivariant map $\pi:\calo\to\FF$, which simply consists in forgetting the metric on the trees.

The free factor graph is hyperbolic whenever $(G,\calf)$ is nonsporadic: this was first proved by Bestvina--Feighn \cite{BeF} for free groups (i.e.\ when $G=F_N$ and $\calf=\emptyset$), and extended by Handel--Mosher in \cite{HM} to the context of free products, see also \cite[Section~2.2]{GH}. The Gromov boundary of $\FF$ was described by Bestvina--Reynolds \cite{BR} and Hamenstädt \cite{Ham} in the context of free groups, with an extension to free products in \cite{GH}, as the space of equivalence classes of arational $(G,\calf)$-trees; in fact these descriptions provide a continuous extension of the map $\pi$ to the set of arational $(G,\calf)$-trees in $\partial\calo$.  

\subsection{Algebraic laminations}

The notion of algebraic laminations for free groups was introduced and extensively studied by Coulbois, Hilion and Lustig in \cite{CHL,CHL2}; it was generalized to the context of free products in \cite{GH1}.

Given a Grushko tree $R$, we denote by $\partial_\infty R$ the Gromov boundary of $R$ and by $V_\infty(R)$ the set of vertices of infinite valence in $R$. We then let $\partial R:=\partial_\infty R\cup V_\infty(R)$. The space $R\cup\partial R$ is compact when equipped with the observers' topology \cite[Proposition~1.13]{CHL3}. The boundary $\partial R$ is closed in $R\cup\partial R$ (see the discussion in the opening paragraph of \cite[Section~2]{GH1}). Up to natural identifications (see \cite[Section 1]{GH1}), the spaces $\partial_\infty R, V_\infty(R)$ and $\partial R$ do not depend on the choice of a Grushko tree $R$; we thus denote them by $\partial_\infty(G,\calf)$, $V_\infty(G,\calf)$ and $\partial(G,\calf)$. We let $\partial^2(G,\calf):=\partial (G,\calf)\times\partial (G,\calf)\setminus\Delta$, where $\Delta$ denotes the diagonal. Given any nonperipheral element $g\in G$, we denote by $(g^{-\infty},g^{+\infty})_R$ the oriented axis of $g$ in $R$, which determines a pair $(g^{-\infty},g^{+\infty})\in\partial^2(G,\calf)$. An element $(\alpha,\omega)\in\partial^2(G,\calf)$ is \emph{simple} if it is the limit of a sequence $(g_n^{-\infty},g_n^{+\infty})$, where for every $n\in\mathbb{N}$, the element $g_n$ is simple, i.e.\ 
contained in 
a proper free factor of $(G,\calf)$.

We let $i:\partial^2(G,\calf)\to\partial^2(G,\calf)$ be the flip map, defined by $i(x,y)=(y,x)$. An \emph{algebraic lamination} is a closed $G$-invariant $i$-invariant subset of $\partial^2 (G,\calf)$. 

Given a tree $T\in\overline{\calo}$ and $\epsilon>0$, we let $\Lambda^2_\epsilon(T)$ be the closure in $\partial^2(G,\calf)$ of the set $\{(g^{-\infty},g^{+\infty})\mid \Vert g\Vert_T<\epsilon\}.$ The \emph{dual lamination} of $T$ is then defined as $$\Lambda^2(T):=\bigcap_{\epsilon>0} \Lambda^2_\epsilon(T).$$ The \emph{one-sided dual lamination} $\Lambda^1(T)$ is defined to be the set of all $\xi\in\partial_\infty(G,\calf)$ such that for every Grushko tree $R$, every $G$-equivariant Lipschitz map $f:R\to T$, and every basepoint $x_0\in R$, the image $f([x_0,\xi)_R)$ has bounded diameter in $T$ (this property does not actually depend on the choices of $R$, $f$ and $x_0$).

The following statement gives a unique duality property for arational trees. An analogous statement for free groups, which relies on work of Coulbois, Hilion and Reynolds \cite{CHR}, was a crucial ingredient in the description by Bestvina and Reynolds of the boundary of the free factor graph \cite{BR}. The version below for free products was also a key ingredient when extending their description to this setting \cite{GH}. For us, it will be essential in the construction of the barycenter map in Section~\ref{sec-barycenter}.

\begin{theo}[{\cite[Theorem~13.1]{GH1}}]\label{gh}
Let $T,T'\in\overline{\calo}$, with $T$ arational.
\\ Assume that there exists a simple pair $(\alpha,\omega)\in\partial^2(G,\calf)$ which belongs to $\Lambda^2(T)\cap \Lambda^2(T')$. 
\\ Then $T'$ is arational and $T'\sim T$. 
\end{theo}

\section{Amenability of the action on the boundary of the free factor graph}\label{sec:amenability}

In the present section, we will show that, under a condition on centralizers of $\calf$-peripheral elements, the $\Out(G,\calf^{(t)})$-action on the set of equivalence classes of arational $(G,\calf)$-trees (equivalently, on the boundary of the associated free factor graph) is Borel amenable (Proposition~\ref{prop:action-amenable} below). This will be derived as a consequence of the following theorem, established by Bestvina and the first two named authors in an earlier work -- the only remaining task for us being to pass from arational trees to equivalence classes of arational trees.

\begin{theo}[{\cite[Theorem~6.4]{BGH}}]\label{bgh}
Let $G_1,\dots,G_k$ be countable groups, let $F_N$ be a finitely generated free group, and let $$G:=G_1\ast\dots\ast G_k\ast F_N.$$ Let $\calf$ be the collection of all $G$-conjugacy classes of the groups $G_i$. Assume that for each $i\in\{1,\dots,k\}$, the centralizer in $G_i$ of every element of $G_i$ is amenable.

Then the $\Out(G,\calf^{(t)})$-action on $\mathbb{P}\AT$ is Borel amenable.   
\end{theo}

We now aim to deduce the following consequence regarding the amenability of the action on the quotient $\mathbb{P}\AT/{\sim}$.

\begin{prop}\label{prop:action-amenable}
Let $G_1,\dots,G_k$ be countable groups, let $F_N$ be a finitely generated free group, and let $$G:=G_1\ast\dots\ast G_k\ast F_N.$$ Let $\calf$ be the collection of all $G$-conjugacy classes of the groups $G_i$. Assume that for each $i\in\{1,\dots,k\}$, the centralizer in $G_i$ of every element of $G_i$ is amenable.

Then the $\Out(G,\calf^{(t)})$-action on $\mathbb{P}\AT/{\sim}$ is Borel amenable.
\end{prop}

The set of all projective classes of arational trees that are equivalent to a given arational tree $T$ is a finite-dimensional simplex, see \cite[Corollary~5.4]{Gui00} or \cite[Proposition~13.5]{GH1}. Trees in $\mathcal{AT}$ that project in $\PAT$ to an extremal point of the simplex of an arational tree are called \emph{ergometric}. 

\begin{lemma}\label{lemma:ergometric}
Let $T\in\AT$. Then $T$ is ergometric if and only if for every tree $T'\in\baro$ that is not homothetic to $T$, there is no Lipschitz $G$-equivariant map from $T$ to $T'$.
\end{lemma}

\begin{proof}
 Since $T\in\AT$, it has dense orbits, so every Lipschitz $G$-equivariant map $T\to T'$ is alignment-preserving, see e.g.\ \cite[Lemma~2.2]{BGH}. 
By \cite[Section~5.1]{Gui00} and the paragraph before Lemma 5.1 therein, 
trees $T'$ with a Lipschitz equivariant alignment preserving map $T\ra T'$ are in one-to-one correspondence with length measures on $T$ with bounded density with respect to Lebesgue measure.
All such measures are homothetic to the Lebesgue measure if and only $T$ is ergometric
by the second definition of \cite[p. 459]{Gui00}.
\end{proof}

We denote by $\mathbb{P}\AT^{\erg}$ the subset of $\mathbb{P}\AT$ made of all projective classes of ergometric arational trees. 

\begin{lemma}\label{lemma:erg-borel}
The set $\PAT^{\erg}$ is a Borel subset of $\PAT$.
\end{lemma}

\begin{proof}
We first choose a (non-equivariant) continuous section $\mathbb{P}\baro\to\baro$, constructed as follows: using Serre's lemma, choose a finite subset $\{a_1,\dots,a_n\}\subset G$ such that for any $T\in \baro$, at least one of the elements $a_i$ is hyperbolic in $T$; then send a projective class $[T]$ to the unique representative such that $||a_1||_T+\dots+||a_n||_T=1$. This allows us to identify $\mathbb{P}\baro$ with a closed subset of $\baro$. We let $\mathbb{P}\baro^{(2)}:=(\mathbb{P}\baro\times\mathbb{P}\baro)\setminus\Delta$, where $\Delta$ denotes the diagonal. 

 Since $\mathbb{P}\baro$ is a compact metrizable space,
there exists an increasing sequence of compact subsets $K_n$ of $\mathbb{P}\baro^{(2)}$ that exhausts $\mathbb{P}\baro^{(2)}$.
 
For every $n\in\mathbb{N}$, let $\Lip_n$ be the set of all pairs $(T,T')\in K_n$ such that there exists an $n$-Lipschitz $G$-equivariant map from $T$ to $T'$ (this makes sense using the above identification of $\mathbb{P}\baro$ as a subset of $\baro$). The set $\Lip_n$ is a closed subset of $K_n$ (see e.g.\ \cite[Proposition~5.5]{Gui00}), and therefore it is compact. We denote by $p_1(\Lip_n)$ the first projection of $\Lip_n$, which is a closed (whence Borel) subset of $\mathbb{P}\baro$.

Notice that the Borel set $\Pbaro\setminus\bigcup_{n\in\mathbb{N}}p_1(\Lip_n)$ is equal to the set of all trees $T\in\Pbaro$ such that for every $T'\in\Pbaro$ distinct from $T$, there is no Lipschitz $G$-equivariant map from $T$ to $T'$. By Lemma~\ref{lemma:ergometric}, the set $\PAT^{\erg}$ is equal to the intersection of this set with $\PAT$, so it is Borel. 
\end{proof}

\begin{lemma}\label{lemma:borel-selection-ue}
There exists a sequence of Borel maps $f_n:\PAT\to\PAT^{\erg}$ such that for every $T\in\PAT$, the set $\{f_n(T)\}_{n\in\mathbb{N}}$ is equal to the (finite) set of all projective classes of ergometric trees that are equivalent to $T$.
\end{lemma}

\begin{proof}
  We identify $\Pbaro$ to a Borel subset of $\baro$ by choosing a continuous section. Notice that the subset $\Comp\subseteq\Pbaro\times\mathbb{P}\baro$ made of pairs of trees that are compatible is 
  closed \cite[Corollary A.12]{GL-jsj}. Using Lemma~\ref{lemma:erg-borel}, we deduce that $\Comp\cap (\PAT\times\PAT^{\erg})$ is a Borel subset of $\Pbaro^{2}$, and it is made of all pairs $(T,T')\in\PAT\times\PAT^{\erg}$ with $T'\sim T$. The conclusion then follows by using a measurable selection theorem \cite{Sri}.
  More precisely, letting $C:=\{0,1\}^{\mathbb{N}}$ be the Cantor space, it follows from \cite[Theorem~1.2]{Sri} that there exists a map $f:\PAT\times (\mathbb{N}\times C)\to\PAT^{\erg}$ such that for every $T\in\PAT$, the map $f(T,\cdot)$ is continuous and its image equals the set of all ergometric trees that are equivalent to $T$, and for every $(n,c)\in\mathbb{N}\times C$, the map $f_{n,c}:=f(\cdot,(n,c))$ is Borel. Let $(n_i,c_i)_{i\in\mathbb{N}}$ be a dense sequence in $\mathbb{N}\times C$. Then the maps $f_{n_i,c_i}$ satisfy the required conclusion.
\end{proof}

\begin{proof}[Proof of Proposition~\ref{prop:action-amenable}]
By Theorem~\ref{bgh}, there exists a sequence of Borel maps $$\nu_n:\PAT\to\Prob(\Out(G,\calf^{(t)}))$$ which is asymptotically $\Out(G,\calf^{(t)})$-equivariant. We will explain how to modify the maps $\nu_n$ so that they become constant on every $\sim$-class.

Let $f_n:\PAT\to\PAT^{\erg}$ be a sequence of Borel maps provided by Lemma~\ref{lemma:borel-selection-ue}. For all $T\in\PAT$, we let $k(T)$ be the number of ergometric trees that are $\sim$-equivalent to $T$, and we let $n_1(T),\dots,n_{k(T)}(T)$ be integers such that the trees $f_{n_i(T)}(T)$ are pairwise distinct, and $\{f_{n_i(T)}(T)\}$ is the set of all ergometric trees that are $\sim$-equivalent to $T$, with $(n_1(T),\dots,n_{k(T)}(T))$ minimal for the lexicographic order among tuples satisfying the above properties. Note that the maps $T\mapsto k(T)$ and $T\mapsto n_i(T)$ are Borel. We then define a sequence of Borel maps $$\tilde{\nu}_n:\PAT\to\Prob(\Out(G,\calf^{(t)}))$$ by letting $$\tilde{\nu}_n(T):=\frac{1}{k(T)}\sum_{i=1}^{k(T)}\nu_n(f_{n_i(T)}(T)).$$ Then the sequence $\tilde{\nu}_n$ is asymptotically $\Out(G,\calf^{(t)})$-equivariant, and the maps $\tilde{\nu}_n$ are constant on every $\sim$-class. Therefore the $\tilde{\nu}_n$ induce a sequence of Borel maps $\PAT/{\sim}\to\Prob(\Out(G,\calf^{(t)}))$ which is asymptotically $\Out(G,\calf^{(t)})$-equivariant. This completes our proof.
\end{proof}

\section{A barycenter map for $\Out(G,\calf)$}\label{sec-barycenter}

Let $G_1,\dots,G_k$ be countable groups, let $F_N$ be a free group of rank $N$, and let $G:=G_1\ast\dots\ast G_k\ast F_N$. We denote by $\calf$ the finite collection made of the conjugacy classes in $G$ of the subgroups $G_i$. We assume throughout the section that $(G,\calf)$ is not sporadic. We recall that $\FS$ denotes the countable set of all free splittings of $(G,\calf)$, and that $\calp_{<\infty}(\FS)$ denotes the (countable) set of nonempty finite subsets of $\FS$.
The goal of the present section is to prove the following theorem.

\begin{theo}\label{barycenter}
There exists a Borel $\Out(G,\calf)$-equivariant map $$\bary:(\mathbb{P}\AT/{\sim})^{(3)}\to\calp_{<\infty}(\FS).$$
\end{theo} 
 
\begin{rk}
  The main point in the statement is the finiteness of the subset of $\FS$ associated to every triple of pairwise distinct equivalence classes of arational trees. 
By construction, one can check that there is a uniform bound of the diameter of this subset in the free factor graph $\FF$.
\end{rk}

Before proving Theorem~\ref{barycenter}, let us mention a consequence, which is the form in which it will be used in the sequel of the paper. 

\begin{cor}\label{cor:barycenter}
There exists a Borel $\Out(G,\calf)$-equivariant map which associates a nonempty finite set of proper $(G,\calf)$-free factors to every triple in $(\mathbb{P}\AT/{\sim})^{(3)}$.
\end{cor}

\begin{proof}
There is an $\Out(G,\calf)$-equivariant Borel map which associates to any free splitting $S$ a nonempty finite set $F(S)$ of conjugacy classes of proper $(G,\calf)$-free factors, namely $F(S)$ is the set of all conjugacy classes of proper $(G,\calf)$-free factors that are elliptic in some collapse of $S$. Corollary~\ref{cor:barycenter} is therefore a consequence of Theorem~\ref{barycenter}.
\end{proof}

\subsection{Unique duality for arational trees}

Throughout the section, we fix a Grushko tree $R$, and given $g\in G$, we write $|g|:=||g||_{R}$. The goal of the present section is to strengthen the following unique duality statement for arational trees obtained in \cite{CHR,GH1} (it essentially follows from \cite[Proposition~4.26 and Theorem~13.1]{GH1}): our stronger statement is Lemma~\ref{gh1} below.

Recall that an element $g\in G$ is \emph{simple} if it is contained in some proper $(G,\calf)$-free factor (equivalently, $g$ is elliptic
in some free splitting of $(G,\calf)$).

\begin{lemma}[{\cite[Proposition~4.26 and Theorem~13.1]{GH1}}]\label{lemma:unique-duality}
Let $T,T'\in\overline{\calo}$, with $T$ arational. Let $(T_n)_{n\in\mathbb{N}},(T'_n)_{n\in\mathbb{N}}\in\calo^{\mathbb{N}}$ be sequences that converge (non-projectively) to $T,T'$ respectively.

Assume that there exists a sequence of nonperipheral simple elements $(g_n)_{n\in\mathbb{N}}\in G^\mathbb{N}$ such that $||g_n||_{T_n}\to 0$ and $||g_n||_{T'_n}\to 0$. 

Then $T'$ is arational and $T'\sim T$.
\end{lemma}

We recall that given a $G$-equivariant map $f:R\to T$ between two $(G,\calf)$-trees, the \emph{bounded backtracking constant} $\BBT(f)$ is defined as the smallest
 real number $K$ such that for all $x,y,z\in R$ aligned in this order, we have $d_{T}(f(y),[f(x),f(z)])\le K$. Notice in particular that if $g$ is a nonperipheral element, then for every point $x$ that belongs to the axis of $g$ in $R$, we have $d_T(f(x),\Char_T(g))\le K$ (where $\Char_T(g)$ denotes the characteristic set of $g$ in $T$, i.e.\ either its axis if $g$ acts hyperbolically on $T$, or else its fix point set).

\begin{lemma}\label{lemma:long-segments}
  Let $(S_n)_{n\in\mathbb{N}}\in\mathcal{O}^\mathbb{N}$, let $(g_n)_{n\in\mathbb{N}}\in G^\mathbb{N}$ be a sequence of nonperipheral elements, and let
  $B>0$.
  Assume that
\begin{enumerate}
\item for every $n\in\mathbb{N}$, there exists a $G$-equivariant map $f_n:R\to S_n$ with $\BBT(f_n)\le B$, and
\item we have $$\frac{||g_n||_{S_n}}{|g_n|}\to 0.$$
\end{enumerate}
Then up to passing to a subsequence, there exist segments $I_n\subseteq (g_n^{-\infty},g_n^{+\infty})_R$, with $|I_n|\to +\infty$, such that $\diam_{S_n}(f_n(I_n))\le 5B$. 
\end{lemma}

\begin{proof}
For every $n\in\mathbb{N}$, one has $\Axis_{S_n}(g_n)\subseteq f_n(\Axis_R(g_n))$. We thus choose $a_n\in\Axis_{R}(g_n)$ with $f_n(a_n)\in \Axis_{S_n}(g_n)$.
  
 First assume that for all $n\geq 0$, one has $||g_n||_{S_n}\leq B$.
  Let   $m_n=\lfloor\frac{B}{||g_n||_{S_n}}\rfloor$, let $b_n=g_n^{m_n}a_n$, and let
    $I_n=[a_n,b_n]\subseteq R$. Using the second assumption from the lemma, we see that $|I_n|=m_n |g_n|$ tends to $+\infty$.
On the other hand $d(f_n(a_n),f_n(b_n))=m_n ||g_n||_{S_n}\leq B$.
Since $\BBT(f_n)\le B$, the image $f_n(I_n)$ is contained in the $B$-neighborhood of $[f_n(a_n),f_n(b_n)]$
so $\diam(f_n(I_n))\le 3B$.

Up to taking a subsequence, we can therefore assume that for all $n\geq 0$, one has $||g_n||_{S_n}>B$.
Consider the fundamental domain $J_n=[a_n,c_n]$ of $\Axis_R(g_n)$ where $c_n=g_na_n$. The segment $K_n=[f_n(a_n),f_n(c_n)]$
is a fundamental domain of $\Axis_{S_n}(g_n)$.
Let $N=\lceil\frac{||g_n||_{S_n}}{B}\rceil$, and subdivide the segment $[a_n,c_n]$ into $t_0=a_n,t_1,\dots,t_N=c_n$ with
$d(t_i,t_{i+1})=\frac{|g_n|}{N}$ for all $i<N$.
We claim that we can take $I_n$ to be one of the intervals $[t_i,t_{i+1}]$.
Note that their length is at least $B\frac{N-1}{N}\frac{|g_n|}{||g_n||_{S_n}}$ which tends to infinity with $n$ by assumption.

For $i<N$, let $q_i=f_n(t_i)$ and $p_i$ be the projection of $q_i$ on $K_n$.
Since $\BBT(f_n)\le B$, one has $d(p_i,q_i)\leq B$.
Using again that $\BBT(f_n)\le B$, we see that $\diam(f_n([t_i,t_{i+1}]))\leq 2B+d(q_i,q_{i+1})$,
hence it suffices to find $i<N$ such that $d(q_i,q_{i+1})\leq 3B$.

We consider the natural ordering on $K_n=[p_0,p_N]$ such that $p_0<p_N$.
First assume that there exists $i<N$ such that $p_{i+1}<p_i$ (i.e.\ there is some backtracking among the $p_i$). As $\BBT(f_n)\le B$, this implies that $B\geq d(q_i,[p_0, q_{i+1}])=d(q_i,p_{i+1})\geq d(q_i,q_{i+1})-B$
and we are done.

We can therefore assume that $p_0\leq p_1\leq \dots \leq p_N$. Let $m=\min\{d(p_{i},p_{i+1})|i<N\}$. It suffices to prove that $m\leq B$.
If not, then
$$||g_n||_{S_n}=d(p_0,p_n)\geq Nm > \lceil\frac{||g_n||_{S_n}}{B}\rceil B\geq ||g_n||_{S_n},$$
a contradiction.
\end{proof}

\begin{lemma}\label{gh1}
Let $T,T'\in\overline{\calo}$, with $T$ arational. Let $(T_n)_{n\in\mathbb{N}},(T'_n)_{n\in\mathbb{N}}\in\calo^{\mathbb{N}}$ be sequences that converge (non-projectively) to $T,T'$ respectively.

Assume that there exists a sequence of nonperipheral simple elements $(g_n)_{n\in\mathbb{N}}\in G^\mathbb{N}$ such that $$||g_n||_{T_n}\to 0 \text{~~and~~~} \frac{||g_n||_{T'_n}}{|g_n|}\to 0.$$

Then $T'$ is arational and $T'\sim T$. 
\end{lemma}

\begin{proof}
We denote by $C$ the covolume of our fixed Grushko tree $R$, i.e.\ the sum of the lengths of all edges of the quotient graph $R/G$. Let $f':R\to T'$ be a $G$-equivariant Lipschitz map. Since $(T'_n)_{n\in\mathbb{N}}$ converges non-projectively to $T'$, we can find $G$-equivariant Lipschitz maps $f'_n:R\to T'_n$ converging to $f'$ in the Gromov--Hausdorff equivariant topology on the set of morphisms introduced in \cite{GL}. In particular, there exists $K\in\mathbb{N}$ such that for all $n\in\mathbb{N}$, the map $f'_n$ is $K$-Lipschitz. By \cite[Lemma~4.1]{BFH} (see \cite[Proposition~3.12]{Hor3} in the context of free products), the map $f'_n$ has BBT at most $KC$.

We claim that up to passing to a subsequence, there exists a sequence $(h_n)_{n\in\mathbb{N}}\in G^\mathbb{N}$ such that $h_n.(g_n^{-\infty},g_n^{+\infty})$ converges to $(\alpha,\omega)\in\partial^2(G,\calf)$, with $(\alpha,\omega)\in (\Lambda^1(T')\cup V_\infty(G,\calf))^2$. 

We now prove our claim. Applying Lemma~\ref{lemma:long-segments} to the trees $T'_n$ and the maps $f'_n$ whose bounded backtracking constant is bounded from above by $KC$, we obtain subsegments $I_n\subseteq (g_n^{-\infty},g_n^{+\infty})_R$ such that $|I_n|\to +\infty$ and $\diam_{T'_n}(f'_n(I_n))\le 5KC$. Let $e_n=[u_n,v_n]$ be an edge in $I_n$ such that both connected components of $I_n\setminus e_n$ have length diverging to $+\infty$ as $n$ goes to $+\infty$. Since the $G$-action on $R$ is cocompact, up to passing to a subsequence, we can assume that all edges $e_n$ lie in the $G$-orbit of a common edge $e=[u,v]\subseteq R$, and choose $h_n\in G$ such that $h_n.e_n=e$. The complementary components of the midpoint of $e$ determine two directions in $R\cup\partial R$, one containing $u$ (denoted by $d^-$) and one containing $v$ (denoted by $d^+$). For every $n\in\mathbb{N}$, we have $h_n\cdot g_n^{-\infty}\in d^-$ and $h_n\cdot g_n^{+\infty}\in d^+$. By compactness of $\partial (G,\calf)$, up to passing to a further subsequence, we can assume that $h_n.(g_n^{-\infty},g_n^{+\infty})$ converges to $(\alpha,\omega)$, and $\alpha\neq\omega$ because they are separated by the edge $e$ (indeed, by definition of the observers' topology, we have $\alpha\in d^-$ and $\omega\in d^+$). We are left proving that $\alpha,\omega\in \Lambda^1(T')\cup V_\infty(G,\calf)$. We prove it for $\alpha$, by symmetry the result also holds for $\omega$. Assume that $\alpha\notin V_\infty(G,\calf)$; we aim to show that $f'([u,\alpha)_R)$ has bounded diameter in $T'$. Let $[u,u']_R$ be an initial segment of $[u,\alpha)_R$. Then there exists $n\in\mathbb{N}$ such that $[u,u']_R\subseteq h_nI_n$. Therefore $f'_n([u,u']_R)$ has diameter at most $5KC$. Since $f'_n$ converges to $f'$, this implies that the diameter of $f'([u,u']_R)$ is uniformly bounded (with a bound not depending on the point $u'$). Therefore $f'([u,\alpha)_R)$ is bounded, which concludes the proof of our claim.

We now finish the proof of the lemma. Since $||g_n||_{T_n}$ converges to $0$ and $T_n$ converges to the tree $T$ which is arational (and therefore has dense $G$-orbits), it follows from \cite[Proposition~4.26]{GH1} that $(\alpha,\omega)\in \Lambda^2(T)$. 
We claim that $(\alpha,\omega)\notin V_\infty(G,\calf)^2$.
Indeed, assume by contradiction that  $\alpha$ and $\omega$ both lie in $V_\infty(G,\calf)$ ; then by Lemma \cite[Lemma 6.4]{GH1}, 
the group $\grp{G_\alpha,G_\omega}$ is elliptic in $T$, contradicting
that point stabilizers in the arational tree $T$ are cyclic or peripheral (see \cite[Section~4]{Hor}).

So either $\alpha$ or $\omega$ (say $\alpha$) belongs to $\Lambda^1(T)\cap \Lambda^1(T')$. By \cite[Proposition~4.19]{GH1}, we thus have $\Lambda^2(\alpha)\subseteq \Lambda^2(T)\cap \Lambda^2(T')$, where $\Lambda^2(\alpha)$ is the \emph{limit set} defined in \cite[Definition~4.9]{GH1}, i.e.\ $(\alpha',\omega')\in\Lambda^2(\alpha)$ if up to exchanging the roles of $\alpha'$ and $\omega'$, there exists a sequence $(k_n)_{n\in\mathbb{N}}\in G^{\mathbb{N}}$ converging to $\omega'$ such that $k_n\cdot \alpha$ converges to $\alpha'$. This limit set is nonempty \cite[Lemma~4.11]{GH1}, and it consists entirely of simple pairs \cite[Lemma~5.6]{GH1}. We thus have found a simple pair in $\Lambda^2(T)\cap \Lambda^2(T')$. Theorem~\ref{gh} thus implies that $T'$ is arational, and $T'\sim T$. 
\end{proof}

\subsection{Pencils between two arational trees}

The following lemma generalizes a theorem of Bestvina--Reynolds \cite[Theorem~6.6]{BR} to the context of free products.

\begin{lemma}\label{u-turn}
  Let $S,T\in\AT$. Let $(S_n)_{n\in\mathbb{N}}\in\calo^{\mathbb{N}}$ be a sequence which converges projectively to 
  $S$, and let $(T_n)_{n\in\mathbb{N}}\in\calo^{\mathbb{N}}$ be a sequence that converges non-projectively to $T$.

Assume that for every $n\in\mathbb{N}$, there exists an optimal morphism $f_n:S_n\to T_n$. Let $U\in\overline{\calo}$, and assume that there exists a sequence $(U_n)_{n\in\mathbb{N}}\in\calo^{\mathbb{N}}$ converging projectively to $U$, such that for every $n\in\mathbb{N}$, the morphism $f_n$ factors through optimal morphisms $S_n\to U_n$ and $U_n\to T_n$. 

Then either $U$ is equivalent to $S$, or $U$ is equivalent to $T$, or $U$ belongs to $\calo$ (as opposed to $\partial\calo$). If $U\in\calo$, then $(U_n)_{n\in\mathbb{N}}$ accumulates non-projectively to $\lambda U$ for some $\lambda\in\mathbb{R}_+^\ast$.
\end{lemma}

\begin{proof}
We can find sequences $(|S_n|),(|U_n|)\in\mathbb{R}^{\mathbb{N}}$  such that $\Tilde S_n:=S_n/|S_n|$ converges to $S$, and $\Tilde U_n:=U_n/|U_n|$ converges to $U$. Also $(T_n)_{n\in\mathbb{N}}$ converges to $T$ by assumption, without renormalizing. 

For all $n\in\mathbb{N}$, there are morphisms $S_n\to U_n\to T_n$. If $|U_n|$ is bounded, then there is a $G$-equivariant Lipschitz map $h:U\to T$. By \cite[Lemma~2.3]{BGH}, either $U$ belongs to $\calo$ (in which case we are done) or else $U$ is a tree with dense $G$-orbits. In the latter case, it follows from \cite[Lemma~2.2]{BGH} that $h$ is alignment-preserving, so $U$ is equivalent to $T$. 

We can therefore assume, up to passing to a subsequence, that $|U_n|$ diverges to $+\infty$. 
We aim to prove that $U$ is arational and equivalent to $S$; the lemma will follow from this fact. 

Since $\Tilde S_n$ converges to $S$ which has dense orbits, the covolume $\vol(\tilde S_n)$ of $\Tilde S_n$ tends to $0$.
For every $n\in\mathbb{N}$, as $f_n$ is optimal, we can find a nonperipheral element $g_n\in G$, whose axis is legal (i.e.\ $f_n$ is an isometry when restricted to the axis of $g_n$) and with $||g_n||_{\Tilde S_n}\leq 2 \vol(\tilde S_n)$, so $||g_n||_{\Tilde S_n}$ tends to $0$.
Since $g_n$ is legal, $||g_n||_{S_n}=||g_n||_{U_n}=||g_n||_{T_n}$. 

Since $T_n$ converges non-projectively to $T$, 
there exists a constant $C$ such that for every $n\in\mathbb{N}$, the Lipschitz constant of any optimal piecewise linear map from our fixed Grushko tree $R$ to $T_n$ is at most $C$. 
It follows that $||g_n||_{T_n}\leq C |g_n|$. 

We thus deduce that $$\frac{||g_n||_{\tilde U_n}}{|g_n|}=\frac{||g_n||_{U_n}}{|g_n|.|U_n|}=\frac{||g_n||_{T_n}}{|g_n|.|U_n|}\leq \frac{C}{|U_n|}\ra 0.$$
Since in addition $||g_n||_{\tilde{S}_n}\to 0$  
and $\tilde{S}_n$ converges to  $S$, Lemma~\ref{gh1} implies that $U$ is arational and equivalent to $S$.
\end{proof}

\subsection{Barycenter of a triple of pairwise inequivalent arational trees}\label{subsec_bary}

Given two inequivalent arational trees $S,T\in\AT$ and $U\in\calo$, we say that the triple $(S,U,T)$ \emph{has an aligned approximation} if there exist 
\begin{itemize}
\item a sequence $(S_n)_{n\in\mathbb{N}}\in\calo^\mathbb{N}$ converging projectively to a tree $S'\sim S$, 
\item a sequence $(U_n)_{n\in\mathbb{N}}\in\calo^\mathbb{N}$ converging non-projectively to $U$,
\item a sequence $(T_n)_{n\in\mathbb{N}}\in\calo^\mathbb{N}$ converging non-projectively to $T$, 
\end{itemize}
such that for every $n\in\mathbb{N}$, there exists an optimal morphism $S_n\to T_n$ which factors through optimal morphisms $S_n\to U_n$ and $U_n\to T_n$. 
In this case, we say that $((S_n)_{n\in\mathbb{N}},(U_n)_{n\in\mathbb{N}},(T_n)_{n\in\mathbb{N}})$ is an \emph{aligned approximation} of $(S,U,T)$. 

\begin{rk}\label{rk_proj}
If $(S,U,T)$ has an aligned approximation, then for any $\lambda,\lambda'\in\bbR^*_+$, so does
 $(\lambda' S,\lambda U,\lambda T)$.
\end{rk}

Let now $L\ge 1$ and $C\ge 0$ be such that the projection to $\FF$ of any optimal folding path between Grushko $(G,\calf)$-trees is an (unparametrized) $(L,C)$-quasigeodesic, see e.g.\ \cite{HM} or \cite[Proposition~2.11]{GH}. The hyperbolicity of $\FF$ allows us to choose an integer $\delta$ such that for every triple $(\xi_1,\xi_2,\xi_3)$ of pairwise distinct points of $\partial_\infty\FF$, there exist neighborhoods $V_1,V_2,V_3$ of $\xi_1,\xi_2,\xi_3$ in $\FF\cup\partial_\infty\FF$ with the following property: given any three $(L,C)$-quasigeodesics $\gamma_{1},\gamma_{2},\gamma_{3}$, where $\gamma_{i}$ has its origin in $V_i$ and its extremity in $V_{i+1}$ (with indices taken mod $3$), there exists a triple of points $(x_{1},x_{2},x_{3})$, with $x_{i}$ lying on the image of $\gamma_{i}$, such that for every $i,j\in\{1,2,3\}$, we have $d_{\FF}(x_{i},x_{j})\le\delta$.

Let $\pi:\calo\to\FF$ be the natural map consisting in forgetting the metric on the tree (recall that the vertex set of $\FF$ consists of non-metric free splittings).
Let $\calt$ be the set of all triples $(U_1,U_2,U_3)\in\calo^3$ such that for all $i,j\in\{1,2,3\}$, one has $d_{\FF}(\pi(U_i),\pi(U_j))\le\delta$.

Given a triple $(T_1,T_2,T_3)\in \AT^3$ of pairwise inequivalent 
arational trees, we next define $B(T_1,T_2,T_3)$ as the set of all trees $U_{12}\in\calo$ such that there exist $U_{23},U_{31}\in\calo$, and for every $i\in\{1,2,3\}$ (taken modulo $3$), an aligned approximation $(S_{i,i+1}^n,U_{i,i+1}^n,T_{i,i+1}^n)$ of $(T_i,U_{i,i+1},T_{i+1})$, such that for every $n\in\mathbb{N}$, we have $(U_{12}^n,U_{23}^n,U_{31}^n)\in\calt$.

We note that for all $\lambda_1,\lambda_2,\lambda_3\in\bbR_+^*$, one has $B(\lambda_1 T_1,\lambda_2 T_2,\lambda_3 T_3)=\lambda_2 B(T_1, T_2, T_3)$.

\begin{lemma}\label{lemma:boite-compacte}
For every triple $(T_1,T_2,T_3)\in \AT^3$ of pairwise inequivalent arational trees, the set $B(T_1,T_2,T_3)$ is a non-empty compact subset of $\calo$.
\end{lemma}

\begin{proof}
We first prove that $B(T_1,T_2,T_3)$ is non-empty. For every $i\in\{1,2,3\}$, let $(T_{i}^n)_{n\in\mathbb{N}}\in\calo^{\mathbb{N}}$ be a sequence that converges nonprojectively to $T_i$. For every $i\in\{1,2,3\}$, we let $T_{i,i+1}^n:=T_{i+1}^n$. Consider $S_{i,i+1}^n$ obtained by changing edge lengths on $T_i^n$ (possibly setting some of them to $0$) so that
there exists an optimal morphism $S_{i,i+1}^n\to T_{i+1}^n$ (see Section \ref{sec_morphisms}). Then up to taking a subsequence, $(S_{i,i+1}^n)_{n\in\mathbb{N}}$ converges projectively to a tree $T'_i$ which is arational and equivalent to $T_i$. Let $\xi_1,\xi_2,\xi_3$ be the points in $\partial_\infty\FF$ corresponding to the arational trees $T_1,T_2,T_3$. Let $V_1,V_2,V_3$ be neighborhoods of $\xi_1,\xi_2,\xi_3$ in $\FF\cup\partial_\infty\FF$ coming from the definition of $\delta$. By \cite[Theorem~3]{GH}, for all sufficiently large $n\in\mathbb{N}$ and all $i\in\{1,2,3\}$, the projection $\pi(S_{i,i+1}^n)$ belongs to $V_i$, and $\pi(T_{i,i+1}^n)$ belongs to $V_{i+1}$. For every such $n$ and every $i\in\{1,2,3\}$, we then consider an optimal folding path from $S_{i,i+1}^n$ to $T_{i,i+1}^n$. Since optimal folding paths between Grushko $(G,\calf)$-trees project to unparametrized $(L,C)$-quasigeodesics in $\FF$, it follows from the definition of $\delta$ that we can find a triple $(U_{12}^n,U_{23}^n,U_{31}^n)\in \calo^3$, with $U_{i,i+1}^n$ on the folding path from $S_{i,i+1}^n$ to $T_{i,i+1}^n$, such that $(U_{12}^n,U_{23}^n,U_{31}^n)\in\calt$. As $\mathbb{P}\baro$ is compact, up to passing to a subsequence, we can assume that the trees $U_{12}^n$, $U_{23}^n$ and $U_{31}^n$ converge projectively to trees $U_{12}^\infty,U_{23}^\infty,U_{31}^\infty\in\mathbb{P}\baro$. Since the $\pi$-images of these trees lie in a bounded region of $\FF$, it follows from \cite[Theorem~3]{GH} that the limits $U_{i,i+1}^\infty$ are all non-arational, and in particular not equivalent to $T_1,T_2,T_3$. It then follows from Lemma~\ref{u-turn} that all trees $U_{i,i+1}^\infty$ belong to $\mathbb{P}\calo$ (as opposed to belonging to the boundary), and they have representatives $U_{12},U_{23},U_{31}\in\calo$ such that for every $i\in\{1,2,3\}$, the sequence $(U_{i,i+1}^n)_{n\in\mathbb{N}}$ accumulates nonprojectively to $U_{i,i+1}$. Therefore $U_{12}\in B(T_1,T_2,T_3)$ and $B(T_1,T_2,T_3)\neq \es$. 

We now prove that $B(T_1,T_2,T_3)$ is compact. Let $(U(n))_{n\in\mathbb{N}}$ be a sequence of points in $B(T_1,T_2,T_3)$. 
Since $\mathbb{P}\baro$ is compact, up to passing to a subsequence, we can assume that $(U(n))_{n\in\mathbb{N}}$ converges projectively to a tree $U(\infty)$. 
From the definition of $B(T_1,T_2,T_3)$, a diagonal argument together with Lemma~\ref{u-turn} shows that either $U(n)$  accumulates nonprojectively to a tree homothetic to $U(\infty)$ that lies in $B(T_1,T_2,T_3)$ or else $U(\infty)$ is equivalent to either $T_1$ or $T_2$. But all the approximations of $U(n)$ used in the definition project to a common bounded subset of $\FF$, so $U(\infty)$ is not an arational tree. Therefore $U(n)$ accumulates nonprojectively to a tree in $B(T_1,T_2,T_3)$, showing that $B(T_1,T_2,T_3)$ is a compact subset of $\calo$. 
\end{proof}

We have thus associated a compact set in an equivariant way to any triple of pairwise inequivalent arational trees.
We now explain how to associate a canonical finite set of simplices of $\bbP\calo$ from this compact subset of $\calo$ -- this is not obvious as vertices in Outer space can belong to infinitely many simplices in general. We will use the following fact.

\begin{lemma}\label{lem_lipschitz}
 If $(S_n)_{n\in\mathbb{N}}\in\calo^{\mathbb{N}}$ converges non-projectively to a tree $S\in\calo$, then for every $\varepsilon>0$ and every sufficiently large $n\in\mathbb{N}$, there exists a $(1+\varepsilon)$-Lipschitz $G$-equivariant piecewise linear map $S\to S_n$.
\end{lemma}

\begin{proof}
For $S,T\in \calo$, denote by $\mathrm{Lip}(S,T)$ the infimum of Lipschitz constants of
$G$-equivariant piecewise linear maps $S\ra T$.
By \cite[Theorem~3.7]{Hor3}, for every $S\in\calo$, there is a finite set of elements $g_i\in G$ (called candidates of $S$)
such that for every $T\in\calo$, one has $\mathrm{Lip}(S,T)=\max_{i} \frac{||g_i||_T}{||g_i||_S}$.
It follows that $\mathrm{Lip}(S,S_n)$ converges to $1$ as $n$ goes to infinity.
\end{proof}

Lemma~\ref{lem_lipschitz} implies that the covolume of a tree in $\calo$ varies upper semicontinuously. Therefore, given a nonempty compact subset $K$ of $\calo$, the set $K_{\max}$ of all trees in $K$ with maximal covolume is therefore nonempty and compact. 
We then let $K_{\max,\min}$ be the subset of $K_{\max}$ made of all trees having the minimal possible number of $G$-orbits of edges.

\begin{lemma}
Let $K$ be a compact subset of $\calo$. Then $\pi(K_{\max,\min})$ is finite.
\end{lemma}

In other words, 
 there are only finitely many trees in $K_{\max,\min}$ up to equivariant homeomorphism (i.e.~up to changing edge lengths, keeping them positive).

\begin{proof}
Let $(U_n)_{n\in\mathbb{N}}$ be a sequence of trees in $K_{\max,\min}$; we aim to show that the underlying simplicial trees of the trees $U_n$ take only finitely many values. Since $K$ is compact, up to passing to a subsequence, we can assume that $(U_n)_{n\in\mathbb{N}}$ converges to a tree $U_\infty\in K_{\max}$. We will now prove that for all sufficiently large $n\in\mathbb{N}$, 
$U_n$ is equivariantly homeomorphic to $U_\infty$. 

To prove this fact, let $\epsilon>0$ be very small compared to the shortest edge of $U_\infty$. By Lemma~\ref{lem_lipschitz}, for all $n$ large enough, there exists a $(1+\epsilon)$-Lipschitz piecewise linear $G$-equivariant map $f_n:U_\infty\ra U_n$. This map folds several edges together, and we can assume that $\epsilon$ has been chosen small enough so that every edge is only partially folded by this map. We observe that for every vertex $v$ in $U_\infty$, either $f_n$ is injective on a small neighborhood of $v$, or else all edges based at $v$ but one are folded together: otherwise, folding would create at least one new orbit of edges, which would imply that $U_n$ has strictly more orbits of edges than $U_\infty$, a contradiction. This implies that $U_n$ and $U_\infty$ are equivariantly homeomorphic. 
\end{proof}

Since $B(\lambda_1 T_1,\lambda_2 T_2,\lambda_3 T_3)=\lambda_2 B(T_1, T_2, T_3)$ for all $\lambda_1,\lambda_2,\lambda_3\in\bbR_+^*$, 
the finite set $\pi(B(T_1,T_2,T_2)_{\max,\min})$ only depends on the projective classes of $T_1$, $T_2$, and $T_3$.
The map $(T_1,T_2,T_3)\mapsto \pi(B(T_1,T_2,T_2)_{\max,\min})$ thus induces an $\Out(G,\calf)$-equivariant map assigning to every triple of pairwise inequivalent projective arational trees, a nonempty finite set of free splittings of $(G,\calf)$. 

In the next section, we will prove that this map is Borel. 

\subsection{Measurability considerations}
 
We denote by $\AT^{((3))}$ the space of triples of pairwise inequivalent arational trees, and by $\mathbb{P}\AT^{((3))}$ the space of triples of pairwise inequivalent projective arational trees. The goal of the present section is to prove the following proposition.

\begin{prop}\label{prop_borel}
The map 
\begin{displaymath}
\begin{array}{rcl}
 \AT^{((3))} & \to & \calp_{<\infty}(\FS)\\
(T_1,T_2,T_3) & \mapsto & \pi(B(T_1,T_2,T_3)_{\max,\min})
\end{array}
\end{displaymath}
is Borel, and therefore induces a Borel map $\PAT^{((3))}\to\calp_{<\infty}(\FS)$.
\end{prop}

We start with a technical lemma.
We denote by $\calo_{\bbQ}$ the subset of $\calo$ made of all trees whose edge lengths are all rational.
Aligned approximations were defined at the beginning of Section~\ref{subsec_bary}.

\begin{lemma}\label{lemma:rationel}
  If $(S,U,T)$ has an aligned approximation, then it has an aligned approximation $(S'_n,U'_n,T'_n)_{n\in \bbN}$ 
with $S'_n,U'_n,T'_n\in \calo_\bbQ$.
\end{lemma}

\begin{proof}
  Consider an aligned approximation of $(S,U,T)$ given by $S_n\xra{f_n} U_n\xra{g_n}T_n$   with $f_n,g_n$ and $g_n\circ f_n$ optimal morphisms.
Up to subdividing $U_n$ and $T_n$, one may assume that $f_n$ and $g_n$ map vertices to vertices.
Now let $T'_n\in \calo_\bbQ$ be obtained by approximating the edge lengths of $T_n$ by rational numbers.
Consider $U'_n$ obtained from $U_n$ by changing the metric on the edges so that the map $g_n:U'_n\ra T'_n$ is
isometric on edges. Construct $S'_n$ similarly so that $f_n:S'_n\ra U'_n$ is isometric on edges.
Then $S'_n,U'_n,T'_n$ are in $\calo_\bbQ$. Since one can choose the rational approximations so that $S'_n$, $U'_n$, and $T'_n$ converge to the same limits as $S_n$, $U_n$, and $T_n$, the lemma follows.
\end{proof}

\begin{proof}[Proof of Proposition \ref{prop_borel}]
Given a subset $E\subseteq \calo$, let $B_E$ be the set of all triples $(T_1,T_2,T_3)$ of pairwise inequivalent arational trees such that $B(T_1,T_2,T_3)$ intersects $E$. We will first prove that if $E$ is compact, then $B_E$ is a Borel subset of $\AT^{((3))}$.

Let $d_{\calo}, d_{\baro}, d_{\mathbb{P}\baro}$ be metrics that induce the Gromov--Hausdorff equivariant topology on the corresponding spaces. 
Given $\eps>0$, we define $\tilde{B}_{E,\epsilon}$ to be the subset of $\AT^{((3))}\times \calo_\bbQ^9$ made of all tuples $$(T_1,T_2,T_3,S_{12},S_{23},S_{31},U_{12},U_{23},U_{31},T_{12},T_{23},T_{31})$$ such that for all $i,j\in\{1,2,3\}$,  denoting by $[T_i]$ the set of all projective classes of arational trees that are equivalent to $T_i$, the following hold true:
\begin{itemize}
\item $d_{\bbP\baro}(S_{i,i+1},[T_i])<\eps$, 
\item $d_{\baro}(T_{i,i+1},T_{i+1})<\eps$,
\item there exists an optimal morphism $S_{i,i+1}\to T_{i,i+1}$ which factors through optimal morphisms $S_{i,i+1}\to U_{i,i+1}$ and $U_{i,i+1}\to T_{i,i+1}$, 
\item $(U_{12},U_{23},U_{13})\in\calt$,
\item $d_\calo(U_{12},E)<\epsilon$.
\end{itemize}
We then let $\pi:\AT^{((3))}\times\calo_\bbQ^9\to \AT^{((3))}$ be the projection to the first three coordinates. Since $\calo_\bbQ$ is countable, the $\pi$-image of any Borel subset is again Borel. Since, in view of Lemma~\ref{lemma:rationel}, we have $$B_E=\bigcap_{n>0}\pi(\tilde B_{E,\frac{1}{n}}),$$ we deduce that $B_E$ is a Borel subset of $\calo$.  
For all integers $i,q$ with $q\neq 0$, define $E^{(i,q)}$ as the subset of $E$ consisting of trees of covolume in the interval $[\frac{i}{q},\frac{i+1}{q}]$.
Let $$\vol_{\max}^E (T_1,T_2,T_3)=\sup_{q>0}\, \sup_{i\geq 0} \, \frac{i}{q}\bbun_{B_{E^{(i,q)}}}(T_1,T_2,T_3)$$
be the maximal 
covolume of a tree in $B(T_1,T_2,T_3)\cap E$. 
Write $\calo$ as a countable union of compact sets $E_n$. 
Then $$\vol_{\max} (T_1,T_2,T_3)=\sup_{n} \vol_{\max}^{E_n}(T_1,T_2,T_3)$$
is the maximal covolume of all trees in $B(T_1,T_2,T_3)$. 
Thus $\vol_{\max}^E (T_1,T_2,T_3)$ and $\vol_{\max} (T_1,T_2,T_3)$ are two Borel functions of $(T_1,T_2,T_3)$.

The Borel set
$$B^{E}_{\max}=\{(T_1,T_2,T_3)\in B_E | \vol_{\max}^E (T_1,T_2,T_3)= \vol_{\max} (T_1,T_2,T_3)\}$$
is the set of all triples $(T_1,T_2,T_3)$ such that $B(T_1,T_2,T_3)\cap E$ contains a tree 
of maximal covolume in $B(T_1,T_2,T_3)$.

Given a standard open cone $\sigma$ of $\calo$ (i.e.\ the set of all trees in $\calo$ obtained from a given tree $T_\sigma$ by varying edge lengths, keeping them positive), write it as a countable union of compact sets $E_{\sigma,n}$. The Borel set
$$B^{\sigma}_{\max}=\cup_n B^{E_{\sigma,n}}_{\max}$$ 
is the set of all triples $(T_1,T_2,T_3)$ such that $B(T_1,T_2,T_3)\cap \sigma$ contains a tree 
of maximal covolume in $B(T_1,T_2,T_3)$.

Finally, we define $a(\sigma)$ to be the number of orbits of edges of any tree in $\sigma$.
Then $\sigma$ intersects $B(T_1,T_2,T_3)_{\max,\min}$
if and only if $(T_1,T_2,T_3)$ lies in the set
$$B^\sigma_{\max,\min}=B^\sigma_{\max}\setminus \bigcup_{a(\sigma')>a(\sigma)} B^{\sigma'}_{\max}.$$
Since this set is Borel, this concludes the proof.
\end{proof}

\subsection{Conclusion}

\begin{proof}[Proof of Theorem~\ref{barycenter}]
By Proposition~\ref{prop_borel}, the map $$\beta:(T_1,T_2,T_3)\mapsto \pi(B(T_1,T_2,T_3)_{\max,\min})$$ can be viewed as a Borel $\Out(G,\calf)$-equivariant map from $\PAT^{((3))}$ to $\calp_{<\infty}(\FS)$. 
But this map might fail to pass to the quotient under equivalence of arational trees.
To solve this problem, for every projective arational tree $T\in\PAT$, define $[T]_{\erg}$ as the finite set of all projective classes of ergometric trees that are equivalent to $T$. For all $(T_1,T_2,T_3)\in\PAT^{(3)}$, we then define
 $$\widetilde\bary(T_1,T_2,T_3)= \bigcup_{T'_1\in [T_1]_\erg,\ T'_2\in [T_2]_\erg,\ T'_3\in [T_3]_\erg} \beta(T'_1,T'_2,T'_3),$$ which is a nonempty finite subset of $\FS$.
To prove that this map is Borel, consider the sequence of Borel maps $f_n:\PAT\ra \PAT^\erg$ given in Lemma~\ref{lemma:borel-selection-ue}. We then note that $$\widetilde\bary(T_1,T_2,T_3)=  \bigcup_{n_1,n_2,n_3} \beta(f_{n_1}(T_1),f_{n_2}(T_2),f_{n_3}(T_3)),$$ showing that $\widetilde\bary:\mathbb{P}\AT^{((3))}\to\calp_{<\infty}(\FS)$ is Borel. The map $\widetilde\bary$ then descends to 
a Borel $\Out(G,\calf)$-equivariant map 
$$\bary:(\mathbb{P}\AT/{\sim})^{(3)}\ra \calp_{<\infty}(\FS)$$ as required. 
\end{proof}

\section{Cocycle rigidity for outer automorphism groups}\label{sec:conclusion}

In this section, we complete the proofs of our main theorems. All the rigidity theorems are stated with a product of higher rank simple algebraic groups $G$ as the source of the cocycle, but in view of the induction statement (Proposition~\ref{induction}), one can also replace $G$ by a lattice in $G$.

\subsection{Cocycle rigidity for mapping class groups}

A \emph{surface of finite type} is a (possibly disconnected, possibly non-orientable) surface $\Sigma$ obtained from a boundaryless compact surface by possibly removing finitely many points and the interior of finitely many disks.
Its \emph{extended mapping class group} $\Mod^*(\Sigma)$ is the group of all isotopy classes of homeomorphisms of $\Sigma$ which restrict to the identity on every boundary component (where isotopies are required to fix the boundary at all times). The main result of the present section is the following theorem, compare with \cite[Section~9]{Ham}.

\begin{theo}\label{mcg}
Let $\Sigma$ be a surface of finite type. Let $G$ be a product of connected higher rank simple algebraic groups over local fields. Let $X$ be a standard probability space equipped with a measure-preserving ergodic $G$-action.

Then every cocycle $G\times X\to\Mod^*(\Sigma)$ is cohomologous to a cocycle that takes its values in a finite subgroup of $\Mod^*(\Sigma)$.
\end{theo}

In order to prove Theorem~\ref{mcg}, we will apply the strategy developed in Section~\ref{sec:criterion}. The \emph{topological complexity} $\xi(\Sigma)$ of a connected, boundaryless, orientable surface $\Sigma$ of finite type is defined as $\xi(\Sigma)=3g(\Sigma)+s(\Sigma)-3$, where $g(\Sigma)$ denotes the genus of $\Sigma$ and $s(\Sigma)$ is the number of punctures. We first establish the following fact. 

\begin{prop}\label{prop:ast-mcg}
Let $\Sigma$ be a connected boundaryless orientable surface of finite type with $\xi(\Sigma)>0$. Then $\Mod^*(\Sigma)$ is geometrically rigid with respect to the countable set of all isotopy classes of essential simple closed curves on $\Sigma$.
\end{prop}

\begin{proof}
Let $K:=\PML$ be the space of projective measured laminations on $\Sigma$, which is compact and metrizable (in fact homeomorphic to a finite-dimensional sphere). Let $K_\infty\subseteq \PML$ be the subspace made of minimal filling laminations: this is a Borel subset of $\PML$, because it is also equal to the subset of $\PML$ made of all laminations that have positive intersection number with every essential simple closed curve on $\Sigma$. Let $K_{\bdd}:=K\setminus K_\infty$. Let $\Delta$ be the Gromov boundary of the curve graph of $\Sigma$ (whose hyperbolicity was proved by Masur and Minsky in \cite{MM}). Being the Gromov boundary of a hyperbolic graph, the space $\Delta$ is completely metrizable (see e.g.\ \cite[Proposition~5.31]{Vai}); using the fact that the vertex set of the curve graph is countable, it is easy to see that $\Delta$ is also separable, whence a Polish space. We will denote by $\bbD$ the countable set of all isotopy classes of essential simple closed curves on $\Sigma$.

By a theorem of Klarreich \cite{Kla}, there exists a continuous (whence Borel) $\Mod^*(\Sigma)$-equivariant map $K_\infty\to\Delta$. By \cite[Section~1.1]{KM} or \cite[Lemma~4.38]{Kid}, there exists a Borel $\Mod^*(\Sigma)$-equivariant map $K\setminus K_\infty\to\calp_{<\infty}(\bbD)$. The existence of a Borel $\Mod^*(\Sigma)$-equivariant `barycenter map' $\Delta^{(3)}\to\calp_{<\infty}(\bbD)$ was proved by Kida in \cite[Section~4.1.2]{Kid}. Finally, universal amenability of the $\Mod^*(\Sigma)$-action on $\Delta$ was established by Hamenstädt in \cite[Corollary~2]{Ham} and independently by Kida in \cite[Theorem~3.19]{Kid}.
\end{proof}

\begin{proof}[Proof of Theorem~\ref{mcg}]
As $\Sigma$ is a surface of finite type, it has finitely many connected components. Therefore $\Mod^*(\Sigma)$ has a finite-index subgroup isomorphic to the direct product of finitely many extended mapping class groups of connected surfaces. Using the fact that cocycle-rigidity, as a property of the target group, is stable under extensions (Proposition~\ref{extension}), we can therefore assume that $\Sigma$ is connected. 

By attaching a once-puctured disk on every boundary component of $\Sigma$, we get a boundaryless surface $\widehat{\Sigma}$ of finite type, and the inclusion $\Sigma\hookrightarrow\widehat{\Sigma}$ induces a homomorphism $\Mod^*(\Sigma)\to\Mod^*(\widehat{\Sigma})$ onto a finite-index subgroup of $\Mod^*(\widehat{\Sigma})$, whose kernel is abelian (see \cite[Proposition~3.19]{FM-primer}). Since $(G,A)$ is cocycle-rigid for every abelian group $A$ by Zimmer's theorem \cite[Theorem~9.1.1]{Zim}, and cocycle-rigidity, as a property of the target group, is stable under group extensions and finite-index overgroups, we can therefore assume that $\Sigma$ is boundaryless. 

We first assume that $\Sigma$ is orientable. The proof is by induction on the topological complexity $\xi(\Sigma)$. If $\xi(\Sigma)\le 0$, then $\Mod^*(\Sigma)$ is finite if $\Sigma$ is a sphere with at most $3$ punctures, and isomorphic to $\mathrm{GL}(2,\mathbb{Z})$ whence virtually free if $\Sigma$ is a closed torus. In both cases $(G,\Mod^*(\Sigma))$ is cocycle-rigid. 

We now assume that $\xi(\Sigma)>0$. Let $c:G\times X\to\Mod^*(\Sigma)$ be a cocycle. By Proposition~\ref{prop:ast-mcg}, the group $\Mod^*(\Sigma)$ is geometrically rigid 
with respect to  the countable set $\bbD$ of all isotopy classes of essential simple closed curves on $\Sigma$. 
Theorem~\ref{theo:abstract} thus shows that $c$ is cohomologous to a cocycle with values in a subgroup $H$ of $\Mod^*(\Sigma)$ that virtually fixes the isotopy class of simple closed curve $c$ on $\Sigma$. Let $\Sigma^c$ be the (possibly disconnected) surface obtained by cutting $\Sigma$ along $c$. 

If $\Sigma^c$ is connected, then $H$ has a finite index subgroup $H^0$ (fixing $c$) for which there is a morphism $H^0\to\Mod^*(\Sigma^c)$, whose kernel is at most cyclic (generated by a twist about $c$). Since $g(\Sigma^c)=g(\Sigma)-1$ and $s(\Sigma^c)=s(\Sigma)+2$, we get $\xi(\Sigma^c)<\xi(\Sigma)$, so we know by induction that $(G,\Mod^*(\Sigma^c))$ is cocycle-rigid. Since $(G,\mathbb{Z})$ is cocycle-rigid \cite[Theorem~9.1.1]{Zim}, and cocycle-rigidity is stable under finite-index overgroups and group extensions (Propositions~\ref{finite-index} and~\ref{extension}), it follows that $(G,\Mod^*(\Sigma))$ is cocycle-rigid, which concludes our proof in this case.

If $\Sigma^c$ is disconnected, denoting by $\Sigma_1$ and $\Sigma_2$ its connected components, then $H$ has a finite index subgroup $H^0$ (fixing $c$ without reversing its orientation) for which there is a morphism $H^0\to\Mod^*(\Sigma_1)\times\Mod^*(\Sigma_2)$, whose kernel is at most cyclic (generated by a twist about $c$). The proof then goes by induction as in the previous case using Proposition \ref{extension} and the fact that $\xi(\Sigma_1),\xi(\Sigma_2)<\xi(\Sigma)$. 

We now assume that $\Sigma$ is closed and non-orientable. If $\Sigma$ is a projective plane or a Klein bottle then its mapping class group is finite. Otherwise, by \cite[Lemma~4]{Fuj}, denoting by $\hat{\Sigma}$ an orientable double cover of $\Sigma$, the group $\Mod^*(\Sigma)$ has a finite-index subgroup that embeds in $\Mod^*(\hat{\Sigma})$. The conclusion then follows from the orientable case. 

If $\Sigma$ is a non-orientable surface with punctures, 
one can use the Birman exact sequence given in \cite[Lemma~6]{Fuj} and apply Proposition \ref{extension} to reduce to the case of a closed surface.
Alternatively,  the mapping class group $\Mod^*(\Sigma)$ is also a subgroup of $\Out(F_N)$, so the conclusion also follows from Corollary~\ref{cor:free} in this case.
\end{proof}

\subsection{Cocycle rigidity for automorphisms of free products}

The main result of the present paper is the following theorem. Given a group $G$, we denote by $Z(G)$ the center of $G$.

\begin{theo}\label{free-product}
Let $H_1,\dots,H_p$ be countable groups, let $F_N$ be a free group of rank $N$, and let $$H:=H_1\ast\dots\ast H_p\ast F_N.$$
Let $\calf=\{[H_1],\dots,[H_p]\}$. Let $G$ be a product of connected higher rank simple algebraic groups over local fields.
 Assume that 
\begin{itemize}
\item for each $i\in\{1,\dots,k\}$, the pair $(G,H_i/Z(H_i))$ is cocycle-rigid, and
\item for each $i\in\{1,\dots,k\}$, the centralizer in $H_i$ of every nontrivial element is amenable.
\end{itemize}
Then $(G,\Out(H,\calf^{(t)}))$ is cocycle-rigid.
\end{theo}

\begin{rk}
If we assume in addition that $(G,\Out(H_i))$ is cocycle-rigid for all $i\in\{1,\dots,k\}$, then we can conclude that $(G,\Out(H,\calf))$ is cocycle-rigid. This follows from the short exact sequence $$1\to\Out(H,\calf^{(t)})\to\Out(H,\calf)\to\prod_{i=1}^k\Out(H_i)\to 1$$ and the fact that cocycle-rigidity is stable under extensions as a property of the target group (Proposition~\ref{extension}).
\end{rk}

\begin{rk}
As shown in \cite[Theorem~6.4 and Remark~6.5]{BGH}, the second assumption in Theorem~\ref{free-product} is exactly the required assumption to ensure that the action of $\Out(H,\calf^{(t)})$ on the space of all arational $(H,\calf)$-trees is universally amenable.
\end{rk}

An immediate consequence of Theorem~\ref{free-product} (applied with $p=0$) is the following.

\begin{cor}\label{cor:free}
Let $N\ge 2$. Let $G$ be a product of connected higher rank simple algebraic groups over local fields. Let $X$ be a standard probability space equipped with an ergodic measure-preserving $G$-action.

Then every cocycle $G\times X\to\Out(F_N)$ is cohomologous to a cocycle that takes its values in some finite subgroup of $\Out(F_N)$.
\qed 
\end{cor}

Our proof of Theorem~\ref{free-product} will follow the strategy established in Section~\ref{sec:criterion}. We start with the following fact.

\begin{prop}\label{prop:ast-out}
Let $H_1,\dots,H_k$ be countable groups, let $F_N$ be a free group of rank $N$, and let $$H:=H_1\ast\dots\ast H_k\ast F_N.$$ 
Let $\calf=\{[H_1],\dots,[H_p]\}$. Assume that for each $i\in\{1,\dots,k\}$, the centralizer in $H_i$ of every nontrivial element is amenable.

Then $\Out(H,\calf^{(t)})$ is geometrically rigid with respect to the countable set $\bbD$ of all conjugacy classes of proper free factors of $(H,\calf)$.
\end{prop}

\begin{proof}
Let $K:=\mathbb{P}\overline{\calo}(H,\calf)$ be the compactification of the corresponding Outer space, which is metrizable. Let $\Delta$ be the Gromov boundary of the free factor graph of $(H,\calf)$, which is a Polish space. 

Let $K_\infty\subseteq K$ be the subset made of all arational $(H,\calf)$-trees: this is a Borel subset of $K$ by \cite[Lemma~5.5]{Hor}. Let $K_{\bdd}:=K\setminus K_\infty$. By \cite[Theorem~3]{GH}, there exists a continuous (whence Borel) $\Out(H,\calf)$-equivariant map $K_\infty\to\Delta$. By \cite{Rey,Hor}, there exists an $\Out(H,\calf)$-equivariant Borel map $K_{\bdd}\to\calp_{<\infty}(\bbD)$. The existence of an $\Out(H,\calf)$-equivariant Borel map $\Delta^{(3)}\to\calp_{<\infty}(\bbD)$ was established in Theorem~\ref{barycenter}. Finally, universal amenability of the $\Out(H,\calf^{(t)})$-action on $\Delta$ was established in Proposition~\ref{prop:action-amenable}.  
\end{proof}

\begin{proof}[Proof of Theorem~\ref{free-product}]
The proof is by induction on the complexity $\xi(H,\calf)=(p+N,N)$ (complexities are ordered lexicographically). If $\xi(H,\calf)\le (2,1)$, then we are in one of the following four cases.

\indent 1. If $\xi(H,\calf)\le (1,0)$, then $H=\{1\}$ or $H=H_1$. In this case $\Out(H,\calf^{(t)})$ is trivial, so the conclusion is obvious.

\indent 2. If $\xi(H,\calf)=(1,1)$, then $H=\mathbb{Z}$, so $\Out(H,\calf^{(t)})=\mathbb{Z}/2\mathbb{Z}$, and again the conclusion is obvious.

\indent 3. If $\xi(H,\calf)=(2,0)$, then $H=H_1\ast H_2$, and $\Out(H,\calf^{(t)})$ is isomorphic to $H_1/Z(H_1)\times H_2/Z(H_2)$ (see \cite{Lev}). As $(G,H_i/Z(H_i))$ is cocycle-rigid for every $i\in\{1,2\}$, we deduce that $(G,\Out(H,\calf^{(t)}))$ is cocycle-rigid.

\indent 4. If $\xi(H,\calf)=(2,1)$, then $H=H_1\ast\mathbb{Z}$ and $\calf=\{[H_1]\}$, and $\Out(H,\calf^{(t)})$ has an index $2$ subgroup which is isomorphic to $(H_1\times H_1)/Z(H_1)$, where $Z(H_1)$ is diagonally embedded in $H_1\times H_1$, see \cite{Lev}. The group $(H_1\times H_1)/Z(H_1)$ maps onto $H_1/Z(H_1)\times H_1/Z(H_1)$ with abelian kernel. Since for every abelian group $A$, the pair $(G,A)$ is cocycle-rigid \cite[Theorem~9.1.1]{Zim}, we deduce (using stability under extensions and finite index overgroups) that $(G,\Out(H,\calf^{(t)}))$ is cocycle-rigid.

We now assume that $\xi(H,\calf)>(2,1)$, i.e.\ $(H,\calf)$ is nonsporadic. Let $X$ be a standard probability space equipped with a measure-preserving $G$-action, and let $$c:G\times X\to \Out(H,\calf^{(t)})$$ be a cocycle. By Proposition~\ref{prop:ast-out}, the group $\Out(H,\calf^{(t)})$ is geometrically rigid with respect to the collection of all conjugacy classes of proper $(H,\calf)$-free factors. Therefore, Theorem~\ref{theo:abstract} ensures that $c$ is cohomologous to a cocycle $c'$ that takes its values in a subgroup of $\Out(H,\calf^{(t)})$ that virtually fixes the conjugacy class of a proper $(H,\calf)$-free factor $A$. Let $\Out(H,\calf^{(t)},A)$ be the subgroup of $\Out(H,\calf^{(t)})$ made of all outer automorphisms that preserve the conjugacy class of $A$. It is enough to prove that $(G,\Out(H,\calf^{(t)},A))$ is cocycle-rigid. 

Recall that $\calf_{|A}$ denotes the free factor system of $A$ induced by $\calf$; we also denote by $\widetilde{\calf}$ the smallest $(H,\calf)$-free factor system that contains both the conjugacy classes in $\calf$ and the conjugacy class of $A$.
One has $\xi(A,\calf_{|A})<\xi(H,\calf)$ and $\xi(H,\widetilde{\calf})<\xi(H,\calf)$ and there is a short exact sequence $$1\to\Out(H,\widetilde{\calf}^{(t)})\to\Out(H,\calf^{(t)},A)\to\Out(A,\calf_{|A}^{(t)})\to 1$$
(see the proof of \cite[Theorem~6.3]{Hor}). By induction, the pairs $(G,\Out(A,\calf_{|A}^{(t)}))$ and $(G,\Out(H,\widetilde{\calf}^{(t)}))$ are cocycle-rigid, so using stability under extensions, we deduce that the pair $(G,\Out(H,\calf^{(t)},A))$ is cocycle-rigid, which concludes our proof.    
\end{proof}

\subsection{Cocycle rigidity for automorphisms of relatively hyperbolic groups}

We now give an extension of Theorem~\ref{free-product} to the context of torsion-free relatively hyperbolic groups.

\begin{theo}\label{theo:rel-hyp}
Let $H$ be a torsion-free group which is hyperbolic relative to a finite collection $\calp$ of finitely generated subgroups. Let $G$ be a product of connected higher rank simple algebraic groups over local fields. Assume that 
\begin{itemize}
\item for every $P\in\calp$, the pair $(G,P/Z(P))$ is cocycle-rigid, and
\item for every $P\in\calp$, the centralizer in $P$ of every nontrivial element is amenable.
\end{itemize}
Then $(G,\Out(H,\calp^{(t)}))$ is cocycle-rigid.
\end{theo}

In particular, by applying Theorem~\ref{theo:rel-hyp} to the case where $\calp=\emptyset$ (in which case $H$ is hyperbolic), we get the following extension of Corollary~\ref{cor:free}.

\begin{cor}
Let $G$ be a product of connected higher rank simple algebraic groups over local fields. Let $X$ be a standard probability space equipped with an ergodic measure-preserving $G$-action. Let $H$ be a torsion-free hyperbolic group.

Then every cocycle $G\times X\to\Out(H)$ is cohomologous to a cocycle that takes its values in some finite subgroup of $\Out(H)$.
\qed
\end{cor}

Deducing Theorem~\ref{theo:rel-hyp} from Theorem~\ref{theo:main} is done using an argument which is similar to the proof of \cite[Appendix, Corollary~3]{Hae}, which we now explain.

\begin{proof}[Proof of Theorem~\ref{theo:rel-hyp}]
We first assume that $H$ is freely indecomposable relative to $\calp$ and argue as in Case 1 of the proof of \cite[Appendix, Corollary~3]{Hae}. Let $S$ be the canonical elementary JSJ decomposition of $H$ relative to $\calp$ (see \cite[Theorem~4]{GL2}). Then $\Out(H,\calp^{(t)})$ has a finite-index subgroup which surjects onto a finite direct product of mapping class groups of surfaces of finite type, with kernel the group of twists $\calt$ of $S$. By \cite[Appendix, Lemma~5]{Hae}, the group $\calt$ maps with abelian kernel to a direct product of finitely many groups, each of which is isomorphic to $P/Z(P)$ for some $P\in\calp$. Using Theorem~\ref{mcg} for mapping class groups and the fact that cocycle-rigidity, as a property of the target group, is stable under finite index overgroups and extensions, we deduce that the pair $(G,\Out(H,\calp^{(t)}))$ is cocycle-rigid.

We now consider the general case, and let $$H=H_1\ast\dots\ast H_k\ast F_N$$ be a Grushko decomposition of $H$ relative to $\calp$. Let $\calf=\{[H_1],\dots,[H_k]\}$. We will now argue as in the general case of the proof of \cite[Appendix, Corollary~3]{Hae}. Each subgroup $H_i$ is torsion-free, hyperbolic relative to $\calp_{|H_i}$, and freely indecomposable relative to $\calp_{|H_i}$. Let $\Out^0(H,\calp^{(t)})$ be the finite-index subgroup of $\Out(H,\calp^{(t)})$ consisting of all automorphisms that preserve the conjugacy class of each subgroup $H_i$ (as opposed to permuting them). There is a morphism $\Out^0(H,\calp^{(t)})\to\prod_{i=1}^k\Out(H_i,\calp_{|H_i}^{(t)})$, whose kernel is equal to $\text{Out}(H,\calf^{(t)})$. By the previous case, for every $i\in\{1,\dots,k\}$, the pair $(G,\Out(H_i,\calp_{|H_i}^{(t)}))$ is cocycle-rigid. We will apply Theorem~\ref{free-product} to show that $(G,\Out(H,\calf^{(t)}))$ is cocycle-rigid, which will conclude our proof. 

By assumption on centralizers in parabolic subgroups, 
the centralizer in the relatively hyperbolic group $H$ of every nontrivial element is amenable. 
Therefore, in order to apply Theorem~\ref{free-product}, we only need to check that for every $i\leq k$, the pair $(G,H_i/Z(H_i))$ is cocycle-rigid. This holds by assumption if $H_i$ is equal to one of the subgroups in $\calp$. Otherwise $Z(H_i)$ is trivial,
and $H_i$ is hyperbolic relative to $\calp_{|H_i}$ so we apply cocycle rigidity for relatively hyperbolic groups (Theorem~\ref{theo:rhyp}), and for that it suffices to check that the pair $(G,P)$ is cocycle rigid for every parabolic group $P\in \calp$.
Using the stability of cocycle rigidity under extensions (Proposition \ref{extension}), this is indeed the case because  
$(G,P/Z(P))$ is cocycle-rigid by assumption and $(G,Z(P))$ is cocycle-rigid because $Z(P)$ is abelian. 
\end{proof}

\footnotesize

\bibliographystyle{alpha}
\bibliography{rigide-bib}

\begin{thebibliography}{CHL08b}

\bibitem[Ada94a]{Ada3}
S.~Adams.
\newblock Boundary amenability for hyperbolic groups and an application to
  smooth dynamics of simple groups.
\newblock {\em Topology}, 33(4):765--783, 1994.

\bibitem[Ada94b]{Ada}
S.~Adams.
\newblock Indecomposability of equivalence relations generated by word
  hyperbolic groups.
\newblock {\em Topol.}, 33(4):785--798, 1994.

\bibitem[Ada96]{Ada2}
S.~Adams.
\newblock Reduction of cocycles with hyperbolic targets.
\newblock {\em Erg. Th. Dyn. Syst.}, 16(6):1111--1145, 1996.

\bibitem[ADR00]{ADR}
C.~Anantharaman-Delaroche and J.~Renault.
\newblock {\em Amenable groupoids}, volume~36 of {\em Monographies de
  l'Enseignement Mathématique}.
\newblock L'Enseignement Mathématique, Geneva, 2000.

\bibitem[AEG94]{AEG}
S.~Adams, G.A. Elliott, and T.~Giordano.
\newblock Amenable actions of groups.
\newblock {\em Trans. Amer. Math. Soc.}, 344(2):803--822, 1994.

\bibitem[Bad]{MO}
U.~Bader.
\newblock Cocycle superrigidity.
\newblock MathOverflow.
\newblock {https://mathoverflow.net/a/301246/81562}.

\bibitem[BC12]{BC}
J.~Behrstock and R.~Charney.
\newblock Divergence and quasimorphisms of right-angled {A}rtin groups.
\newblock {\em Math. Ann.}, 352(2):339--356, 2012.

\bibitem[BF]{BFHyp}
U.~Bader and A.~Furman.
\newblock Superrigidity via {W}eyl groups: hyperbolic-like targets.
\newblock {\em in preparation}.

\bibitem[BF94]{BeF2}
M.~Bestvina and M.~Feighn.
\newblock Outer limits.
\newblock preprint, available at
  http://andromeda.rutgers.edu/~feighn/papers/outer.pdf, 1994.

\bibitem[BF14a]{BF}
U.~Bader and A.~Furman.
\newblock Boundaries, rigidity of representations, and {L}yapunov exponents.
\newblock {\em Proceedings of ICM 2014, Invited Lectures}, pages 71--96, 2014.

\bibitem[BF14b]{BeF}
M.~Bestvina and M.~Feighn.
\newblock Hyperbolicity of the complex of free factors.
\newblock {\em Adv. Math.}, 256:104--155, 2014.

\bibitem[BFH97]{BFH}
M.~Bestvina, M.~Feighn, and M.~Handel.
\newblock Laminations, trees, and irreducible automorphisms of free groups.
\newblock {\em Geom. Funct. Anal.}, 7(2):215--244, 1997.

\bibitem[BFS06]{BFS}
U.~Bader, A.~Furman, and A.~Shaker.
\newblock Superrigidity, {W}eyl groups, and actions on the circle.
\newblock {\em arXiv:math/0605276}, 2006.

\bibitem[BG17]{BaderGelander}
U.~{Bader} and T.~{Gelander}.
\newblock Equicontinuous actions of semisimple groups.
\newblock {\em Groups Geom. Dyn.}, 11(3):1003--1039, 2017.

\bibitem[BGH17]{BGH}
M.~Bestvina, V.~Guirardel, and C.~Horbez.
\newblock Boundary amenability of $\text{{O}ut}({F}_{N})$.
\newblock {\em arXiv:1705.07017}, 2017.

\bibitem[BM99]{BM}
M.~Burger and N.~Monod.
\newblock Bounded cohomology of lattices in higher rank {L}ie groups.
\newblock {\em J. Eur. Math. Soc.}, 1(2):199--235, 1999.

\bibitem[Bor91]{Borel}
A.~Borel.
\newblock {\em Linear algebraic groups}, volume 126 of {\em Graduate Texts in
  Mathematics}.
\newblock Springer-Verlag, New York, second edition, 1991.

\bibitem[Bow12]{Bow}
B.H. Bowditch.
\newblock Relatively hyperbolic groups.
\newblock {\em Internat. J. Algebra Comput.}, 22(3), 2012.

\bibitem[BR15]{BR}
M.~Bestvina and P.~Reynolds.
\newblock The boundary of the complex of free factors.
\newblock {\em Duke Math. J.}, 164(11):2213--2251, 2015.

\bibitem[BT65]{BorelTits}
A.~Borel and J.~Tits.
\newblock Groupes r\'eductifs.
\newblock {\em Inst. Hautes \'Etudes Sci. Publ. Math.}, (27):55--150, 1965.

\bibitem[BW11]{BW}
M.R. Bridson and R.D. Wade.
\newblock Actions of higher-rank lattices on free groups.
\newblock {\em Compos. Math.}, 147(5):1573--1580, 2011.

\bibitem[CFI16]{CFI}
I.~Chatterji, T.~Fern\'os, and A.~Iozzi.
\newblock The median class and superrigidity of actions on $\mathrm{CAT}(0)$
  cube complexes.
\newblock {\em J. Topol.}, 9(2):349--400, 2016.

\bibitem[CHL07]{CHL3}
T.~Coulbois, A.~Hilion, and M.~Lustig.
\newblock Non-unique ergodicity, observers' topology and the dual algebraic
  lamination for $\mathbb{R}$-trees.
\newblock {\em Illinois J. Math.}, 51(3):897--911, 2007.

\bibitem[CHL08a]{CHL}
T.~Coulbois, A.~Hilion, and M.~Lustig.
\newblock $\mathbb{R}$-trees and laminations for free groups {I} :algebraic
  laminations.
\newblock {\em J. Lond. Math. Soc.}, 78(3):723--736, 2008.

\bibitem[CHL08b]{CHL2}
T.~Coulbois, A.~Hilion, and M.~Lustig.
\newblock $\mathbb{R}$-trees and laminations for free groups {II} :the dual
  lamination of an $\mathbb{R}$-tree.
\newblock {\em J. Lond. Math. Soc.}, 78(3):737--754, 2008.

\bibitem[CHR15]{CHR}
T.~Coulbois, A.~Hilion, and P.~Reynolds.
\newblock Indecomposable ${F}_{N}$-trees and minimal laminations.
\newblock {\em Groups Geom. Dyn.}, 9(2):567--597, 2015.

\bibitem[CL95]{CL}
M.M. Cohen and M.~Lustig.
\newblock Very small group actions on $\mathbb{{R}}$-trees and {D}ehn twist
  automorphisms.
\newblock {\em Topology}, 34(3):575--617, 1995.

\bibitem[CM87]{CM}
M.~Culler and J.W. Morgan.
\newblock Group actions on $\mathbb{R}$-trees.
\newblock {\em Proc. London Math. Soc.}, 55(3):571--604, 1987.

\bibitem[CV86]{CV}
M.~Culler and K.~Vogtmann.
\newblock Moduli of graphs and automorphisms of free groups.
\newblock {\em Invent. Math.}, 84(1):91--119, 1986.

\bibitem[Duc18]{Duc}
B.~Duchesne.
\newblock Groups acting on spaces of non-positive curvature.
\newblock In {\em Handbook of group actions. Vol. III.}, volume~40 of {\em Adv.
  Lect. Math.}, pages 101--141, Int. Press, Sommerville, MA, 2018.

\bibitem[Edw65]{Edw}
R.E. Edwards.
\newblock {\em Functional analysis. Theory and applications.}
\newblock Holt, Rinehart and Winston, 1965.

\bibitem[Fer18]{Fer}
T.~Fern\'os.
\newblock The {F}urstenberg-{P}oisson boundary and $\mathrm{CAT}(0)$ cube
  complexes.
\newblock {\em Erg. Theo. Dyn. Syst.}, 38(6):2180--2223, 2018.

\bibitem[FLM18]{FLM}
T.~Fern\'os, J.~Lécureux, and F.~Mathéus.
\newblock Random {W}alks and {B}oundaries of {CAT}($0$) {C}ubical {C}omplexes.
\newblock {\em Comment. Math. Helv.}, 93(2):291--333, 2018.

\bibitem[FM98]{FM}
B.~Farb and H.~Masur.
\newblock Superrigidity and mapping class groups.
\newblock {\em Topology}, 37(6):1169--1176, 1998.

\bibitem[FM12]{FM-primer}
B.~Farb and D.~Margalit.
\newblock {\em A primer on mapping class groups}, volume~49 of {\em Princeton
  Mathematical Series}.
\newblock Princeton University Press, Princeton, NJ, 2012.

\bibitem[FM15]{FrMa}
S.~Francaviglia and A.~Martino.
\newblock Stretching factors, metrics and train tracks for free products.
\newblock {\em Illinois J. Math.}, 59(4):859--899, 2015.

\bibitem[Fuj02]{Fuj}
K.~Fujiwara.
\newblock On the outer automorphism group of a hyperbolic group.
\newblock {\em Israel J. Math.}, 131:277--284, 2002.

\bibitem[Fur99]{Fur}
A.~Furman.
\newblock Gromov's measure equivalence and rigidity of higher rank lattices.
\newblock {\em Ann. Math. (2)}, 150(3):1059--1081, 1999.

\bibitem[GH19a]{GH1}
V.~Guirardel and C.~Horbez.
\newblock Algebraic laminations for free products and arational trees.
\newblock {\em Algebr. Geom. Topol.}, 19(5):2283--2400, 2019.

\bibitem[GH19b]{GH}
V.~Guirardel and C.~Horbez.
\newblock Boundaries of relative factor graphs and subgroup classification for
  automorphisms of free products.
\newblock {\em arXiv:1901.05046}, 2019.

\bibitem[GL07]{GL}
V.~Guirardel and G.~Levitt.
\newblock The outer space of a free product.
\newblock {\em Proc. London Math. Soc.}, 94(3):695--714, 2007.

\bibitem[GL11]{GL2}
V.~Guirardel and G.~Levitt.
\newblock Trees of cylinders and canonical splittings.
\newblock {\em Geom. Topol.}, 15(2):977--1012, 2011.

\bibitem[GL17]{GL-jsj}
V.~Guirardel and G.~Levitt.
\newblock {JSJ} decompositions of groups.
\newblock {\em Astérisque}, 395:vii+165 pp., 2017.

\bibitem[Gui00]{Gui00}
V.~Guirardel.
\newblock Dynamics of $\text{{O}ut}({F}_n)$ on the boundary of outer space.
\newblock {\em Ann. Scient. Ec. Norm. Sup.}, 33(4):433--465, 2000.

\bibitem[Hae16]{Hae}
T.~Haettel.
\newblock Hyperbolic rigidity of higher-rank lattices.
\newblock {\em Ann. Sci. \'Ec. Norm. Supér. (4)}, 53(2):439--468, 2016.

\bibitem[Ham09]{Ham}
U.~Hamenstädt.
\newblock Geometry of the mapping class groups {I}: {B}oundary amenability.
\newblock {\em Invent. Math.}, 175(3):545--609, 2009.

\bibitem[Ham12]{Ham2}
U.~Hamenstädt.
\newblock The boundary of the free factor graph and the free splitting graph.
\newblock {\em arXiv:1211.1630}, 2012.

\bibitem[HH20]{HH}
C.~Horbez and J.~Huang.
\newblock Measure equivalence classification of transvection-free right-angled
  {A}rtin groups.
\newblock {\em arXiv:2010.03613}, 2020.

\bibitem[HM14]{HM}
M.~Handel and L.~Mosher.
\newblock Relative free factor and free splitting complexes {I}:
  {H}yperbolicity.
\newblock {\em arXiv:1407.3508}, 2014.

\bibitem[Hor14]{Hor}
C.~Horbez.
\newblock The {T}its alternative for the automorphism group of a free product.
\newblock {\em arXiv:1408.0546}, 2014.

\bibitem[Hor16]{Hor3}
C.~Horbez.
\newblock Hyperbolic graphs for free products, and the {G}romov boundary of the
  graph of cyclic splittings.
\newblock {\em J. Topol.}, 9(2):401--450, 2016.

\bibitem[Hor17]{Hor2}
C.~Horbez.
\newblock The boundary of the outer space of a free product.
\newblock {\em Israel J. Math.}, 221(1):179--234, 2017.

\bibitem[JKL02]{JKL}
S.~Jackson, A.S. Kechris, and A.~Louveau.
\newblock Countable {B}orel equivalence relations.
\newblock {\em J. Math. Logic}, 2(1):1--80, 2002.

\bibitem[Kai03]{Kaimanovich}
V.~A. Kaimanovich.
\newblock Double ergodicity of the {P}oisson boundary and applications to
  bounded cohomology.
\newblock {\em Geom. Funct. Anal.}, 13(4):852--861, 2003.

\bibitem[Kec95]{Kec}
A.S. Kechris.
\newblock {\em Classical {D}escriptive {S}et {T}heory}, volume 156 of {\em
  Graduate Texts in Mathematics}.
\newblock Springer-Verlag, New York, 1995.

\bibitem[Kid08]{Kid}
Y.~Kida.
\newblock The mapping class group from the viewpoint of measure equivalence
  theory.
\newblock {\em Mem. Amer. Math. Soc.}, 196(916), 2008.

\bibitem[Kla99]{Kla}
E.~Klarreich.
\newblock The {B}oundary at {I}nfinity of the {C}urve {C}omplex and the
  {R}elative {T}eichm{\"u}ller {S}pace.
\newblock {\em arXiv:1803.10339}, 1999.

\bibitem[KM96]{KM}
V.A. Kaimanovich and H.~Masur.
\newblock The {P}oisson boundary of the mapping class group.
\newblock {\em Invent. Math.}, 125(2):221--264, 1996.

\bibitem[Lev05]{Lev}
G.~Levitt.
\newblock Automorphisms of hyperbolic groups and graphs of groups.
\newblock {\em Geom. Dedic.}, 114:49--70, 2005.

\bibitem[Mar91]{Margulis}
G.~A. Margulis.
\newblock {\em Discrete subgroups of semisimple {L}ie groups}, volume~17 of
  {\em Ergebnisse der Mathematik und ihrer Grenzgebiete (3) [Results in
  Mathematics and Related Areas (3)]}.
\newblock Springer-Verlag, Berlin, 1991.

\bibitem[Mim18]{Mim}
M.~Mimura.
\newblock Superrigidity from {C}hevalley groups into acylindrically hyperbolic
  groups via quasi-cocycles.
\newblock {\em J. Eur. Math. Soc.}, 20(1):103--117, 2018.

\bibitem[MM99]{MM}
H.A. Masur and Y.N. Minsky.
\newblock Geometry of the complex of curves {I} : {H}yperbolicity.
\newblock {\em Invent. math.}, 138(1):103--149, 1999.

\bibitem[MS03]{MS2}
N.~Monod and Y.~Shalom.
\newblock Negative curvature from a cohomological viewpoint and cocycle
  superrigidity.
\newblock {\em C.R. Acad. Sci. Paris, Ser. I}, 337:635--638, 2003.

\bibitem[MS04]{MS}
N.~Monod and Y.~Shalom.
\newblock Cocycle superrigidity and bounded cohomology for negatively curved
  spaces.
\newblock {\em J. Diff. Geom.}, 67:395--455, 2004.

\bibitem[MS06]{MS3}
N.~Monod and Y.~Shalom.
\newblock Orbit equivalence rigidity and bounded cohomology.
\newblock {\em Ann. Math. (2)}, 164(3):825--878, 2006.

\bibitem[NS13]{NS}
A.~Nevo and M.~Sageev.
\newblock The {P}oisson boundary of $\mathrm{CAT}(0)$ cube complex groups.
\newblock {\em Groups Geom. Dyn.}, 7(3):653--695, 2013.

\bibitem[OW80]{OW}
D.S. Ornstein and B.~Weiss.
\newblock Ergodic theory of amenable group actions. {I}. {T}he {R}ohlin lemma.
\newblock {\em Bull. Amer. Math. Soc.}, 2(1):161--164, 1980.

\bibitem[Oza06]{Oza}
N.~Ozawa.
\newblock Boundary amenability of relatively hyperbolic groups.
\newblock {\em Topol. Appl.}, 153:2624--2630, 2006.

\bibitem[Pau88]{Pau}
F.~Paulin.
\newblock Topologie de {G}romov équivariante, structures hyperboliques et
  arbres réels.
\newblock {\em Invent. Math.}, 94(1):53--80, 1988.

\bibitem[Pau89]{Pau2}
F.~Paulin.
\newblock The {G}romov topology on $\mathbb{R}$-trees.
\newblock {\em Topology Appl.}, 32(3):197--221, 1989.

\bibitem[Pra82]{Prasad}
G.~Prasad.
\newblock Elementary proof of a theorem of {B}ruhat-{T}its-{R}ousseau and of a
  theorem of {T}its.
\newblock {\em Bull. Soc. Math. France}, 110(2):197--202, 1982.

\bibitem[Rey12]{Rey}
P.~Reynolds.
\newblock Reducing systems for very small trees.
\newblock {\em arXiv:1211.3378}, 2012.

\bibitem[Rou77]{Rousseau}
G.~Rousseau.
\newblock Immeubles des groupes réductifs sur les corps locaux.
\newblock {\em Thèse d'état, Université Paris-Sud Orsay}, 1977.

\bibitem[Sri81]{Sri}
S.M. Srivastava.
\newblock Selection and representation theorems for $\sigma$-compact valued
  multifunctions.
\newblock {\em Proc. Amer. Math. Soc.}, 83(4):775--780, 1981.

\bibitem[SZ91]{SZ}
R.J. Spatzier and R.J. Zimmer.
\newblock Fundamental groups of negatively curved manifolds and actions of
  semisimple groups.
\newblock {\em Topology}, 30(4):591--601, 1991.

\bibitem[Vä05]{Vai}
J.~Väisälä.
\newblock Gromov hyperbolic spaces.
\newblock {\em Expo. Math.}, 23(3):187--231, 2005.

\bibitem[Var63]{Varadarajan}
V.~S. Varadarajan.
\newblock Groups of automorphisms of {B}orel spaces.
\newblock {\em Trans. Amer. Math. Soc.}, 109:191--220, 1963.

\bibitem[Zim78]{Zim2}
R.J. Zimmer.
\newblock Amenable ergodic group actions and an application to {P}oisson
  boundaries of random walks.
\newblock {\em J. Funct. Anal.}, 27(3):350--372, 1978.

\bibitem[Zim80]{Zim3}
R.J. Zimmer.
\newblock Strong rigidity for ergodic actions of semisimplie {L}ie groups.
\newblock {\em Ann. Math. (2)}, 112(3):511--529, 1980.

\bibitem[Zim84]{Zim}
R.J. Zimmer.
\newblock {\em Ergodic {T}heory and {S}emisimple {G}roups}, volume~81 of {\em
  Monographs in Mathematics}.
\newblock Birkhäuser, Basel; Boston; Stuttgart, 1984.

\end{thebibliography}

 \begin{flushleft}
 Vincent Guirardel\\
Univ Rennes, CNRS, IRMAR - UMR 6625, F-35000 Rennes, France\\
 \emph{e-mail:}\texttt{vincent.guirardel@univ-rennes1.fr}\\[8mm]
 \end{flushleft}

\begin{flushleft}
Camille Horbez\\
Universit\'e Paris-Saclay, CNRS,  Laboratoire de math\'ematiques d'Orsay, 91405, Orsay, France \\
\emph{e-mail:}\texttt{camille.horbez@universite-paris-saclay.fr}\\[8mm]
\end{flushleft}

\begin{flushleft}
Jean Lécureux\\
Universit\'e Paris-Saclay, CNRS,  Laboratoire de math\'ematiques d'Orsay, 91405, Orsay, France \\
\emph{e-mail:}\texttt{jean.lecureux@universite-paris-saclay.fr}\\[8mm]
\end{flushleft}

\end{document}